\documentclass[11pt]{article}
\usepackage {amssymb} %, adjustbox, enumerate, amsbsy}
\usepackage {amsmath}
\usepackage{amsthm}
\usepackage{graphicx}
\usepackage {amscd}
\usepackage{subfigure}

\usepackage{color}
\usepackage{multirow}
\usepackage{tabu}

\usepackage{diagbox}
\usepackage[colorlinks, citecolor=blue]{hyperref}

\newcommand*{\LargerCdot}{{\raisebox{0ex}{\scalebox{2}{$\cdot$}}}}
\newcommand{\bull}{{\bullet}}

\usepackage{geometry}
\newcommand{\vol}{{\rm vol}}
\newcommand{\ord}{{\rm ord}}
\newcommand{\fm}{\mathfrak{m}}
\newcommand{\fa}{\mathfrak{a}}
\newcommand{\cO}{\mathcal{O}}
\newcommand{\bR}{\mathbb{R}}
\newcommand{\bC}{\mathbb{C}}
\newcommand{\bZ}{\mathbb{Z}}
\newcommand{\lct}{{\rm lct}}
\newcommand{\wt}{{\rm wt}}

\newcommand{\Val}{{\rm Val}}
\newcommand{\fe}{{\mathfrak{e}}}
\newcommand{\hvol}{{\widehat{\rm vol}}}
\newcommand{\bx}{{\bf x}}

\newcommand{\hV}{{\widehat{V}}}
\newcommand{\cF}{{\mathcal{F}}}
\newcommand{\bV}{{\mathbb{V}}}
\newcommand{\cI}{{\mathcal{I}}}

\newcommand{\hE}{{\widehat{E}}}

\newcommand{\tcF}{{\widetilde{\mathcal{F}}}}
\newcommand{\bQ}{{\mathbb{Q}}}

\newcommand{\cX}{{\mathcal{X}}}
\newcommand{\cV}{{\mathcal{V}}}
\newcommand{\cL}{{\mathcal{L}}}
\newcommand{\ft}{{\mathfrak{t}}}
\newcommand{\cR}{{\mathcal{R}}}

\newcommand{\cW}{{\mathcal{W}}}
\newcommand{\bE}{{\mathbb{E}}}
\newcommand{\cG}{{\mathcal{G}}}

\newcommand{\bT}{{\mathbb{T}}}
\newcommand{\ocV}{{\bar{\cV}}}
\newcommand{\bP}{{\mathbb{P}}}
\newcommand{\CM}{{\rm CM}}
\newcommand{\bF}{{\mathbb{F}}}
\newcommand{\tcI}{{\widetilde{\mathcal{I}}}}

\newcommand{\ep}{{\epsilon}}
\newcommand{\bin}{{\bf in}}
\newcommand{\cC}{{\mathcal{C}}}
\newcommand{\cS}{{\mathcal{S}}}

\newcommand{\covol}{{\rm covol}}

\newcommand{\cA}{{\mathcal{A}}}

\newcommand{\bD}{{\mathbb{D}}}

\newtheorem{thm}{Theorem}[section]

\newtheorem{lem}[thm]{Lemma}
\newtheorem{cor}[thm]{Corollary}
\newtheorem{defn}[thm]{Definition}

\newtheorem{prop}[thm]{Proposition}

\newtheorem{rem}[thm]{Remark}
\newtheorem{examp}[thm]{Example}
\begin{document}

\title{K-semistability is equivariant volume minimization}
\author{Chi Li}
%\date{}

\maketitle{}

\abstract{This is a continuation to the paper \cite{Li15a} in which a problem of minimizing normalized volumes over $\bQ$-Gorenstein klt singularities was proposed. Here we consider its relation with K-semistability, which is an important concept in the study of K\"{a}hler-Einstein metrics on Fano varieties. In particular, we prove that for a $\bQ$-Fano variety $V$, the K-semistability of $(V, -K_V)$ is equivalent to the condition that the normalized volume is minimized at the canonical valuation $\ord_V$ among all $\bC^*$-invariant valuations on the cone associated to any positive Cartier multiple of $-K_V$. In this case, it's shown that $\ord_V$ is the unique minimizer among all $\bC^*$-invariant quasi-monomial valuations. These results allow us to give characterizations of K-semistability by using equivariant volume minimization, and also by using inequalities involving divisorial valuations over $V$. }

\tableofcontents

\section{Introduction}

Valuation theory is a classical subject which has been studied in different fields of mathematics for more than a century. As a celebrated example of their applications in algebraic geometry, Zariski proved the resolution of surface singularities by proving a local uniformization theorem showing that every valuation of a surface can be resolved. Recently valuations have also become important in studying singularities of plurisubharmonic functions in complex analysis (see for example \cite{BFJ08}). In this paper, we use the tool of valuation to study K-stability of Fano varieties.

The concept of K-stability was introduced by Tian \cite{Tia97} in his study of K\"{a}hler-Einstein problem on any Fano manifold. It was introduced there to test the properness of so-called Mabuchi-energy. 
This energy is defined on the space of smooth K\"{a}hler potentials whose critical points are K\"{a}hler-Einstein metrics. 
Tian proved in \cite{Tia97} that the properness of the Mabuchi-energy is equivalent to the existence of a K\"{a}hler-Einstein metric on a smooth Fano manifold without holomorphic vector field. 
The notion of K-stability was later generalized to a more algebraic formulation in \cite{Don02}. By the work of many people on Yau-Tian-Donaldson conjecture spanning a long period of time (see in particular \cite{Tia97}, \cite{Ber15} and \cite{CDS15, Tia15}), we now know that, for a smooth Fano manifold $V$, the existence of K\"{a}hler-Einstein metric on $V$ is equivalent to K-polystability of $V$. In \cite{Li13}, based on recent progresses on the subject, the author proved some equivalent characterizations of K-semistability. Roughly speaking, for smooth Fano manifolds, ``K-semistable" is equivalent to ``almost K\"{a}hler-Einstein'' (see \cite{Li13} for details). In particular, it was proved there that K-semistability is equivalent to the lower boundedness of the Mabuchi-energy.

Motivated by the work of \cite{MSY08} in the study of Sasaki-Einstein metric, which is a metric structure closely related to the K\"{a}hler-Einstein metric, in \cite{Li15a} the author proposed to understand K-semistability of Fano manifolds from the point view of volume minimizations. In this point of view, for any Fano manifold $V$, we consider the space of real valuations on the affine cone $X=C(V,-K_V)$ with center being the vertex $o$ (we denote this space by $\Val_{X,o}$) and study a normalized volume functional $\hvol$ on $\Val_{X,o}$. Such space of valuations were studied in \cite{JM10, BFFU13} with close relations to the theory of Berkovich spaces. The volume functional $\hvol(v)$ of a real valuation $v$ is defined to be $A_X(v)^n \vol(v)$, where $A_X(v)$ is the log discrepancy of $v$ (see \cite{JM10,BFFU13, Kol13}) and $\vol(v)$ is the volume of $v$ (see \cite{ELS03, Laz96, BFJ12}). In \cite{Li15a}, we conjectured that the Fano manifold $V$ is K-semistable if and only if $\hvol$ is minimized at the canonical valuations $\ord_V$. Here we prove one direction of this conjecture, i.e. volume minimizing implies K-semistability. Indeed we prove that any $\mathbb{Q}$-Fano variety $V$ is K-semistable if and only if $\hvol$ is minimized among $\bC^*$-invariant valuations (see Theorem \ref{main}). \footnote{Recently it has been clear that the general case can be reduced to the equivariant case considered in this paper, see \cite{LX16}.} 
%The complete confirmation of the other direction of the conjecture is the joint work with Yuchen Liu in the paper \cite{LL16} %(see also the Postscript Notes \ref{sec-ps}). 

The key idea to prove our result is to extend the calculations in \cite{Li15a} to show a generalization of a result in \cite{MSY08} in which they showed that the derivative of a volume functional on the space of (normalized) Reeb vector fields is the classical Futaki invariant originally defined in \cite{Fut83}. Here we will show that the derivative of the normalized volume at the canonical valuation $\ord_V$ is a variant of CM weight that generalizes the classical Futaki invariant. We achieve this by first deriving an 
integral formula for the normalized volumes for any real valuations centered at $o$. The derivation of this formula uses the tool of filtrations (see \cite{BC11}) and has a very concrete convex geometric meaning which allows us to use the theory of 
Okounkov bodies (\cite{Oko96, LM09}) and coconvex sets (\cite{KK14}) to understand it. 

Given this formula, then as in \cite{Li15a} we can take its derivative and use the recent remarkable work of Fujita \cite{Fuj15b} to show that Ding-semistability implies equivariant volume minimization, where Ding-semistability is a concept derived from Berman's work in \cite{Ber15} and is shown very recently to be equivalent to K-semistability in \cite{BBJ15, Fuj16}. Moreover, the volume formula shows that any $\bC^*$-invariant minimizer has a Dirac-type Duistermaat-Heckman measure. This together with some gap estimates allow us to conclude the uniqueness of $\bC^*$-invariant minimizing valuation that is quasi-monomial (Theorem \ref{thm-unique}).

For the other direction of implication, we directly relate
the volume function on $X$ (i.e. on $\Val_{X,o}$) to that on some degeneration of $X$ which appears in the definition of K-stability. The argument in this part also depends on the work of \cite{BHJ15}, which seems be the first to systematically use
the tool of valuations to study K-stability. On the other hand, we will also use our earlier work in \cite{LX14} which essentially helps us to reduce the test of K-stability to Tian's original definition. 

As a spin-off of our proof, we get the an equivalent characterization of K-semistability using divisorial valuations in Theorem \ref{thmdiv}. This generalizes a result about Fujita's divisorial stability in \cite{Fuj15a} (see also \cite{Li15a}).

We end this introduction by remarking that while the work here is a natural continuation of the idea initiated in \cite{Li15a}, the techniques and tools used here are mostly independent of those used in \cite{Li15a} and so this paper can be read independent of \cite{Li15a}.

%It is known that if we have a K\"{a}hler-Einstein metric on a smooth Fano manifold $V$ with anti canonical line bundle denoted by $-K_V$, then on the cone $X=C(V, -K_V)$ there is a Ricci-flat cone metric. This is a special example in the study of Ricci cone metric and the related Sasaki-Einstein metric. The Reeb vector field corresponding to a Sasaki-Einstein metric is unique if it exist. It was proved in \cite{} this Reeb vector field is the critical point of a volume function on the space of Reeb vector fields. If a Reeb vector field generates a free
%$S^1$-action, then the quotient space is a K\"{a}hler-Einstein Fano manifold. 
%Apriori we don't know if there is a minimizer of the volume functional or not, since the space of Reeb vector fields is not easily determined. So in the literature usually a torus action is fixed and one first minimize volume function on the fixed 

\section{Preliminaries}

\subsection{Normalized volumes of valuations}

In this section, we will briefly recall the concept of real valuations and their associated invariants (see \cite{ZS60} and \cite{ELS03} for details). We will also recall normalized volume functional defined in \cite{Li15a}. 

Let $X={\rm Spec}(R)$ be an $n$-dimensional normal affine variety. A real valuation $v$ on the function field $\bC(X)$ is a map $v: \bC(X)\rightarrow \bR$, satisfying:
\begin{enumerate}
\item
$v(fg)=v(f)+v(g)$;
\item $v(f+g)\ge \min\{v(f), v(g)\}$.
\end{enumerate}
In this paper we also require $v(\bC^*)=0$, i.e. $v$ is trivial on $\bC$. Denote by $\cO_v=\{f\in \bC(X); v(f)\ge 0\}$ the valuation ring of $v$. Then $\cO_v$ is a local ring. The valuation $v$ is said to be finite over $R$, or on $X$, if $\cO_v\supseteq R$. Let $\Val_X$ denote the space of all real valuations that are trivial on $\bC$ and finite over $R$. For any $v\in \Val_X$, the center of $v$ over $X$, denoted by ${\rm center}_X(v)$ or $c_X(v)$, is defined
to be the image of the closed point of ${\rm Spec}(\cO_v)$ under the map ${\rm Spec}(\cO_v)\rightarrow {\rm Spec}(R)=X$. 
For any $v\in \Val_{X}$, let $\Gamma'_v=v(R)$ denote its valuation semigroup and $\Gamma_v$ denote the abelian group generated by $\Gamma'_v$ inside $\bR$. Denote by ${\rm rat.rank}\; v$ the rational rank of the abelian group $\Gamma_v$. Denote by ${\rm tr.deg}\; v$ the transcendental degree of the field extension $K_v/\bC$ where $K_v$ is the residue field of $\cO_v$. Abhyankar proved the following inequality in \cite{Abh56}:
\begin{equation}\label{eq-Abh}
{\rm tr.deg}\; v+ {\rm rat.rank}\; v\le n.
\end{equation}
$v$ is called Abhyankar if the equality in \eqref{eq-Abh} holds. A valuation is called divisorial if ${\rm tr.deg}\; v=n-1$ (and hence ${\rm rat.rank}\; v=1$ by \eqref{eq-Abh}). Divisorial valuations are Abhyankar. Since we are working in characteristic 0, it's well-known that Abhyankar valuations are quasi-monomial (see \cite[Proposition 2.8]{ELS03}, \cite[Proposition 3.7]{JM10}). $v\in \Val_{X,o}$ is quasi-monomial if there exists a birational morphism $\mu: Y\rightarrow X$, and a regular system of parameters $z=\{z_1,\dots, z_r\}$ for the local ring $\cO_{Y,W}$ of $v$ on $Y$ ($W$ is the center of $v$ on $Y$) such that $v(z_1), \dots, v(z_r)$ freely generates the value group $\Gamma_v$. In this case, we can assume that there exist $r$ $\bQ$-linearly independent positive real numbers $\alpha_1, \dots, \alpha_r$ such that $v$ is defined as follows. For any $f\in R$, we can expand $\mu^*f=\sum_{m\in \bZ^r_{\ge 0}} c_m z^m$ with each $c_m$ either 0 or a unit in $\widehat{\cO_{Y,W}}$, and then:
\[
v(f)=\min\left\{\sum_{i=1}^r m_i \alpha_i; c_m\neq 0 \right\}.
\]
Following \cite{JM10}, we will also use log smooth models in addition to algebraic coordinates for representing quasi-monomial valuations. A log smooth pair over $X$ is a pair $(Y, D)$ with $Y$ regular and $D$ a reduced effective simple normal crossing divisor, together with a proper birational morphism $\mu: Y\rightarrow X$. We will denote by ${\rm QM}_{\eta}(Y,D)$ the set of all quasi-monomial valuations $v$ that can be described in the above fashion at the point $\eta\in Y$ with respect to coordinates $\{z_1, \dots, z_r\}$ such that $z_i$ defines at $\eta$ an irreducible component $D_i$ of $D$. We put ${\rm QM}(Y, D)=\cup_{\eta}{\rm QM}_{\eta}(Y,D)$.

From now on in this paper, we assume that $X$ has $\bQ$-Gorenstein klt singularity. Following \cite{JM10} and \cite{BFFU13}, we can define the log discrepancy for any valuation $v\in \Val_X$. This is achieved in three steps in \cite{JM10} and \cite{BFFU13}. Firstly, for a divisorial valuation $\ord_E$ associated to a prime divisor $E$ over $X$, define $A_X(E)=\ord_{E}(K_{Y/X})+1$, where $\pi: Y\rightarrow X$ is a smooth model of $X$ containing $E$. Next for any
quasi-monomial valuation $v\in {\rm QM}_\eta(Y,D)$ where $(Y,D=\sum_{k=1}^N D_k)$ is log smooth and $\eta$ is a generic point of an irreducible component of $D_1\cap \dots\cap D_N$, we define $A_X(v)=\sum_{k=1}^N v(D_k)A_X(D_k)$. Lastly for any valuation $v\in \Val_X$, we define 
\[
A_X(v)=\sup_{(Y,D)}A_X(r_{Y,D}(v))
\]
where $(Y,D)$ ranges over all log smooth models over $X$, and $r_{(Y,D)}: \Val_X\rightarrow {\rm QM}(Y,D)$ are contraction maps that induce a homeomorphism $\displaystyle\Val_X\rightarrow \lim_{\stackrel{\longleftarrow}{(Y,D)}}{\rm QM}(Y,D)$. For details, see \cite{JM10} and \cite[Theorem 3.1]{BFFU13}. We will need one basic property of $A_X$: for any proper birational morphism $Y\rightarrow X$, we have (see \cite[Remark 5.6]{JM10}, \cite[Proof 3.1]{BFFU13}):
\begin{equation}\label{ldbirat}
A_X(v)=A_Y(v)+v(K_{Y/X}).
\end{equation}

From now on, we also fix a closed point $o\in X$ with the corresponding maximal ideal of $R$ denoted by $\fm$. We will be interested in the space $\Val_{X,o}$ of all valuations $v$ with ${\rm center}_X(v)=o$. If $v\in \Val_{X,o}$, then $v$ is centered on the local ring $R_\fm$ (see \cite{ELS03}). In other words, $v$ is nonnegative on $R_\fm$ and is strictly positive on the maximal ideal of the local ring $R_\fm$.
For any $v\in \Val_{X,o}$, we consider its valuative ideals:
\[
\fa_p(v)=\{f\in R; v(f)\ge p\}.
\]
Then by \cite[Proposition 1.5]{ELS03}, $\fa_p(v)$ is $\fm$-primary, and hence is of finite codimension in $R$ (cf. \cite{AM69}). We define the volume of $v$ as:
\[
\vol(v)=\lim_{p\rightarrow +\infty}\frac{\dim_{\bC}R/\fa_p(v)}{p^n/n!}.
\]
By \cite{ELS03, Mus02, LM09, Cut12}, the limit on the right hand side exists and is equal to the multiplicity of the graded family of ideals $\fa_\bull=\{\fa_p\}$:
\begin{equation}\label{vol=mul}
\vol(v)=\lim_{p\rightarrow +\infty}\frac{e(\fa_p)}{p^n}=:e(\fa_\bullet),
\end{equation}
where $e(\fa_p)$ is the Hilbert-Samuel multiplicity of $\fa_p$.

Now we can define the normalized volume for any $v\in \Val_{X,o}$ (see \cite{Li15a}):
\[
\hvol(v)=\left\{
\begin{array}{lll}
A_X(v)^n\vol(v), & \mbox{ if } & A_X(v)<+\infty,\\
+\infty, & \mbox{ if } & A_X(v)=+\infty.
\end{array}
\right.
\]
The following estimates were proved in \cite{Li15a}. The second estimate motivates the definition of $\hvol(v)=+\infty$ when $A_X(v)=+\infty$.
\begin{thm}[{\cite[Corollary 2.6, Theorem 3.3]{Li15a}}]
Let $(X,o)$ be a $\mathbb{Q}$-Gorenstein klt singularity. The following uniform estimates hold:
\begin{enumerate}
\item 
There exists a positive constant $K=K(X,o)>0$ such that
$\hvol(v)>K$ for any $v\in \Val_{X,o}$. 
\item
There exists a positive constant $c=c(X,o)>0$ 
that 
\[
\hvol(v)\ge c \frac{A_X(v)}{v(\mathfrak{m})}, \text{ for all } v\in \Val_{X,o} \text{ with } A_X(v)<+\infty.
\]
\end{enumerate}
\end{thm}
From now on in this paper, we will always assume our valuation $v$ satisfies $A_X(v)<+\infty$. 

Notice that $\hvol(v)$ is rescaling invariant: $\hvol(\lambda v)=\hvol(v)$ for any $\lambda>0$. By Izumi-type theorem proved in \cite[Proposition 2.3]{Li15a}, we know that
$\hvol$ is uniformly bounded from below by a positive number on $\Val_{X,o}$. If $\hvol$ has a global minimum $v_*$ on $\Val_{X,o}$, i.e. $\hvol(v_*)=\inf_{v\in \Val_{X,o}}\hvol(v)$, then
we will say that $\hvol$ is globally minimized at $v_*$ over $(X,o)$. In \cite{Li15a}, the author conjectured that this holds for all $\bQ$-Gorenstein klt singularities and proved that this is the case under a continuity hypothesis.\footnote{H.Blum recently confirmed this conjecture in \cite{Blum16}} In this paper, we will be interested in the following $\bC^*$-invariant setting. 
\begin{defn}\label{defeqmin}
Assume that there is a $\bC^*$-action on $(X,o)$. We denote by $(\Val_{X,o})^{\bC^*}$ the space of $\bC^*$-invariant valuations in $\Val_{X,o}$. 
We say that $\hvol$ is $\bC^*$-equivariantly minimized at a ($\bC^*$-invariant) valuation $v_*\in (\Val_{X,o})^{\bC^*}$ if 
\[
\hvol(v_*)=\inf_{v\in (\Val_{X,o})^{\bC^*}}\hvol(v).
\]
\end{defn}

\subsection{K-semistability and Ding-semistability}
In this section we recall the definition of K-semistability, and its recent equivalence Ding-semistability.
\begin{defn}[{\cite{Tia97, Don02}, see also \cite{LX14}}]\label{defKstable}
Let $V$ be an $(n-1)$-dimensional $\bQ$-Fano variety. 
\begin{enumerate}
\item For $r>0\in \bQ$ such that $r^{-1}K^{-1}_V=L$ is Cartier. A test configuration (resp. a semi test configuration) of $(X, L)$ consists of the following data
\begin{itemize}
\item A variety $\cV$ admitting a $\bC^*$-action and a $\bC^*$-equivariant morphism $\pi: \cV\rightarrow \bC$, where the action of $\bC^*$ on $\bC$ is given by the standard multiplication.
\item A $\bC^*$-equivariant $\pi$-ample (resp. $\pi$-semiample) line bundle $\cL$ on $\cV$ such that $(\cV, \cL)|_{\pi^{-1}(\bC\backslash\{0\})}$ is equivariantly isomorphic to $(V, r^{-1}K^{-1}_V)\times \bC^*$ with the natural $\bC^*$-action.
\end{itemize}
A test configuration is called a special test configuration, if the following are satisfied
\begin{itemize}
\item $\cV$ is normal, and $\cV_0$ is an irreducible $\bQ$-Fano variety;
\item $\cL=r^{-1}K^{-1}_{\cV/\bC}$.
%and the $\bC^*$-action on $\cL_0=-r^{-1}K_{\cV_0}\rightarrow \cV_0$ is the induced by the action on $\cV_0$.
\end{itemize}
\item 
Assume that $(\cV, \cL)\rightarrow \bC$ is a normal test configuration. Let $\bar{\pi}: (\ocV, \bar{\cL})\rightarrow \bP^1$ be the
natural equivariant compactification of $(\cV, \cL)\rightarrow \bC$. The CM weight of $(\cV, \cL)$ is defined by the intersection formula (see \cite{Wan12, Oda13}):
\[
{\rm CM}(\cV, \cL)=\frac{1}{n(-K_V)^{n-1}}\left((n-1)r^{n}\bar{\cL}^{n}+n r^{n-1} \bar{\cL}^{n-1}\cdot K_{\ocV/\bP^1}\right).
\]
\item
\begin{itemize}
\item %The pair $(V, -K_{V})$ 
$V$ is called K-semistable if $\CM(\cV, \cL)\ge 0$ for any normal test configuration $(\cV, \cL)/\bC^1$ of $(V, r^{-1}K_V^{-1})$.
\item %The pair is $(V, -K_{V})$ 
$V$ is called K-polystable if $\CM(\cV, \cL)\ge 0$ for any normal test configuration $(\cV, \cL)/\bC^1$ of $(V, r^{-1}K_V^{-1})$, and the equality holds if and only if $\cV\cong V\times\bC^1$.
\end{itemize}
\end{enumerate}
\end{defn}
\begin{rem}\label{remaction}
For any special test configuration, the $\bC^*$-action on $\cV_0$ induces a natural action on $K_{\cV_0}^{-1}$ by pushing forward the holomorphic $(n,0)$-vectors. Because $\cL_0=r^{-1}K_{\cV_0}^{-1}$, we can always assume that the $\bC^*$-action on $\cL_0$ is induced from this natural $\bC^*$-action. Moreover, for special test configurations, CM weight reduces to a much simpler form:
\begin{equation}\label{CMstc}
{\rm CM}(\cV, \cL)=\frac{-r^n \bar{\cL}^n}{n(-K_V)^{n-1}}=-\frac{(-K_{\cV/\bP^1})^n}{n(-K_V)^{n-1}}.
\end{equation}
\end{rem}
We will need the concept of Ding-semistability, which was derived from Berman's work in \cite{Ber15}.
\begin{defn}[\cite{Ber15, Fuj15b}]
\begin{enumerate}
\item
Let $(\cV, \cL)/\bC^1$ be a normal semi-test configuration of $(V, r^{-1}K^{-1}_V)$ and $(\ocV, \bar{\cL})/\bP^1$ be its natural
compactification. Let $D_{(\cV, \cL)}$ be the $\bQ$-divisor on $\cV$ satisfying the following conditions:
\begin{itemize}
\item The support ${\rm Supp}D_{(\cV, \cL)}$ is contained in $\cV_0$. 
\item The divisor $-r^{-1} D_{(\cV, \cL)}$ is a $\bZ$-divisor corresponding to the divisorial sheaf $\bar{\cL}(r^{-1}K_{\ocV/\bP^1})$.
\end{itemize}
\item 
The Ding invariant ${\rm Ding}$ invariant ${\rm Ding}(\cV, \cL)$ of $(\cV, \cL)/\bC$ is defined as:
\[
{\rm Ding}(\cV, \cL):=\frac{-r^n \bar{\cL}^{n}}{n (-K_V)^{n-1}}-\left(1-\lct(\cV, D_{(\cV,\cL)}; \cV_0)\right).
\]
\item
$V$ is called Ding semistable if ${\rm Ding}(\cV, \cL)\ge 0$ for any normal test configuration $(\cV, \cL)/\bC$ of $(V, -r^{-1}K_V)$.
\end{enumerate}
\end{defn}
Notice that ${\rm Ding}(\cV, \cL)={\rm CM}(\cV, \cL)$ for special test configurations. 
It was proved in \cite{LX14} that to test K-semistability (or K-polystability), one only needs to test on special test configurations. By the recent work in \cite{BBJ15, Fuj16}, the same is true for Ding-semistability. Moreover we have the following equivalence result:
\begin{thm}[\cite{BBJ15, Fuj16}]\label{thm-BBJ}
For a $\mathbb{Q}$-Fano variety $V$, $V$ is K-semistable if and only if $V$ is Ding-semistable.
\end{thm}
\begin{rem}
In \cite{BBJ15} Berman-Boucksom-Jonsson outlined a proof of the above result when $V$ is a smooth Fano manifold. They showed that it's sufficient to check Ding semistability on special test configurations. This was achieved by showing that the (normalized) Ding invariant is decreasing along a process that transforms any test configuration into a special test configuration. This process was first used in \cite{LX14} for proving results on K-(semi)stability of $\mathbb{Q}$-Fano varieties. 
The equivalence of Ding-semistability and K-semistability for the general case of $\mathbb{Q}$-Fano varieties has been proved in \cite{Fuj16} by similar and more detailed arguments. 
\end{rem}

We will need another result of Fujita, which was proved by applying a criterion for Ding semistability (\cite[Proposition 3.5]{Fuj15b}) to a sequence of semi-test configurations constructed from a filtration. 
%We refer the detailed definitions about filtrations to \cite{BC11} (see also \cite{WN12, BHJ} and \cite[Section 4.1]{Fuj15b}). 
We here briefly recall the relevant definitions about filtrations and refer the details to \cite{BC11} (see also \cite{BHJ15} and \cite[Section 4.1]{Fuj15b}).
\begin{defn}\label{defn-gdfiltr}
 A good filtration of a graded $\bC$-algebra $R=\bigoplus_{m=0}^{+\infty}R_m$ is a decreasing, % left continuous,
multiplicative
and linearly bounded $\bR$-filtrations of $R$. In other words, for each $m\ge 0\in \bZ$, there is a family of subspaces $\{\cF^{x}R_m\}_{x\in \bR}$ of $R_m$ such that:
\begin{enumerate}
\item $\cF^x R_m\subseteq \cF^{x'}R_m$, if $x\ge x'$;
%\item $\cF^xR_m=\bigcap_{x'<x}\cF^{x'}R_m$; 
\item $\cF^x R_m\cdot \cF^{x'} R_{m'}\subseteq \cF^{x+x'}R_{m+m'}$, for any $x, x'\in \bR$ and $m, m'\in \bZ_{\ge 0}$;
\item $e_{\min}(\cF)>-\infty$ and $e_{\max}(\cF)<+\infty$, where $e_{\min}(\cF)$ and $e_{\max}(\cF)$ are defined by the following 
operations:
\begin{equation}
\def\arraystretch{1.5}
\begin{array}{l}
e_{\min}(R_m,\cF)=\inf\{t\in\bR; \cF^t R_m\neq R_m\}; \\
e_{\max}(R_m,\cF)=\sup\{t\in\bR; \cF^t R_m\neq 0\};\\
\displaystyle e_{\min}(\cF)=e_{\min}(R_{\LargerCdot}, \cF)=\liminf_{m\rightarrow +\infty} \frac{e_{\min}(R_m, \cF)}{m}; \\
\displaystyle e_{\max}(\cF)=e_{\max}(R_{\LargerCdot}, \cF)=\limsup_{i\rightarrow +\infty} \frac{e_{\max}(R_m, \cF)}{m}. 
\end{array}
\end{equation}
\end{enumerate}
\end{defn}
Following \cite{BHJ15}, given any good filtration $\{\cF^{x}R_m\}_{x\in\bR}$ and $m\in \bZ_{\ge 0}$, the successive minima is the decreasing sequence 
\[
\lambda^{(m)}_{\max}=\lambda^{(m)}_1\ge \cdots \ge \lambda^{(m)}_{N_m}=\lambda^{(m)}_{\min}
\]
where $N_m=\dim_{\bC} R_m$, defined by:
\[
\lambda^{(m)}_j=\max\left\{\lambda \in \bR; \dim_{\bC} \cF^{\lambda} R_m \ge j \right\}.
\]
%In this paper, the
%filtrations we shall consider are always decreasing, multiplicative, linearly bounded $\bR$-filtrations of the graded algebra $R=\bigoplus_{m=0}^{+\infty}H^0(V, L^{\otimes m})=\bigoplus_{m=0}^{+\infty}R_m$. %, and we will call such filtrations to be {\it good}.
%For such filtrations, we define the following invariants after \cite{BC11}:
%\begin{equation}
%\def\arraystretch{1.5}
%\begin{array}{l}
%e_{\min}(R_m,\cF)=\inf\{t\in\bR; \cF^t R_m\neq R_m\}; \\
%e_{\max}(R_m,\cF)=\sup\{t\in\bR; \cF^t R_m\neq 0\};\\
%\displaystyle e_{\min}(\cF)=e_{\min}(R_{\bull}, \cF)=\liminf_{m\rightarrow +\infty} \frac{e_{\min}(R_m, \cF)}{m}; \\
%\displaystyle e_{\max}(\cF)=e_{\max}(R_{\bull}, \cF)=\limsup_{i\rightarrow +\infty} \frac{e_{\max}(R_m, \cF)}{m}. 
%\end{array}
%\end{equation}
Denote $R^{(t)}=\bigoplus_{k=0}^{+\infty} \cF^{kt}R_k$. When we want to emphasize the dependence of $R^{(t)}$ on the filtration $\cF$, we also denote $R^{(t)}$ by $\cF R^{(t)}$.

From now on, we let $R_m=H^0(V, mL)$.
The following concept of volume will be important for us:
\begin{equation}
\vol\left(R^{(t)}\right)=\vol\left(\cF R^{(t)}\right):=\limsup_{k\rightarrow+\infty}\frac{\dim_{\bC}\cF^{mt}H^0(V, mL)}{m^{n-1}/(n-1)!}.
\end{equation}
We need the following lemma by in \cite{BHJ15} in our proof of uniqueness result.
\begin{prop}[{\cite{BC11}, \cite[Corollary 5.4]{BHJ15}}]\label{BHJvol}

\begin{enumerate}
\item 
The probability measure 
\[
\nu_m:=\frac{1}{N_m} \sum_{j}\delta_{m^{-1}\lambda^{(m)}_j}=-\frac{d}{d t} \frac{{\rm dim}_{\bC} \cF^{mt}H^0(V, mL)}{N_m}
\]
converges weakly as $m\rightarrow+\infty$ to the probability measure:
\[
\nu:=-V^{-1} d\; \vol\left(R^{(t)}\right)=-V^{-1} \frac{d}{d t} \vol\left(R^{(t)}\right) dt.
\]
\item The support of the measure $\nu$  is given by ${\rm supp}\; \nu=[\lambda_{\min}, \lambda_{\max}]$ with 
\begin{align}
&
\displaystyle \lambda_{\min}:=\inf\left\{t\in \bR | \vol\left(R^{(t)}\right)<L^{n-1} \right\}; \label{lambdamin} \\
&
\lambda_{\max}=\lim_{m\rightarrow+\infty}\frac{\lambda_{\max}^{(m)}}{m}=\sup_{m\ge 1}\frac{\lambda_{\max}^{(m)}}{m}.\label{lambdamax}
\end{align}
\end{enumerate}
Moreover, $\nu$ is absolutely continuous with respect to the Lebesgue measure, except perhaps for a Dirac mass at $\lambda_{\max}$.
\end{prop}

%\begin{eqnarray*}
%&&e_{\min}(R_i,\cF)=\inf\{t\in\bR; \cF^t R_i\neq R_i\}\\
%&&e_{\max}(R_i,\cF)=\sup\{t\in\bR; \cF^t R_i\neq 0\}\\
%&&e_{\min}(R_{\bull}, \cF)=\liminf_{i\rightarrow +\infty} \frac{\inf\{t\in\bR; \cF^t R_i\neq R_i\}}{i}>-\infty;\\
%&&e_{\max}(R_{\bull}, \cF)=\limsup_{i\rightarrow +\infty} \frac{\sup\{t\in\bR; \cF^t R_i\neq 0\}}{i}<+\infty.
%\end{eqnarray*} 
Following \cite{Fuj15b}, we also define a sequence of ideal sheaves on $V$:
\begin{equation}\label{filtideal}
I^{\cF}_{(m,x)}={\rm Image}\left(\cF^xR_m\otimes L^{-m}\rightarrow \cO_V\right), %=\cF^{x}R_m
\end{equation}
and define $\overline{\cF}^x R_m:=H^0(V, L^m\cdot I^{\cF}_{(m,x)})$ to be the saturation of $\cF^x R^m$ such that
$\cF^x R^m\subseteq \overline{\cF}^x R_m$.
$\cF$ is called saturated if $\overline{\cF}^x R_m=\cF^x R_m$ for any $x\in \bR$ and $m\in \bZ_{\ge 0}$. 
Notice that with our notations we have:
\[
\vol\left(\overline{\cF}R^{(t)}\right):=\limsup_{k\rightarrow+\infty}\frac{\dim_{\bC}\overline{\cF}^{kt}H^0(V, kL)}{k^{n-1}/(n-1)!}.
\]
The following result of Fujita will be crucial for us in proving that Ding-semistability implies volume minimization.
\begin{thm}[{\cite[Theorem 4.9]{Fuj15b}}]\label{Fujthm}
Assume $(V, -K_V)$ is Ding-semistable. Let $\cF$ be a {\it good} 
$\bR$-filtration of $R=\bigoplus_{m=0}^{+\infty} H^0(V, mL)$ where $L=r^{-1}K_V^{-1}$. 
Then the pair $(V\times\bC, \cI_\bull^r\cdot (t)^{d_\infty})$ is sub log canonical, where
\begin{eqnarray}\label{eq-cI}
&&\mathcal{I}_m=I^{\cF}_{(m, m e_{+})}+I^{\cF}_{(m, m e_{+}-1)}t^1+\cdots+I^{\cF}_{(m, m e_{-}+1)} t^{m(e_+-e_-)-1}+(t^{m(e_+-e_-)}),\\
&&d_{\infty}=1-r(e_+-e_-)+\frac{r^n}{(-K_V)^{n-1}}\int_{e_-}^{e_+}\vol\left(\overline{\cF} R^{(t)}\right)dt\nonumber\\
%&&
%\cI^{\cF}_{(r,x)}={\rm Image}\left(\cF^xR_r\otimes L^{-r}\rightarrow \cO_V\right)=\cF^{x}R_r
\end{eqnarray}
and $e_+, e_-\in \bZ$ with $e_+\ge e_{\max}(R_\bull, \cF)$ and $e_-\le e_{\min}(R_\bull, \cF)$.
\end{thm}
\begin{rem}\label{rem-D2K}
By Theorem \ref{thm-BBJ}, we can replace ``Ding-semistable" by ``K-semistable" in Fujita's result.
\end{rem}

\section{Statement of Main Results}\label{sec-results}
Let $V$ be a $\bQ$-Fano variety. If $L=r^{-1}K^{-1}_V$ is an ample Cartier divisor for an $r\in \bQ_{>0}$, we will denote the affine cone by $C_r:=C(V, L)={\rm Spec}\bigoplus_{k=0}^{+\infty}H^0(V, kL)$. Notice that $C_r$ has $\bQ$-Gorenstein klt singularities at the vertex $o$ (see \cite[Lemma 3.1]{Kol13}). $(C_r,o)$ has a natural $\bC^*$-action. Denote by $v_0$ the canonical $\bC^*$-invariant divisorial valuation $\ord_V$ where $V$ is considered as the exceptional divisor of the blow up $Bl_oX\rightarrow X$.
The following is the main theorem of this paper, which in particular confirms one direction of part 4 of the conjecture in \cite{Li15a}. 
\begin{thm}\label{main}
Let $V$ be a $\bQ$-Fano variety. The following conditions are equivalent:
\begin{enumerate}
\item $(V, -K_V)$ is K-semistable;
\item For an $r_*\in \bQ_{>0}$ such that $L=-r_*^{-1}K_V$ is Cartier, $\hvol$ is $\bC^*$-equivariantly minimized at $\ord_V$ over $(C_{r_*}, o_*)$;
\item For any $r\in \bQ_{>0}$ such that $L=-r^{-1}K_V$ is Cartier, $\hvol$ is $\bC^*$-equivariantly minimized at $\ord_V$ over $(C_r, o)$.
%\item For an $r_*\in \bQ_{>0}$ such that $L=-r_*^{-1}K_V$ is Cartier, $\hvol$ is globally minimized at $\ord_V$ over $(C_{r_*}, o_*)$;
%\item For any $r\in \bQ_{>0}$ such that $L=-r^{-1}K_V$ is Cartier, $\hvol$ is globally minimized at $\ord_V$ over $(C_r, o)$;
\end{enumerate}
\end{thm}
\begin{examp}
\begin{enumerate}
\item
Over $(\bC^n, 0)$, $\hvol$ is globally minimized at the valuation $v_*=\ord_E$, where $E\cong \bP^{n-1}$ is the exceptional divisor of the blow up $Bl_o\bC^n\rightarrow \bC^n$.
This follows from an inequality of de-Fernex-Ein-Musta\c{t}\u{a} as pointed out in \cite{Li15a}. As a corollary we have a purely algebraic proof of the following statement:
\begin{cor}
$\mathbb{P}^{n-1}$ is K-semistable.
\end{cor}
As pointed out in \cite{PW16}, an algebraic proof of this result could also be given by combining Kempf's result on Chow stabilities of rational homogeneous varieties, and the fact that asymptotic Chow-semistability implies K-semistability.  
\item
Over $n$-dimensional $A_1$ singularity $X^n=\{z_1^2+\dots+z^2_{n+1}=0\}\subset\bC^{n+1}$ with $o=(0,\dots, 0)$, $\hvol$ is minimized among $\bC^*$-invariant valuations at $v_*=\ord_E$, where $E\cong Q$ (the smooth quadric  
hypersurface) is the exceptional divisor
of the blowup $Bl_oX\rightarrow X$. This follows from the above theorem plus the fact that $Q$ has K\"{a}hler-Einstein metric and hence is K-polystable. In \cite{Li15a}, $v_*$ was shown to minimize $\hvol$ among valuations of the special form $v_{\bx}$ determined by weight $\bx\in \bR^{n+1}_{>0}$. In the following paper \cite{LL16}, we will show that $v_*=\ord_Q$ indeed globally minimizes $\hvol$.
\end{enumerate}
\end{examp}

Under any one of the equivalent conditions of Theorem \ref{main}, we can prove the uniqueness of minimizer among $\bC^*$-invariant quasi-monomial valuations.
Further existence and uniqueness results for general klt singularities will be dealt with in a forthcoming paper \cite{LX16}.
\begin{thm}\label{thm-unique}
Assume $V$ is K-semistable. Then $\ord_V$ is the unique minimizer among all $\bC^*$-invariant quasi-monomial valuations. 
\end{thm}

Through the proof of the above theorem, we get also a characterization of K-semistability using divisorial valuations of the field $\bC(V)$.
%by essentially combining the existing literature. We state it here for the convenience of the reader.
\begin{defn}
For any divisorial valuation $\ord_F$ over $V$ where $F$ is a prime divisor on a birational model of $V$, we define the quantity:
\begin{equation}
\Theta_V(F)=A_V(F)-\frac{r^n}{(-K_V)^{n-1}}\int^{+\infty}_0 \vol\left(\cF_F R^{(t)}\right)dt,
\end{equation}
where $L=-r^{-1}K_V$ is an ample Cartier divisor for $r\in \bQ_{>0}$ and $\cF_F R_\bullet=\{\cF_F R_k\}$ is the filtration 
defined in \eqref{divfil} by $\ord_F$ on $R=\bigoplus_{k=0}^{+\infty} H^0(V, kL)$.
\end{defn}
\begin{rem}
\begin{enumerate}
\item
It's easy to verify that the $\Theta_V(F)$ does not depend on the choice of $L$. See Lemma \ref{intresc}. In particular when $-K_V$ is Cartier then we can choose $L=-K_V$ and $r=1$. 
\item
If $F$ is a prime divisor on $V$, then $A_V(F)=1$ and $\Theta_V(F)$ is nothing but (a multiple of) Fujita's invariant $\eta(F)$ defined in \cite{Fuj15a}. In particular, the following theorem can be seen as a generalization of result about Fujita's divisorial stability in \cite{Fuj15a} (see also \cite{Li15a}).
\end{enumerate}
\end{rem}
\begin{thm}\label{thmdiv}
Assume $V$ is a $\bQ$-Fano variety. Then $V$ is K-semistable if and only if $\Theta_V(F)\ge 0$ for any divisorial 
valuation $\ord_F$ over $V$. Moreover, if $\Theta_V(F)>0$ for any divisorial valuation $\ord_F$ over $V$, then $V$ is K-stable.
 \end{thm}
%We expect that the converse to the last statement is also true and there should be a K-polystable version. We hope to discuss them elsewhere. 

The rest of this paper is devoted to the proof of Theorem \ref{main}, Theorem \ref{thm-unique} and Theorem \ref{thmdiv}. Before we go into details, we make some general discussion on idea of the proof of Theorem \ref{main}, from which the proof of Theorem \ref{thm-unique} and the proof Theorem \ref{thmdiv} will be naturally derived.

In Section \ref{sec-D2V}, we will prove that K-semistability, or equivalently Ding-semistability (by Theorem \ref{thm-BBJ}), implies equivariant volume minimization. More precisely we prove $(1)\Rightarrow (3)$. To achieve this, we will generalize the calculations in \cite{Li15a} where we compared the
$\hvol(v_0)$, the normalized volume of the canonical valuation $v_0=\ord_V$,  with $\hvol(v_1)$ for some special $\bC^*$-invariant divisorial valuation $v_1$. %However since the continuity of volume functions at non-quasi-monomial valuations is not known at present, it's hard to use the approximation approach. So 
Here we will deal with general real valuations directly using the tool of filtrations. More precisely, we will associate an appropriate graded filtration to any real valuation $v_1$
with $A_X(v_1)<+\infty$, and will be able to
%we will consider the family of weight functions $\{v_t; t\in [0,1]\}$ as in Section \ref{secwtfn}, which interpolates $v_0$ and any fixed real valuation $v_1$. Under the assumption that $v_1$ is $\bC^*$-invariant, 
calculate $\vol(v_1)$ using the volumes of associated sublinear series. Izumi's theorem is a key to make this work, ensuring that we get linearly bounded filtrations. 
Motivated by the formulas in the case of $\bC^*$-invariant valuations, we will consider a concrete convex interpolation between $\hvol(v_0)$ and $\hvol(v_1)$. 
%By the formula we obtained in \eqref{eqvolvt2}, we will see that $\hvol(v_t)$ is convex with respect to $t$.
By the convexity, we just need to show that the directional derivative of the interpolation at $v_0$ is nonnegative. The formulas up to this point works for any real valuation not necessarily $\bC^*$-invariant.

In the $\bC^*$-invariant case, same as for the case in \cite{Li15a}, it will turn out that the derivative matches Fujita's formula and his result in Theorem \ref{Fujthm} (\cite{Fuj15b}) gives exactly the nonnegavitity. As mentioned earlier, we will deal with the non-$\bC^*$-invariant case in \cite{LL16}. % In this respect, Fujita's formula is again important as the case in \cite{Li15a}, giving exactly the derivative of $\hvol$ at the canonical valuation.  %Another point in the argument is that Fujita's estimate is actually sharp and can be used without going into inversion of adjunctions.

In Section \ref{secvol2semi}, we will prove the implication $(2)\Rightarrow (1)$. Since $(3)\Rightarrow (2)$ trivially, this completes the proof of Theorem \ref{main}. 
This direction of the implication depends on the work \cite{BHJ15} and \cite{LX14}. We learned from \cite{BHJ15} that test configurations can be studied using the point of view of valuations: the irreducible components of the central fibre give rise to $\bC^*$-invariant divisorial valuations, which however is not finite over the cone in general. In our new point of view, these $\bC^*$-invariant valuations should be considered as tangent vectors at the canonical valuation. By \cite{LX14}, special test configurations are enough for testing K-semistability. For any special test configuration, there is just one tangent vector and so it's much easier to deal with. Then again the point is that the derivative of $\hvol$ along the tangent direction is exactly the Futaki invariant on the central fibre. So we are done. 

%For completeness, We consider a family of quasi-monomial valuations starting from $v_0$ and ``tangent" to a valuation on $X$ determined by $F$, similar to what we did in \cite{Li15a}. Using the volume formula, we will see that the derivative of the normalized volume is given by $\Theta_V(F)$. Combining this with the proof of Theorem \ref{main} we can finish its proof.

\section{K-semistable implies equivariant volume minimizing}\label{sec-D2V}
Assume that $V$ is a $\bQ$-Fano variety and $L=-r^{-1}K_V$ is an ample Cartier divisor for a fixed $r\in \bQ_{>0}$. Define $R=\bigoplus_{i=0}^{+\infty} H^0(V, iL)=:\bigoplus_{i=0}^{+\infty} R_i$ and denote by $X=C(V, L)={\rm Spec}(R)$ the affine cone over $V$ with the polarization $L$. Let $v_0=\ord_V$ be the ``canonical valuation" on $R$ corresponding to the canonical $\bC^*$-action on $X$.
\begin{thm}\label{ks2mv}
If $V$ is K-semistable, then $\hvol$ is $\mathbb{C}^*$-equivariantly minimized at the canonical valuation $\ord_V$ over $(X=C(V,L),o)$.
\end{thm}
The rest of this section is devoted to the proof of this theorem.
\subsection{A general volume formula} 
Let $v_1$ be any real valuation centered at $o$ with $A_X(v_1)<+\infty$. In this section, we don't assume $v_1$ is $\bC^*$-invariant. 
We will derive a volume formula for  $\vol(v_1)$ with the help of an appropriately defined filtration associated to $v_1$.
%We will prove $\hvol(v_1)\ge \hvol(v_0)$ with the help of an appropriate graded filtration associated to any real valuation. 
To define this filtration, we decompose any $g\in R=\bigoplus_{k=0}^{+\infty}R_k$ into ``homogeneous"
components $g=g_{k_1}+\cdots+g_{k_p}$ with $g_{k_j}\neq 0$ and $k_1<k_2<\cdots<k_p$ and define $\bin(g)$ to be the ``initial component" $g_{k_1}$. 
Then the filtration we will consider is the following:
\begin{equation}\label{mirafil}
\cF^x R_k=\{f\in R_k; \exists g\in R \text{ such that } v_1(g)\ge x \text{ and } \bin(g)=f\}\bigcup\{0\}.
\end{equation}
We notice that following properties of $\cF^x R_{\bull}$.
\begin{enumerate}
\item 
For fixed $k$, 
$\{\cF^x R_k\}_{x\in \bR}$ is a family of decreasing $\bC$-subspace of $R_k$. 
\item
$\cF$ is multiplicative:
$v_1(g_i)\ge x_i$ and $\bin(g_i)=f_i\in R_{k_i}$ implies 
\[
v_1(g_1 g_2)\ge x_1+x_2, \quad \bin(g_1 g_2)=f_1\cdot f_2\in R_{k_1+k_2}.
\]
\item 
If $A(v_1)<+\infty$, then $\cF$ is linearly bounded. Indeed by Izumi's theorem in \cite[Proposition 1.2]{Li13} (see also \cite{Izu85, Ree89, BFJ12}), there exist $c_1, c_2\in (0, +\infty)$ such that
\[
c_1 v_0\le v_1\le c_2 v_0.
\]
If $v_1(g)\ge x$, then  $k:=v_0(g)=v_0(\bin(g))\ge c_2^{-1}x$. This implies that if $x> c_2 k$, then 
$\cF^x R_k=0$. So $\cF$ is linearly bounded from above. On the other hand, if $v_0(f)= k$, then $v_1(f)\ge c_1 k$. This implies that if $x \le c_1 k$, then $\cF^x R_k=R_k$. So $\cF$ is linearly bounded from below. 
Note that the argument in particular shows the following relation:
\begin{equation}\label{cFlbd}
\inf_{\fm}\frac{v_1}{v_0}\le e_{\min}(\cF)\le e_{\max}(\cF)\le \sup_{\fm}\frac{v_1}{v_0}.
\end{equation}
\end{enumerate}
For later convenience, from now on we will fix the following constant:
\begin{equation}\label{defc1}
c_1:=\inf_{\fm}\frac{v_1}{v_0}>0.
\end{equation}
To state the following lemma, we first introduce a notation (see \cite[Section 2]{DH09}). If $Z$ is a proper closed subscheme of $Y$ and $v\in \Val_Y$, let $\cI_Z$ be the ideal sheaf of $Z$ and define:
\begin{equation}\label{eq-vZ}
v(Z)=v(\cI_Z):=\min\left\{v(\phi); \phi \in \cI_Z(U), U\cap {\rm center}_Y(v)\neq \emptyset\right\}.
\end{equation}
%The following lemma holds for any $v_1\in \Val_{X,o}$ %with $A_X(v_1)<+\infty$ 
%and does not depend on the filtration. 
\begin{lem}\label{lemdefc1}
For any $v_1\in \Val_{X,o}$, we have the following identity:
\begin{equation}\label{infv1}
c_1=v_1(V)=\inf_{i>0} \frac{v_1(R_i)}{i},
\end{equation}
where $v_1(V)$ is defined as in \eqref{eq-vZ} and $v_1(R_i)=\inf\{v(f); f\neq 0\in R_i\}$.
\end{lem}
\begin{proof}
Let $\mu: Y:=Bl_oX\rightarrow X$ be the blow up of $o$ with exceptional divisor $V$. Then $Y$ can be identified with the global space of the line bundle 
$L^{-1}\rightarrow V$ and we have the natural projection $\pi: Y\rightarrow V$.
For any $f\in H^0(V, iL)$, we have $v_1(f)=v_1(\mu^*f)\ge i v_1(V)$. To see this, we write $f=h\cdot s^i$ on any affine open set $U$ of $X$ where $h\in \mathcal{O}_X(U)$ and
$s$ is a generator of $\cO_V(L)(U)$. Then $\mu^*f=\mu^*h\cdot s^i$ can be considered as a regular function on $\pi^{-1}U\cong U\times \mathbb{C}$. Since $V\cap \pi^{-1}(U)=\{s=0\}$, we have $v_1(s)=v_1(V)$ and $v_1(\mu^*f)=v_1(h\cdot s^i)=v_1(h)+i v_1(s)\ge i v_1(V)$. The identity holds if and only $v_1(h)=0$, i.e. if and only $f$ does not vanish on
the center of $v_1$ over $Y$ which is contained in $V$. 
%From this we get:
%\begin{equation}
%v_1(V)\le \inf_{i>0}\frac{v_1(R_i)}{i}.
%\end{equation}
When $m$ is sufficiently large, $mL$ is globally generated on $V$ so that we can find $f\in H^0(V, mL)$ such that $f$ does not vanish on the center of $v$ on $V$. Then 
we have $\frac{v_1(f)}{v_0(f)}=v_1(f)/m=v_1(\mu^*f)/m=v_1(V)$. So from the above discussion we get:
\begin{equation*}
c_1=\inf_{\fm}\frac{v_1}{v_0}\le \inf_{i>0} \frac{v_1(R_i)}{i}=v_1(V).
\end{equation*}
For any $t>c_1$, there exists $f\in R$ such that $v_1(f)/v_0(f)<t$. Decompose $f$ into components: $f=f_{i_1}+\dots+f_{i_k}$, where $f_{i_j}\in R_{i_j}=H^0(V, i_j L) (1\le j\le k)$ and $1\le i_1<\dots<i_k$. Then there exists $i_*\in \{i_1, \dots, i_k\}$ such that
\begin{eqnarray*}
\frac{v_1(f)}{v_0(f)}&=&\frac{v_1(f_{i_1}+\dots+f_{i_k})}{i_1}\ge \frac{\min\{v_1(f_{i_1}), \dots, v_1(f_{i_k})\}\}}{i_1}\\
&=&\frac{v_1(f_{i_*})}{i_1}\ge \frac{v_1(f_{i_*})}{i_*}\ge v_1(V).
\end{eqnarray*}
So we get $t > v_1(V)$ for any $t>c_1$. As a consequence $c_1\ge v_1(V)$.
\end{proof}
The reason why we can use filtration in \eqref{mirafil} to calculate volume comes from the following observation.
\begin{prop}\label{propdim}
For any $m\in \bR$, we have the following identity:
\[
\sum_{k=0}^{+\infty}\dim_{\bC}\left(R_k/\cF^m R_k\right)=\dim_{\bC}\left(R/\fa_m(v_1)\right).
\]
\end{prop}
Notice that because of linear boundedness, the sum on the left hand side is a finite sum. More precisely, by the above discussion, when $k\ge m/c_1$, then $\dim_{\bC}\left(R_k/\cF^m R_k\right)=0$.
\begin{proof}
For each fixed $k$, let $d_k=\dim_{\bC}(R_k/\cF^m R_k)$. Then we can choose a basis of $R_k/\cF^m R_k$:
\[
\left\{[f^{(k)}_i]_k \; |\; f^{(k)}_i\in R_k, 1\le i\le d_k \right\},
\] 
where $[\cdot]_k$ means taking quotient class in $R_k/\cF^mR_k$. Notice that for $k\ge \lceil m/c_1 \rceil$, the set is empty.
We want to show that the set 
\[
\mathfrak{B}:=\left\{[f^{(k)}_i]\;|\; 1\le i\le d_k, 0\le k\le \lceil m/c_1\rceil-1 \right\},
\] 
is a basis of $R/\fa_m(v_1)$, where $[\cdot]$ means taking quotient in $R/\fa_m(v_1)$. 
\begin{itemize}
\item
We first show that $\mathfrak{B}$ is a linearly independent set. 
For any nontrivial linear combination of $[f^{(k)}_i]$:
 \[
\sum_{k=0}^{N} \sum_{i=1}^{d_k}c^{(k)}_i [f^{(k)}_i]=\left[\sum_{k=0}^{N} \sum_{i=1}^{d_k} c^{(k)}_i f^{(k)}_i\right]=[f^{(k_1)}+\cdots+f^{(k_p)}]=:[F],
 \] 
where $f^{(k_j)}\neq 0\in R_{k_j}\setminus \cF^m R_{k_j}$ and 
$k_1<k_2<\cdots<k_p$. In particular $\bin(F)=f^{(k_1)}\not\in \cF^m R_{k_1}$. By the definition of $\cF^m R_{k_1}$, we know that
\[
f^{(k_1)}+\cdots+f^{(k_p)}\not\in \fa_m(v_1),
\]
which is equivalent to $[F]\neq 0\in R/\fa_m(v_1)$.
\item 
We still need to show that $\mathfrak{B}$ spans $R/\fa_m(v_1)$.
%We only need to show that for any $f\in R_k$ such that $[f]\neq 0\in R/\fa_m(v_1)$, 
%$[f]$ can be written as a linear combination of $[f^{(k)}_i]\in \mathfrak{B}$. 
Suppose on the contrary $\mathfrak{B}$ does not span $R/\fa_m(v_1)$. Then there is some $k\in \bZ_{>0}$ and $f\in R_k-\fa_m(v_1)$ such that $[f]\neq 0\in R/\fa_m(v_1)$ can not
be written as a linear combination of $[f^{(k)}_i]$, i.e. not in the span of $\mathfrak{B}$. We claim that we can choose a maximal $k$ such that this happens. Indeed, this follows from 
the fact that the set
\[
\left\{v_0(g)\; |\; g\in R-\fa_m(v_1)\right\}
\]
is finite (because $v_1(g)<m$ implies $v_0(g)\le c_1^{-1}v_1(g)<m/c_1$). 

So from now on we assume that $k$ has been chosen such that for any $k'>k$ and $g\in R_{k'}$, $[g]$ is in the span of $\mathfrak{B}$.
%Notice that $f\in R_k$ implies $v_1(f)\ge c_1 v_0(f)=c_1 k$. So if
%$k\ge m c_1^{-1}$, then $[f]=0\in R/\fa_m(v_1)$. Equivalently 
%\begin{equation}\label{kgemc1}
%R_k\subseteq \fa_m(v_1) \text{ if } k\ge m c_1^{-1}.
%\end{equation}
Then there are two cases to consider.
\begin{enumerate}
\item
If $f\in R_k\setminus \cF^m R_k$, then since $\{[f^{(k)}_i]_k\}$ is a basis of $R_k/\cF^m R_k$, we can write $f=\sum_{j=1}^{d_k} c_j f^{(k)}_j+h_k$ where $h_k\in \cF^m R_k$. So 
there exists $h\in \fa_m(v_1)$ such that $\bin(h)=h_k$ and
$f=\sum_{j=1}^{d_k} c_j f^{(k)}_j+h$. By the maximality of $k$, we know that $[h_k-h]=[h_k]$ is in the span of $\mathfrak{B}$. Then so is $[f]=\sum_{j=1}^{d_k} c_j [f^{(k)}_j]\in R/\fa_m(v_1)$. This contradicts the condition that $[f]$ is not in the span of $\mathfrak{B}$.
\item
If $f\in \cF^m R_k\subseteq R_k$, then by the definition of $\cF^mR_k$, $f+h\in \fa_m(v_1)$ for some $h\in R$ such that $\bin(f+h)=f$. Since we assumed that
$[f]\neq 0\in R/\fa_m(v_1)$, we have $h\neq 0$ and $k':=v_0(h)>v_0(f)=k$. 
Now we can decompose $h$ into homogeneous components: 
\[
h=h^{(k_1)}+\cdots+h^{(k_p)},
\]
with $h^{(k_j)}\in R_{k_j}$ and $k'=k_1<\cdots<k_p$. Because $k'>k$ and the maximal property of $k$, we know that each $[h^{(k_j)}]$ in the span of $\mathfrak{B}$. 
So we have $[f]=[(f+h)-h]=[-h]$ is in the span of $\mathfrak{B}$. This contradicts our assumption that $[f]$ is not in the span of $\mathfrak{B}$. 
\end{enumerate}
\end{itemize}
\end{proof}
%\[
%v_1(H^0(V, iL))\ge c_1 i \Rightarrow v_1(f)\ge c_1 (i+j) \mbox{ for } f\in H^0(V, \cO_V(iL)\otimes\cI_Z^j).
%\]
%because $f\in H^0(V, iL)$ means $\pi_0^*f$ vanishes on $V$ to the $i$-th order and in particular vanishes on $Z$ to the $i$-th order.
%\[
%v_1(f)\ge c_1 i \Rightarrow f\in \cI_Z^i
%\]
%Our goal is to prove $\hvol(v_1)\ge \hvol(v_0)$. 
The above proposition allows us to derive a general formula for $\hvol(v_1)$ with the help of $\cF$ defined in \eqref{mirafil}.
Indeed, we have:
\begin{eqnarray}\label{wc1acodim}
n!\dim_{\bC}R/\fa_m(v_1)&=&n!\sum_{k=0}^{+\infty}\dim_{\bC}R_k/\cF^{m} R_k\nonumber \\
&=&n!\sum_{k=0}^{\lfloor m/c_1\rfloor }\left(\dim_{\bC}R_k-\dim_{\bC}\cF^{m} R_k\right). 
\end{eqnarray}
For the first part of the sum, we have:
\begin{equation}\label{vtsuma1}
n!\sum_{k=1}^{\lfloor m/c_1\rfloor }\dim_{\bC}R_k=\frac{m^n}{c_1^n}L^{n-1}+O(m^{n-1}).
\end{equation}
For the second part of the sum, we use Lemma \ref{lemlim} in Section \ref{secpflems} to get:
%\begin{lem}\label{lemlim}
\begin{eqnarray}\label{vtsuma2}
\lim_{m\rightarrow +\infty}\frac{n!}{m^{n}}\sum_{k=0}^{\lfloor m/c_1 \rfloor}\dim_{\bC}(\cF^{m}R_k)= n \int_{c_1}^{+\infty}\vol\left(R^{(t)}\right)\frac{ dt}{t^{n+1}}.
\end{eqnarray}
%\end{lem}
Combining \eqref{wc1acodim}, \eqref{vtsuma1} and \eqref{vtsuma2}, we get the first version of volume formula:
\begin{eqnarray}\label{volv1a}
\vol(v_1)&=&\lim_{p\rightarrow +\infty} \frac{n!}{m^n}\dim_{\bC} R/\fa_{m}(v_1)\\
&= &\frac{1}{c_1^n}L^{n-1}- n\int_{c_1}^{+\infty}\vol\left(R^{(t)}\right)\frac{dt}{t^{n+1}}\nonumber.
\end{eqnarray}
We can use integration by parts to get a second version:
\begin{eqnarray}\label{volv1b}
\vol(v_1)&=&\frac{1}{c_1^n}L^{n-1}+\int_{c_1}^{+\infty}\vol\left(R^{(t)}\right)d\left(\frac{1}{t^n}\right)\nonumber\\
&=&\frac{1}{c_1^n}L^{n-1}+\left[\left.\vol\left(R^{(t)}\right)\frac{1}{t^n}\right|^{+\infty}_{c_1}\right]-\int^{+\infty}_{c_1}\frac{d\vol\left(R^{(t)}\right)}{t^n}\nonumber\\
&=&-\int_{c_1}^{+\infty} \frac{d\vol\left(R^{(t)}\right)}{t^n}.
\end{eqnarray}
Motivated by the case of $\bC^*$-invariant valuations (see \eqref{eqvolvs} in Section \ref{secC*inv}), we define a function of two parametric variables $(\lambda, s)\in (0,+\infty)\times [0,1]$:
\begin{eqnarray}\label{eqintfn}
\Phi(\lambda, s)&=&\frac{1}{(\lambda c_1 s+(1-s))^n}L^{n-1}-n \int^{+\infty}_{c_1}
\vol\left(R^{(t)}\right)\frac{\lambda s dt}{(1-s+\lambda s t)^{n+1}}\nonumber\\
&=&\int^{+\infty}_{c_1} \frac{-d\vol\left(R^{(t)}\right)}{((1-s)+\lambda st)^{n}}.
\end{eqnarray}
In Section \ref{seconvex}, we will interpret and re-derive the above formulas \eqref{volv1a}-\eqref{eqintfn} using the theory of Okounkov bodies and coconvex sets.

The usefulness of $\Phi(\lambda, s)$ can be seen from the following lemma.
\begin{lem}\label{lemintfn}
$\Phi(\lambda, s)$ satisfies the following properties:
\begin{enumerate}
\item 
For any $\lambda\in (0,+\infty)$, we have:
\[
\Phi\left(\lambda, 1\right)= \vol(\lambda v_1)=\lambda^{-n}\vol(v_1), \quad \Phi(\lambda, 0)=\vol(v_0)=L^{n-1}.
\]
\item
For fixed $\lambda\in (0,+\infty)$, $\Phi(\lambda, s)$ is continuous and convex with respect to $s\in [0,1]$.
\item 
The derivative of $\Phi(\lambda, s)$ at $s=0$ is equal to:
\begin{equation}\label{dirder0}
\Phi_s(\lambda,0)=n\lambda L^{n-1}\left(\lambda^{-1}-c_1-\frac{1}{L^{n-1}}\int^{+\infty}_{c_1} \vol\left(R^{(t)}\right) dt \right).
\end{equation}
\end{enumerate}
\end{lem}
\begin{proof}
The first item is clear. To see the second item,  first notice that $-d \vol(R^{(t)})$ has a positive density with respect to the Lebesgue measure $dt$ since $\vol(R^{(t)})$ is decreasing with respect to $t\in \bR$. Moreover it has finite total measure equal to $\vol(R^{(x)})=L^{n-1}$ for any $x\le c_1$.  So the claim follows from the fact that the function $s\rightarrow 1/((1-s)+\lambda s t)^n$ is continuous and convex. The third item follows from direct calculation using the formula in \eqref{eqintfn}.
\end{proof}

Roughly speaking, the parameter $\lambda$ is a rescaling parameter, and $s$ is an interpolation parameter. To apply this lemma 
to our problem, we let $\lambda=\frac{r}{A_X(v_1)}=:\lambda_*$ such that
\[
\Phi(\lambda_*, 1)=\frac{A_X(v_1)^n\vol(v_1)}{r^n}=\frac{\hvol(v_1)}{r^n}.
\]
Recall that $\hvol(v_0)=r^n L^{n-1}$. So our problem of showing $\hvol(v_1)\ge \hvol(v_0)$ is equivalent to showing that $\Phi(\lambda_*, 1)\ge \Phi(\lambda_*, 0)$.
By item 2-3 of the above lemma, we just need to show that $\Phi_s(\lambda_*, 0)$ is non-negative. This will be proved in the $\bC^*$-invariant case in the next section.

\subsubsection{A summation formula }\label{secpflems}
\begin{lem}\label{lemlim}
Let $\{\cF^{x}R_i\}_{x\in \bR}$ be a good filtration as in Definition \ref{defn-gdfiltr}. 
For any $\alpha\ge 0\in \bR$ and $\beta>-c_1$, we have the following identity:
\begin{eqnarray}\label{eqlimsum}
&&\lim_{p\rightarrow +\infty}\frac{n!}{p^{n}}\sum_{i=0}^{\lfloor \alpha p/(\beta+c_1)\rfloor}\dim_{\bC}(\cF^{\alpha p-\beta i}R_i)\\
&&\hskip 4cm = n \int_{c_1}^{+\infty}\vol\left(R^{(x)}\right)\frac{\alpha^n dx}{(\beta+x)^{n+1}}\nonumber.
\end{eqnarray}
\end{lem}
\begin{proof}
Define $\phi(y)=\dim_{\bC}(\cF^{\alpha p-\beta  y}R_{\lfloor y\rfloor})$. Then $\phi(y)$ is an increasing function on $[m, m+1)$ for any $m\in \bZ_{\ge 0}$ and $\phi(y)\le \dim_{\bC}R_{\lfloor y\rfloor}\le C y^{n-1}$.
Moreover, because $\cF^x$ is decreasing in $x$, $\phi(y)\ge \dim_{\bC}(\cF^{\alpha p-\beta\lfloor y\rfloor)}R_{\lfloor y\rfloor})$. So we have:
\begin{eqnarray*}
\sum_{i=0}^{\lfloor \alpha p/(\beta+c_1)\rfloor}\dim_{\bC}(\cF^{\alpha p-\beta i}R_i)&\le &\left(\sum_{i=0}^{\lfloor \alpha p/(\beta+c_1) \rfloor-1}\dim_{\bC}(\cF^{\alpha p-\beta i}R_i)\right)+\dim_{\bC}R_{\lfloor\alpha p/(\beta+c_1)\rfloor}\\
&\le & 
\left(\int_0^{\alpha p/(\beta+c_1)}\phi(y)dy\right)+O(p^{n-1})\\
&=&\left(p\int_{c_1}^{+\infty} \phi\left(\frac{\alpha p}{\beta+x}\right)\frac{\alpha dx}{(\beta+x)^2}\right)+O(p^{n-1}).
\end{eqnarray*}
We notice that:
%\begin{eqnarray*}
%&&\frac{\phi(p/(\alpha+x))}{p^{n-1}/(n-1)!}=\frac{\dim_{\bC}(\cF^{c_1px/(\alpha+x)}R_{\lfloor p/(\alpha+x)\rfloor})}{p^{n-1}/(n-1)!}\\
%&\le &\frac{\dim_{\bC}(\cF^{c_1 x\lfloor p/(\alpha+x)\rfloor}R_{\lfloor p/(\alpha+x)\rfloor})}{\lfloor p/(\alpha+x)\rfloor^{n-1}/(n-1)!}\frac{\lfloor p/(\alpha+x)\rfloor^{n-1}}{(p/(\alpha+x))^{n-1}}\frac{1}{(\alpha+x)^{n-1}}\\
%&\le &\frac{\vol(\cF_{\bull}^{c_1x})}{(\alpha+x)^{n-1}}.
%\end{eqnarray*}
% we get:
\begin{eqnarray*}
&&\limsup_{p\rightarrow +\infty}\frac{\phi(\alpha p/(\beta+x))}{p^{n-1}/(n-1)!}=\limsup_{p\rightarrow +\infty}\frac{\dim_{\bC}(\cF^{\alpha px/(\beta+x)}R_{\lfloor \alpha p/(\beta+x)\rfloor})}{p^{n-1}/(n-1)!}\\
&\le &\limsup_{p\rightarrow+\infty}\frac{\dim_{\bC}(\cF^{x\lfloor \alpha p/(\beta+x)\rfloor}R_{\lfloor \alpha p/(\beta+x)\rfloor})}{\lfloor \alpha p/(\beta+x)\rfloor^{n-1}/(n-1)!}\frac{\lfloor \alpha p/(\beta+x)\rfloor^{n-1}}{(\alpha p/(\beta+x))^{n-1}}\frac{\alpha^{n-1}}{(\beta+x)^{n-1}}\\
&=&\vol\left(R^{(x)}\right)\frac{\alpha^{n-1}}{(\beta+x)^{n-1}}.
\end{eqnarray*}
The last identity holds by \cite{LM09} and \cite{BC11} (see also \cite[Theorem 5.3]{BHJ15}). So by Fatou's lemma, we have:
\begin{eqnarray}\label{limsup}
&&\limsup_{p\rightarrow +\infty}\frac{n!}{p^{n}}\sum_{i=0}^{\lfloor \alpha p/(\beta+c_1)\rfloor}\dim_{\bC}(\cF^{\alpha p-\beta i}R_i)\nonumber\\
&\le& n \limsup_{p\rightarrow +\infty}\left(\int_{c_1}^{+\infty}\frac{(n-1)!}{p^{n-1}}\phi\left(\frac{\alpha p}{\beta+x}\right)\frac{\alpha dx}{(\beta+x)^2}+O(p^{-1})\right)\nonumber\\
&\le & n \int_{c_1}^{+\infty}\limsup_{p\rightarrow +\infty}\frac{\phi(\alpha p/(\beta+x))}{p^{n-1}/(n-1)!}\frac{\alpha dx}{(\beta+x)^2}\nonumber\\
&\le & n \int_{c_1}^{+\infty}\vol\left(R^{(x)}\right)\frac{\alpha^n dx}{(\beta+x)^{n+1}}.
\end{eqnarray}
Similarly we can prove the other direction. Define $\psi(y)=\dim_{\bC}(\cF^{\alpha p-\beta y}R_{\lceil y\rceil})$. Then $\psi(y)$ is an increasing function on $(m, m+1]$ for any $m\in \bZ_{\ge 0}$ and satisfies $\psi(y)\le \dim_{\bC}R_{\lceil y\rceil}\le C y^{n-1}$. Moreover, $\psi(y)\le \dim_{\bC}(\cF^{\alpha p-\beta\lceil y\rceil )}R_{\lceil y\rceil})$. So we have:
\begin{eqnarray*}
\sum_{i=0}^{\lfloor \alpha p/(\beta+c_1)\rfloor}\dim_{\bC}(\cF^{\alpha p-\beta i}R_i)&\ge &\left(\sum_{i=1}^{\lceil \alpha p/(\beta+c_1) \rceil}\dim_{\bC}(\cF^{\alpha p-\beta i}R_i)\right)-\dim_{\bC} R_{\lceil \alpha p/(\beta+c_1)\rceil}\\
&\ge & 
\left(\int_0^{\alpha p/(\beta+c_1)}\psi(y)dy\right)+O(p^{n-1})\\
&=&\left(p\int_{c_1}^{+\infty} \psi\left(\frac{\alpha p}{\beta+x}\right)\frac{\alpha dx}{(\beta+x)^2}\right)+O(p^{n-1}).
\end{eqnarray*}
%We estimate:
%\begin{eqnarray*}
%&&\frac{\psi(p/(\alpha+x))}{p^{n-1}/(n-1)!}=\frac{\dim_{\bC}(\cF^{c_1px/(\alpha+x)}R_{\lceil p/(\alpha+x)\rceil})}{p^{n-1}/(n-1)!}\\
%&\ge &\frac{\dim_{\bC}(\cF^{c_1 x\lceil p/(\alpha+x)\rceil}R_{\lceil p/(\alpha+x)\rceil})}{\lceil p/(\alpha+x)\rceil^{n-1}/(n-1)!}\frac{\lceil p/(\alpha+x)\rceil^{n-1}}{(p/(\alpha+x))^{n-1}}\frac{1}{(\alpha+x)^{n-1}}\\
%&\le &\frac{\vol(\cF_{\bull}^{c_1x})}{(\alpha+x)^{n-1}}.
%\end{eqnarray*}
We can then estimate:
\begin{eqnarray*}
&&\liminf_{p\rightarrow +\infty}\frac{\psi(\alpha p/(\beta+x))}{p^{n-1}/(n-1)!}=\liminf_{p\rightarrow +\infty}\frac{\dim_{\bC}(\cF^{\alpha px/(\beta+x)}R_{\lceil \alpha p/(\beta+x)\rceil})}{p^{n-1}/(n-1)!}\\
&\ge &\liminf_{p\rightarrow+\infty}\frac{\dim_{\bC}(\cF^{x\lceil \alpha p/(\beta+x)\rceil}R_{\lceil \alpha p/(\beta+x)\rceil})}{\lceil \alpha p/(\beta+x)\rceil^{n-1}/(n-1)!}\frac{\lceil \alpha p/(\beta+x)\rceil^{n-1}}{(\alpha p/(\beta+x))^{n-1}}\frac{\alpha^{n-1}}{(\beta+x)^{n-1}}\\
&= &\vol\left(R^{(x)}\right)\frac{\alpha^{n-1}}{(\beta+x)^{n-1}}.
\end{eqnarray*}
Similar as before, by using Fatou's lemma, we get the other direction of the inequality:
\begin{eqnarray}\label{liminf}
&&\liminf_{p\rightarrow +\infty}\frac{n!}{p^{n}}\sum_{i=0}^{\lfloor \alpha p/(\beta+c_1)\rfloor}\dim_{\bC}(\cF^{\alpha p-\beta i}R_i)\nonumber\\
&\ge& n \liminf_{p\rightarrow +\infty}\left(\int_{c_1}^{+\infty}\frac{(n-1)!}{p^{n-1}}\psi\left(\frac{\alpha p}{\beta+x}\right)\frac{\alpha dx}{(\beta+x)^2}+O(p^{-1})\right)\nonumber\\
&\ge & n \int_{c_1}^{+\infty}\liminf_{p\rightarrow +\infty}\frac{\psi(\alpha p/(\beta+x))}{p^{n-1}/(n-1)!}\frac{\alpha dx}{(\beta+x)^2}\nonumber\\
&\ge & n \int_{c_1}^{+\infty}\vol\left(R^{(x)}\right)\frac{\alpha^n dx}{(\beta+x)^{n+1}}.
\end{eqnarray}
Combining \eqref{limsup} and \eqref{liminf}, we get the identity \eqref{eqlimsum} we wanted:
\[
\lim_{p\rightarrow+\infty}\frac{n!}{p^n}\sum_{i=0}^{\lfloor \alpha p/(\beta+c_1)\rfloor}\dim_{\bC}(\cF^{\alpha p-\beta i}R_i)=n\int^{+\infty}_{c_1}\vol\left(R^{(x)}\right)\frac{\alpha^n dx}{(\beta+x)^{n+1}}.
\]
\end{proof}

\subsection{Case of $\bC^*$-invariant valuations}\label{secC*inv}

$v_0=\ord_V$ corresponds to the canonical $\bC^*$-action on $R$ given by $a\circ f_k=a^k f_k$ for $f_k\in R_k$ and $a\in \bC^*$.
From now on, $v_1$ is assumed to be $\bC^*$-invariant: $v_1(a\circ f)=v_1(f)$ for any $f\in R$ and $a\in \bC^*$. Let $w:=v_1|_{\bC(V)}$. Then it's easy to see that $v_1$ is a $\bC^*$-invariant extension of $w$ in the following way. First choose an affine neighborhood $U$ of ${\rm center}_V(w)$ and a local generator $s$ of $\cO_X(L)(U)$. Each $f_k\in R_k=H^0(V, kL)$ can be written as $f_k=h\cdot s^k$ on the open set $U$. Define
$w(f_k):=w(h)$. Then for any $f_k\in R_k=H^0(V, kL)$, we have:
\[
v_1(f_k)=w(f_k)+c_1 k
\]
where $c_1=v(V)$ is the constant appearing in Lemma \ref{lemdefc1}. More generally if $f=\sum_k f_k\in R=\bigoplus_k R_k$, we have:
\begin{equation}\label{eqC*ext}
v_1(f)=\min\left\{w(f_k)+c_1 k; f=\sum_k f_k \text{ with } f_k\neq 0 \right\}.
\end{equation}
By \eqref{eqC*ext} it's clear that the valuative ideals $\fa_m(v_1)$ are homogeneous with respect to the $\mathbb{N}$-grading of $R=\bigoplus_{k\in \mathbb{N}} R_k$. 
\begin{lem}\label{lemsat} %\label{cor2fil}
Assume that $v_1$ is $\bC^*$-invariant. $\cF=\{\cF^x\}_{x\in \bR}$ defined in \eqref{mirafil} is equal to:
\begin{equation}\label{eqC*fil}
\cF^x R_k=\fa_x(v_1)\cap R_k=\{ f\in R_k \;|\; v_1(f)\ge x\}.
%&&\{f\in R_k \;|\; \exists g\in R \text{ such that } v_1(g)\ge x \text{ and } \bin(g)=f\}=\cF^x R_k.
\end{equation}
As a consequence, $\cF$ is left-continuous: $\cF^xR_m=\bigcap_{x'<x}\cF^{x'}R_m$ and saturated.
%\end{lem}
%$\cF^x H^0(V, mL)=\bar{\cF}^x H^0(V, mL)$ for any $x\in \bR$ and $m\gg 1$.
%The filtration $\cF$ associated to $\bC^*$-invariant valuation $v_1$ is 
\end{lem}
\begin{proof}
The identity \eqref{eqC*fil} follows easily from the $\bC^*$-invariance of $v_1$ and the left continuity is clear from \eqref{eqC*fil}. We just need to show the saturatedness.
Let $\pi: Y=Bl_oX\rightarrow X$ be the blow up of $o$ on $X$.
For any $f\in H^0(V, mL)$, locally we can write $f=h\cdot s$ where $s$ is a local trivializing section of $L^m$ on an open neighborhood of ${\rm center}_Y(v_1)\subset V$. Then we have
\[
v_1(f)=v_1(\pi^*(h s))=v_1(\pi^*h)+v_1(V)m.
\]
So we have $v_1(f)\ge x$ if and only if $v_1(\pi^*h)=v_1(h)\ge x-c_1 m$. From this we see that $v_1\left(I_{(m,x)}^{\cF}\right)\ge x-c_1 m$. If $f\in H^0(V, mL\cdot I_{(m,x)}^{\cF})$, then locally $f=h\cdot s$ with $h\in I_{(m,x)}^{\cF}$. So we have:
$v_1(f)=v_1(h)+c_1 m\ge x$ and hence $f\in \cF^x H^0(V, mL)$.
%$\fa_{x}(v'_1)=\{h\in \cO_V; v'_1(h)\ge x\}$. Then from the above discussion, we have 
%$I^{\cF}_{(m,x)}=\fa_{x-m}(v'_1)$ and $\cF^x H^0(V, mL)=\cF^x H^0(V, L^m\cdot I^{\cF}_{(m,x)})$, where $I^{\cF}_{(m,x)}$ was defined in 
%\eqref{filtideal}. This by definition means that $\cF^x$ is saturated.
\end{proof}
Remember that we want to prove $\Phi_s(\lambda_*, 0)\ge 0$. 
By \eqref{dirder0} and $\lambda_*=r/A_X(v_1)$, we have 
%\[
%s=s(t)=\frac{t A_X(v_1)}{(1-t)A_X(v_0)+t A_X(v_1)}=:\frac{t A_1}{A_t}.
%\]Then we have:
%\begin{eqnarray*}
%\frac{d}{dt}f_\lambda(0)&=&-n (\lambda c_1 -1)L^{n-1}- n\int^{+\infty}_{\lambda c_1}\vol(\cF^{x/\lambda})dx\\
%&=& n \lambda L^{n-1} \left(\lambda^{-1}- c_1- \frac{1}{L^{n-1}}\int^{+\infty}_{c_1}\vol\left(R^{(t)}\right)dx\right)\\
%\end{eqnarray*}
%When $\lambda=\frac{r}{A_X(v_1)}$, then
\begin{eqnarray}\label{Phis0}
\Phi_s(\lambda_*,0)&=&n\lambda_* L^{n-1}\left(\lambda_*^{-1}-c_1-\frac{1}{L^{n-1}}\int^{+\infty}_{c_1} \vol\left(R^{(t)}\right) dt \right)\\
&=&\frac{n L^{n-1}}{A_X(v_1)}\left(A_X(v_1)-c_1 r- \frac{r}{L^{n-1}}\int^{+\infty}_{c_1}\vol\left(R^{(t)}\right)dt\right).
%&=&n\lambda L^{n-1}\left(A_Y(v_1)-r-\frac{r^n}{(-K)^n}\int^{+\infty}_{c_1}\vol\left(R^{(t)}\right)dx\right).
\end{eqnarray}
%\begin{eqnarray}\label{hvoltwalp}
%\vol(\tw_\alpha)&= &\frac{(A_0\alpha+ A_1)^n}{(1+\alpha)^n}L^{n-1}- n\int_0^{+\infty}\vol(\cF^{c_1(x+1)})\frac{(A_0\alpha+A_1)^n dx}{(\alpha+x+1)^{n+1}}\nonumber\\
%&=&\frac{(A_0+A_1\beta)^n}{(1+\beta)^n}L^{n-1}-n\beta\int_0^{+\infty}\vol(\cF^{c_1 (x+1)})\frac{(A_0+A_1\beta)^ndx}{(1+(x+1)\beta)^{n+1}}\nonumber\\
%&=:&\Phi(\beta).
%\end{eqnarray}
%Notice that $\Phi(0)=r^n L^{n-1}=\hvol(v_0)$. The derivative of $\hvol(v_t)$ at $t=0$ can be easily calculated by using \eqref{hvolvt1} (one can also get the same formula by using
%\eqref{hvolvt2} and integration by parts):
%\begin{eqnarray}\label{Phi'0}
%\Phi'(0)&=&(n r^{n-1}(A_X(v_1)-r)-r^n n (c_1-1))L^{n-1}-n r^n\int^{+\infty}_{c_1}\vol(\cF^x)dx\nonumber\\
%&=&n(A_X(v_1)-c_1 r)r^{n-1}L^{n-1}-n r^n \int^{+\infty}_{c_1}\vol(\cF^x)dx\nonumber\\
%\Phi'(0)&=&-n r^{n-1}(A_X(v_1)-r) \int_{c_1}^{+\infty}d\vol(\cF^x)+n r^n \int_{c_1}^{+\infty} (x-1) d\vol(\cF^{x})\\
%&=&n (A_X(v_1) r^{n-1}- r^n) L^{n-1}+nr^n\left[\left.{(x-1)\vol(\cF^x)}\right|^{+\infty}_{c_1}\right]-n r^n \int_{c_1}^{+\infty} \vol(\cF^x) dx\\
%&=&n (A_X(v_1)-r)r^{n-1}L^{n-1}+n (r-c_1 r)r^{n-1} L^{n-1} -n r^n\int_{c_1}^{+\infty} \vol(\cF^x)dx\\
%&=&n (A_X(v_1)-c_1 r) (-K_V)^{n-1}-n r^n \int_{c_1}^{+\infty} \vol(\cF^x)dx.
%\end{eqnarray}
Under the blow-up $\pi: Y\rightarrow X$, we have
 $\pi^*K_X=K_Y-(r-1)V$ (see \cite[3.1]{Kol13}). So by the property of log discrepancy in \eqref{ldbirat}, we have
\begin{eqnarray*}
A_{X}(v_1)&=&A_{(Y, (1-r)V)}(v_1)=A_{Y}(v_1)+(r-1)v_1(V)\\
&=&A_{Y}(v_1)+(r-1)c_1. \quad (\text{by } \eqref{infv1})
\end{eqnarray*}
This gives $A_X(v_1)-c_1 r=A_Y(v_1)-c_1$. So, using also $L=r^{-1}K_V^{-1}$, we get another expression for $\Phi_s(\lambda_*, 0)$ in \eqref{Phis0}:
\begin{eqnarray}\label{dPhi}
\Phi_s(\lambda_*, 0)&=&\frac{n L^{n-1}}{A_X(v_1)} \left(A_Y(v_1)-c_1-\frac{r^n}{(-K_V)^{n-1}}\int^{+\infty}_{c_1}\vol\left(\cF R^{(t)}\right)dt\right)\\
&=&\frac{n L^{n-1}}{A_X(v_1)}\left(A_Y(v_1)-c_1\left(1+\frac{r^n}{(-K_V)^{n-1}}\int^{+\infty}_{1}\vol(\cF R^{(c_1 t)})dt\right)\right)\nonumber.
\end{eqnarray}
Next we bring in Fujita's criterion for Ding-semistability in \cite{Fuj15b}. For simplicity, we denote $\tcF^x=\cF^{c_1 x}$. By \eqref{cFlbd} and \eqref{defc1} we can choose $e_{+}$ with 
\[
e_+\ge \frac{\sup_{\fm} \frac{v_1}{v_0}}{\inf_{\fm}\frac{v_1}{v_0}}\ge \frac{e_{\max}(\cF)}{c_1}=e_{\max}(\tcF)
\] 
and $e_-=1$. Notice that by relation \eqref{cFlbd}, $e_{\min}(\tcF)=e_{\min}(\cF)/c_1\ge 1$. Denote by the $\cI_V$ the ideal sheaf of $V$ as a subvariety of $Y$. Notice that $Y$ is nothing but the total space of the line bundle $L^{-1}\rightarrow V$ so that we have a canonical projection denoted by $\rho: Y\rightarrow V$. Now similar to that in \cite{Fuj15b}, we define:
\begin{eqnarray}\label{eq-tcI}
&&\cI^{\tcF}_{(m,x)}={\rm Image}\left(\tcF^xR_m\otimes L^{-m}\rightarrow \cO_V\right)\nonumber\\
&&\tcI_m=\rho^*\cI_{(m, m e_+)}^{\tcF}+\rho^*\cI_{(m, me_+-1)}^{\tcF} \cdot \cI_V+\cdots\nonumber\\
&&\hskip 4cm \cdots+\rho^*\cI_{(m, me_-+1)}^{\tcF}\cdot \cI_V^{m(e_+-e_-)-1}+\cI_V^{m(e_+-e_-)}.
\end{eqnarray}

Then $\cI^{\tcF}_{(m,x)}$ (resp. $\tcI_m$) is an ideal sheaf on $V$ (resp. $Y$). Moreover, $\tcI_{\bull}:=\{\tcI_m\}_{m}$ is a graded family of coherent ideal sheaves by \cite[Proposition 4.3]{Fuj15b}. By Lemma \ref{lemdefc1}, we have $v_1(\cI_V)=c_1$. On the other hand, using Lemma \ref{lemsat}, we get:
\begin{equation}\label{v1cI}
v_1(\cI^{\tcF}_{(m,x)})=v_1\left({\rm Image}(\cF^{c_1x}R_m\otimes L^{-m}\rightarrow \cO_V)\right)\ge c_1x-c_1m.
\end{equation}
Combining \eqref{v1cI} and the inclusion $\cI_V^{m(e_+-e_-)}\subset \tcI_m$, we get $v_1(\tcI_m)=v_1(\cI_V^{m(e_+-e_-)})=m(e_+-e_-)v_1(\cI_V)=m c_1 (e_+-1)$ and hence
\[
v_1(\tcI_\bull)=\inf_{m} \frac{v_1(\tilde{\cI}_m)}{m}=c_1 (e_+-e_-)=c_1(e_+-1).
\]
We define $d_\infty$ as in Fujita's Theorem \ref{Fujthm}:
\[
d_{\infty}=1-r(e_+-e_-)+\frac{r^{n}}{(-K_V)^{n-1}}\int_1^{e_+} \vol\left(\overline{\tcF} R^{(t)}\right)dt.
\]
Then we get:
\begin{eqnarray*}
v_1(\tcI_\bull^{r}\cdot \cI_V^{d_\infty})&=&r c_1(e_+-e_-)+c_1  d_\infty=c_1\left(1+\frac{r^n}{(-K_V)^{n-1}}\int_1^{e_+} \vol\left(\overline{\tcF} R^{(t)}\right)dt\right)\\
&=&c_1 \left(1+\frac{r^n}{(-K_V)^{n-1}}\int_1^{+\infty} \vol\left(\overline{\cF} R^{(c_1 t)}\right)dt\right)
\end{eqnarray*}
Comparing the above expression with that in \eqref{dPhi} we %denote $C_*=n\lambda_* (-K_V)^{n-1}$ to 
get the estimate:
\begin{eqnarray*}
\Phi_s(\lambda_*, 0)&=&\frac{nL^{n-1}}{A_X(v_1)}\left(A_Y(v_1)-v_1(\tcI_\bull^{r}\cdot \cI_V^{d_\infty})\right)\\
&&\hskip 0.3cm+\frac{nL^{n-1}}{A_X(v_1)}\left(\frac{r^n}{(-K_V)^{n-1}}\int^{+\infty}_1 \left(\vol(\overline{\cF}R^{(c_1 t)})-\vol(\cF R^{(c_1 t)})) \right)dt\right)\\
&\ge& \frac{nL^{n-1}}{A_X(v_1)}\left(A_Y(v_1)-v_1(\tcI_\bull^{r}\cdot \cI_V^{d_\infty})\right).
\end{eqnarray*} 
Note that the last inequality is actually an equality because $\cF$ is saturated by Lemma \ref{lemsat}. 
So to show $\Phi_s(\lambda_*, 0)\ge 0$, we just need to show 
\begin{equation}\label{lcineq}
A_Y(v_1)-v_1(\tcI_\bull^{r}\cdot \cI_V^{d_\infty})\ge 0.
\end{equation}
We prove this by deriving the following regularity from Fujita's result.
\begin{lem}
Assume that $(V, K_V^{-1})$ is K-semistable. Then 
$(Y, \tcI_{\bull}^{r}\cdot \cI_V^{d_\infty})$ is sub log canonical. 
\end{lem}
\begin{proof}
By Fujita's Theorem \ref{Fujthm} (\cite[Theorem 4.9]{Fuj15b}) and Remark \ref{rem-D2K}, $(V\times \bC, \cI_\bull^r\cdot (t)^{d_\infty})$ is sub log canonical where $\cI_\bullet=\{\cI_m\}$ is the ideal sheaf on $V\times \bC$ defined in \eqref{eq-cI}. Informally we get $\cI_\bullet$ on $V\times\bC$ if in \eqref{eq-tcI} $\rho$ is replaced by the canonical fibration $\rho_0: V\times \bC\rightarrow V$. By choosing an affine cover of $\{U_i\}$ of $V$, we have
$(\tcI^r_m\cdot \cI^{d_\infty}_V)(\rho^{-1}(U_i))\cong (\cI^r_m\cdot (t)^{d_\infty})(\rho_0^{-1}(U_i))$ for any $m\ge 0$. Since sub log canonicity can be tested by testing on all open sets of an affine cover, we get
the conclusion.
\end{proof}
By the above Lemma, we get \eqref{lcineq} using the same approximation argument as in \cite[Proof of Theorem 4.1]{BFFU13}. 
Because the space of divisorial valuations is dense in $\Val_{X,o}$ we want to use some semicontinuity properties to get the inequality \eqref{lcineq} that already holds for divisorial valuations. More precisely, the sub log canonicality means that, for any divisorial valuation $\ord_F$ over $Y$, we have:
\[
A_Y(F)-r\cdot \ord_F(\tcI_\bullet)- d_\infty\cdot \ord_F(\cI_V)\ge 0.
\]
Consider the function $\phi: \Val_{X,o}\rightarrow \mathbb{R}\cup \{+\infty\}$ defined by:
\[
\phi(v)=A_Y(v)-r v(\tcI_\bullet)-d_\infty\cdot v(\cI_V)=A_X(v)-r v(\tcI_\bullet)- (r-1+d_\infty) \cdot v(\cI_V).
\]
If we endow $\Val_{X,o}$ with the topology of pointwise convergence, then
\begin{itemize}
\item $v\mapsto v(\cI_V)$ is continuous by \cite[Proposition 2.4]{BFFU13};
\item $v\mapsto v(\tcI_\bullet)$ is upper semicontinuous by \cite[Proposition 2.5]{BFFU13}; 
\item $v\mapsto A_Y(v)$ is lower semicontinuous by \cite[Theorem 3.1]{BFFU13}. 
\end{itemize}
Combining these properties we know that $\phi$ is a lower semi-continuous function on $\Val_{X,o}$. Furthermore it was shown in \cite{BFFU13} that $\phi$ is continuous on small faces of $\Val_{X,o}$ and satisfies $\phi\ge \phi \circ r_\tau$ for any $\tau$ a good resolution dominating the blow up $\tau: Y\rightarrow X$ (we refer to \cite{BFFU13} for the definition of the contraction $r_\tau$ and the notions of {\it small faces} and {\it good resolutions}). This allows us %to follow \cite[Proof of Theorem 4.1]{BFFU13} 
to show that $\phi\ge 0$ on the space of divisorial valuations implies $\phi(v)\ge 0$ for any valuation $v\in \Val_{X,o}$.

\subsubsection{Appendix: Interpolations in the $\bC^*$-invariant case} \label{appC*}

Assume $v_1$ is $\bC^*$-invariant. Let $w=v_1|_{\bC(V)}$ be its restriction as considered at the beginning of Section \ref{secC*inv}. We can connect the two valuations $v_0$ and $v_1$ by a family of $\bC^*$-invariant valuations $v_s$ for $s\in [0,1]$. $v_s$ is defined as the following $\bC^*$-invariant extension of $s w$:
for any $f_k\in R_k$ define:
\begin{equation}\label{eq-vsfk}
v_s(f_k)=s w(f_k)+(s c_1+(1-s))k=(1-s) v_1(f_k)+s c_1 v_0(f_k).
\end{equation}
The second identity follows from \eqref{eqC*ext}. For $f\in R$, define
\[
v_s(f)=\min\left\{v_s(f_k); f=\sum_k f_k \text{ with } f_k\neq 0 \right\}.
\]
%For our later purpose, we introduce the following condition:
%\begin{defn}
%Con
%\end{defn}
%Similar to the setting of quasi-monomial valuations, we define the log discrepancy of the weight function $v_t$ as 
%$A_X(v_t)=(1-t)A_X(v_0)+t A_X(v_1)=:A_t$, and the normalized volume $\hvol(v_t)=A_X(v_t)^n \vol(v_t)$. If we introduce the normalizing 
%parameter:
%\[
%s=s(t)=\frac{t A_X(v_1)}{(1-t)A_X(v_0)+t A_X(v_1)}=:\frac{t A_1}{A_t},
%\]
%then
%\begin{lem}
%If $\vol(v_t)$ is convex with respect to $t\in [0,1]$, then $\hvol(v_t)$ is convex with respect to $s$. 
%\end{lem}
%\begin{proof}
%Define a new family of weight functions 
%\[
%w_t:=\frac{v_t}{A_t}=\frac{t A_X(v_0)}{A_t}\frac{v_0}{A_X(v_0)}+\frac{t A_X(v_1)}{A_t}\frac{v_1}{A_X(v_1)}=(1-s) w_0+s w_1.
%\] 
%Then we have $A(w_t)\equiv 1$ and $\hvol(v_t)=\hvol(w_t)=\vol(w_t)$. So the corollary follows from Lemma \ref{lemconvex}.
%\end{proof}
Because $c_1>0$, it's easy to see that ${\rm center}_X(v_s)=o$, i.e. $v_s\in \Val_{X,o}$. From \eqref{eq-vsfk} $v_s$ is indeed an interpolation between $v_0=\ord_V$ and the given valuation $v_1$. The valuation $v_s$ is $\bC^*$-invariant and its valuative ideals $\fa_m(v_s)$ are homogeneous under the natural $\mathbb{N}$-grading of $R=\bigoplus_k R_k$. 
%\begin{lem}\label{lemhom}
%Assume that $v$ is a $\bC^*$-invariant valuation in $\Val_{X,o}$. Then for any $\lambda\in \bR$, the valuative ideal $\fa_\lambda(v)$ is homogeneous:
%\begin{equation}\label{eqhom}
%\fa_\lambda(v)=\bigoplus_{k=0}^{+\infty}\left( \fa_\lambda(v)\cap R_k\right).
%\end{equation}
%\end{lem}
%The above argument clearly implies Lemma \ref{lemsat}. 
%\begin{cor}\label{lemsat}
%Assume that $v_1$ is $\bC^*$-invariant. Then the following two filtrations are the same:
%\begin{eqnarray*}
%&&\{ f\in R_k \;|\; v_1(f)\ge x\};\\
%&&\{f\in R_k \;|\; \exists g\in R \text{ such that } v_1(g)\ge x \text{ and } \bin(g)=f\}=\cF^x R_k.
%\end{eqnarray*}
%In other words, $\cF^x R_k=\fa_m(v_1)\cap R_k$. %In particular, $\cF^xR_k$ is left-continuous with respect to $x$.
%\end{cor}
%It's clear that $\cF:=\cF_{v_1}$ is a decreasing, left-continuous, multiplicative filtration of the graded $\bC$-algebra $R$. 
%However, in general, we don't expect $\cF$ to be saturated. However, we 
%have
\begin{prop}
Assume that $v_1$ is $\bC^*$-invariant. 
The volume of $v_s$ defined above is equal to $\Phi(1,s)$ defined in \eqref{eqintfn}. In other words, we
have the formula:
\begin{eqnarray}\label{eqvolvs}
\vol(v_s)&=&\int^{+\infty}_{c_1}\frac{-d\vol\left(R^{(t)}\right)}{((1-s)+st)^n}.
\end{eqnarray}
\end{prop}
\begin{proof}
Because $\fa_m(v_s)$ is homogeneous and using \eqref{eq-vsfk} and Lemma \ref{lemsat}, we have:
\begin{eqnarray*}
\dim_{\bC} R/\fa_m(v_s)&=&\sum_{i=0}^{+\infty} \dim_{\bC} R_i/\left(\fa_m(v_s)\cap R_i\right)\\
&=&\sum_{i=0}^{+\infty}\left( \dim_{\bC} R_i - \dim_{\bC}\left(\fa_{m-(1-s)i)/s}(v_1)\cap R_i\right)\right) \\
&=&\sum_{i=0}^{+\infty} \left(\dim_{\bC} R_i-\dim_{\bC}\left(\cF^{\frac{m-(1-s)i}{s}}R_i\right)\right).
\end{eqnarray*}
Because $\cF^xR_i=R_i$ if $x\le c_1 i$ by the definition of $c_1$ (see the discussion before Lemma \ref{lemdefc1}), we deduce that:
\[
\cF^{\frac{m-(1-s)i}{s}}R_i=R_i \mbox{ if } \frac{m-(1-s) i}{s}\le c_1i \text{ or equivalently } 
i\ge \frac{m}{c_1 s+(1-s)}.
\]
So we get the following identity:
\begin{eqnarray}\label{codimvs}
n!\dim_{\bC}R/\fa_m(v_s)%&=&n!\sum_{i=0}^{+\infty}\dim_{\bC}R_i/\cF^{p}_{v_t}R_i\\
&=&n!\sum_{i=0}^{\lfloor m/(c_1 s+(1-s))\rfloor }\dim_{\bC}R_i-\dim_{\bC}\cF^{\frac{m-(1-s)i}{s}} R_i. 
\end{eqnarray}
For the first part of the sum, we have:
\begin{equation}\label{vtsum1}
n!\sum_{i=1}^{\lfloor m/(c_1 s+(1-s))\rfloor }\dim_{\bC}R_i=\frac{m^n}{(c_1 s+(1-s))^n}L^{n-1}+O(p^{n-1}).
\end{equation}
For the second part of the sum, we use Lemma \ref{lemlim} to get:
%\begin{lem}\label{lemlim}
\begin{eqnarray}\label{vtsum2}
&&\lim_{m\rightarrow +\infty}\frac{n!}{m^{n}}\sum_{i=0}^{\lfloor m/(c_1 s+(1-s))\rfloor}\dim_{\bC}(\cF^{\frac{m-(1-s)i}{s}}R_i)\\
&&\hskip 4cm = n \int_{c_1}^{+\infty}\vol\left(R^{(t)}\right)\frac{s dt}{((1-s)+s t)^{n+1}}\nonumber.
\end{eqnarray}
%\end{lem}
Combining \eqref{codimvs}, \eqref{vtsum1} and \eqref{vtsum2}, we have:
\begin{eqnarray}\label{eqvolvt1}
\vol(v_s)&=&\lim_{m\rightarrow +\infty} \frac{n!}{m^n}\dim_{\bC} R/\fa_m(v_s)\\
&= &\frac{1}{(c_1 s+(1-s))^n}L^{n-1}- n\int_{c_1}^{+\infty}\vol\left(R^{(t)}\right)\frac{s dt}{((1-s)+st)^{n+1}}\nonumber.
\end{eqnarray}
We can use integration by parts to further simplify the formula:
\begin{eqnarray}\label{eqvolvt2}
\vol(v_s)&=&\frac{1}{(c_1 s+(1-s))^n}L^{n-1}+\int_{c_1}^{+\infty}\vol\left(R^{(t)}\right)d\left(\frac{1}{((1-s)+st)^n}\right)\nonumber\\
&=&\frac{1}{(c_1 s+(1-s))^n}L^{n-1}+\nonumber\\
&&\hskip 0.7cm\left[\left.\vol\left(R^{(t)}\right)\frac{1}{(s t+(1-s))^n}\right|^{+\infty}_{t=c_1}\right] -\int^{+\infty}_{c_1}\frac{d\vol(R^{(t)})}{(st+(1-s))^n}\nonumber\\
&=&-\int_{c_1}^{+\infty} \frac{d\vol\left(R^{(t)}\right)}{(st+(1-s))^n}.
\end{eqnarray}\end{proof}
\section{Uniqueness among $\bC^*$-invariant valuations}

In this section, we prove Theorem \ref{thm-unique}. We first prove the result for divisorial valuations. Suppose $E_2$ is a prime divisor centered at $o$ such that $\hvol(\ord_{E_2})=\hvol(\ord_V)$. 
We want to show that $E_2=V$. The main observation is that the interpolation function $\Phi(\lambda_*, s)$ in \eqref{eqintfn} must be a constant function independent of $s\in [0,1]$. Indeed in this case  
$\Phi(s):=\Phi(\lambda_*, s)$ is a convex function satisfying $\Phi(0)=\Phi(1)=\min_{s\in [0,1]}\Phi(s)$. So $\Phi(s)\equiv \Phi(0)=\Phi(1)$. Recall the expression in \eqref{eqintfn}:
\begin{eqnarray}\label{eq-Phis}
\Phi(s)&=&\frac{1}{(\lambda_* c_1 s+(1-s))^n}L^{n-1}-n \int^{c_2}_{c_1}
\vol\left(R^{(t)}\right)\frac{\lambda_* s dt}{(1-s+\lambda_* s t)^{n+1}}.
%&=&\int^{+\infty}_{c_1} \frac{-d\vol\left(R^{(t)}\right)}{((1-s)+\lambda_* st)^{n}}.
\end{eqnarray}
Here we have changed the improper integral to a finite integral by choosing $c_2\gg 1$ such that $\vol\left(R^{(t)}\right)=0$ for $t\ge c_2$.
Because $\vol\left(R^{(t)}\right)$ is piecewise continuous on $[c_1, c_2]$, we easily verify that $\Phi(s)$ in \eqref{eq-Phis} is a smooth function of $s\in [0,1]$. We can calculate its second order derivative:
\begin{eqnarray}\label{eq-ddPhi}
\Phi''(s)&=&\frac{(n+1)n(\lambda_* c_1-1)^2}{(\lambda_* c_1 s+(1-s))^{n+2}}L^{n-1}\nonumber\\
&&\hskip 10mm -n(n+1)\int^{c_2}_{c_1}\vol\left(R^{(t)}\right)\frac{\lambda_* (\lambda_* t-1)(-2+n s(\lambda_* t-1)) dt}{(1-s+\lambda_* st)^{n+3}}\nonumber\\
&=&-n(n+1)\int_{c_1}^{c_2}(\lambda_* t-1)^2\frac{d \vol\left(R^{(t)}\right)}{(\lambda_* st+(1-s))^{n+2}}.
\end{eqnarray}
The second identity follows from integration by parts. 
By Proposition \ref{BHJvol}, $\nu=-\frac{1}{L^{n-1}}d\vol\left(R^{(t)}\right)$, supported on the interval $[\lambda_{\min}, \lambda_{\max}]$, is absolutely continuous on $[\lambda_{\min}, \lambda_{\max})$ with respect to the Lebesgue measure and possibly has a Dirac mass at $\lambda_{\max}$.
We also know that $\Phi''(s)\equiv 0$ for $s\in [0,1]$ since $\Phi(s)$ is a constant. Using the last expression in \eqref{eq-ddPhi}, we see that $\lambda_{\max}=\lambda_{\min}$ and
$\nu=\delta_{\lambda_{\max}}$. Otherwise, $\Phi''(s)$ is not identically equal to zero on the nonempty open interval 
$(\lambda_{\min}, \lambda_{\max})$.
%Because $\vol\left(R^{(t)}\right)$ is continuous with respect to $t$, we see that $\Phi(s)$ is differentiable. Its second derivative is equal to:
%\[
%\Phi''(s)=-n(n+1)\int^{+\infty}_{c_1}(\lambda_* t-1)^2\frac{d \vol\left(R^{(t)}\right)}{(\lambda_* st+(1-s))^{n+1}}
%\]
%By looking at its second order derivative, we see that the measure $-d\vol(R^{(t)})$ must be a Dirac measure by the second expression in \eqref{eqintfn}. 

We will show that $\nu$ being a Dirac measure indeed implies $\ord_{E_2}=\ord_{V}$. The latter statement can be thought of a counterpart of the result in \cite[Theorem 6.8]{BHJ15}. Indeed, the argument given below is motivated the proof in \cite[Lemma 5.13]{BHJ15}.

%We will use the same notations as before.
Let $\ord_{E_2}^{(0)}$ denote the restriction of $\ord_{E_2}$ under the inclusion $\bC(V)\subset \bC(X)$. 
%$v_{E}$ induces a valuation $v^{(0)}_E$ on $\bC(X)$ and a valuation $v^{(0)}_{*E}$ on $\bC(X_*)$ in the following way. 
As before, $v_1:=\ord_{E_2}$ is a $\bC^*$-invariant extension of $\ord_{E_2}^{(0)}$ in the following way.
Choose an affine open neighborhood $U$ of ${\rm center}_V(\ord_{E_2}^{(0)})$ and a trivializing section $s\in \cO_V(L)(U)$, then any $f\in H^0(V, kL)$ can be locally written as $f=h \cdot s^k$. We define $\ord_{E_2}^{(0)}(f)=\ord_{E_2}(h)$ such that 
%It's clear that $v^{(0)}_{E}(f)\ge 0$ for any $f\in H^0(V, mL)$ and the equality holds if and only if $f$ does not vanish on the center $F$ of $v_{E}$.
%For any $f\in R_k$, we can write $f=h_j \cdot s_j^{k}$ %and $\pi^*f=\pi^*(h)\cdot \pi^*s$.  where $\pi: L\rightarrow V$ and $s$ is a meromorphic section of $L^k$ which is at the center of $E_2$ and is considered as rational function on the total space $L$.
%over $U_j$ that is an affine open set meeting the center of $\ord_{E_2}$.
%Then we have:
\[
\ord_{E_2}(f)=\ord_{E_2}(h_j)+k\cdot \ord_{E_2}(s_j)=\ord_{E_2}(h_j)+k c_1=\ord_{E_2}^{(0)}(f)+k c_1.
\]
%If we defined $\ord^{(0)}_{E_2}$ be the restriction of $\ord_{E_2}$ on $\bC(V)$, then 

In particular we have,
\[
\cF^{kx}R_k=\cF^{k(x-c_1)}_{\ord_{E_2}^{(0)}}R_k=\{f\in R_k; \ord_{E_2}^{(0)}(f)\ge k(x-c_1) \}.
\]

Let $Z$ be the center of $\ord^{(0)}_{E_2}$ on $V$. Then by general Izumi's theorem (see \cite{HS01}), we have:
\begin{equation}\label{eq-gIzumi}
\ord^{(0)}_{E_2}(h)\le C_2 w(h), \text{ for any } h\in \cO_{V,Z},
\end{equation}
where $w$ is a Rees valuation of $Z$. Let $V'\rightarrow V$ be the normalized blow up of $Z$ inside $V$ and $G=\mu^{-1}Z$. The Rees valuations are given up to scaling by vanishing order along irreducible components of $G$ (see \cite[Example 9.6.3]{Laz04}, \cite[Definition 1.9]{BHJ15}). By Izumi's theorem, the Rees valuations of $Z$ are comparable to each other (\cite{Ree89, HS01}). %. By the proof of \cite[Lemma 5.13]{BHJ15}, 
So for any $x> c_1$ and $0<\delta\ll 1$, we have:
\[
\cF^{k x}R_k=\cF^{k(x-c_1)}_{\ord^{(0)}_{E_2}}R_k\subseteq H^0(Bl_ZV, k(L-\delta G)).
\]
Since $\mu^*L-\delta G$ is ample on $V'$ for $0<\delta\ll 1$, this implies:
\[
\vol(\cF R^{(x)})\le (\mu^*L-\delta G)^{n-1}< L^{n-1}
\]
for any $x>c_1$. So $\lambda_{\min}=c_1=\inf_{\fm}\frac{\ord_{E_2}}{\ord_V}$ by \eqref{lambdamin} in Lemma \ref{BHJvol}. For any $x<\lambda_{\max}$, $\vol(R^{(x)})>0$ by \cite[Lemma 1.6]{BC11} (see also \cite[Theorem 5.3]{BHJ15}). It's easy to see that
$\lambda_{\max}=\sup_{\fm}\frac{\ord_{E_2}}{\ord_V}=:c_2$ by \eqref{lambdamax} and our definition of filtration $\cF R$ (in the $\bC^*$-invariant case).
So if $d \vol(R^{(x)})$ is the Dirac measure, 
then $c_1=c_2$ and hence $\ord_{E_2}=\ord_V$.

It's clear that the above proof works well for a $\bC^*$-invariant valuation $v_1$ as long as the general Izumi's inequality as in \eqref{eq-gIzumi} holds for $w:=v_1|_{\bC(V)}$. In particular, this holds for any $\bC^*$-invariant quasi-monomial valuation $v_1$. Indeed, if $v_1$ is a $\bC^*$-invariant quasi-monomial valuation. Then $w=v_1|_{\bC(V)}$ is also a quasi-monomial valuation by Abhyankar-Zariski inequality (see \cite[(1.3)]{BHJ15}). Denote $Z={\rm center}_V(w)$. We can find a log smooth model $(W, \sum_i E_i)\rightarrow V$ such ${\rm center}_V E_i=Z$ and $w$ is a monomial valuation on the model $(W, \sum_i E_i)$. Since the general Izumi inequality holds for $\ord_{E_i}$, it also holds for $w$ (cf. \cite{BFJ12}).

\begin{rem}
We expect that the generalized Izumi inequality holds for any valuation $v^{(0)}$ of $\bC(V)$ with $A_V(v^{(0)})<+\infty$, i.e. there exists $C=C(v^{(0)})$ such that
\[
v^{(0}(h)\le C w(h) \text{ for any } h\in \cO_{V,Z}
\] 
where $Z={\rm center}_V (v^{(0)})$ and $w$ is a Rees valuation of $Z$.
If $Z$ is a closed point, this is indeed true (see \cite{Li15a}). 
The general case can probably be reduced to the case of closed point (cf. \cite[Proof of Proposition 5.10]{JM10}).
\end{rem}

%\begin{lem}\label{lemconv}
%For any $\bC^*$-invariant valuation $v_1$ over $(X=C(V,L), o)$, the volume function $\vol(v_t)$ is continuous and convex with respect to $t\in [0,1]$.
%\end{lem}
%\begin{proof}
%Consider the family of functions of $t$ with parameter $x\in [c_1, e_{\max}]$. $F_x(t)=1/(tx+(1-t))^n$
%Using the formula \eqref{eqvolvt2}, the statement follows from the convexity and uniform continuity of $F_x(t)$ with respect to $x\in [c_1, e_{\max}]$.
%\end{proof}
%\begin{rem}
%The convexity of $\vol(v_t)$ should hold in much more generality. We plan to study this elsewhere.
%\end{rem}

\section{Equivariant volume minimization implies K-semistability}\label{secvol2semi}

In this section, we will prove the other direction of implication. 

\begin{thm}\label{mv2ks}
Let $V$ be a $\bQ$-Fano variety and $L_*=r_*^{-1}K_V^{-1}$ an ample Cartier divisor for a fixed $r_*\in \bQ_{>0}$. Assume that $\hvol$ is minimized among $\bC^*$-invariant valuations at $\ord_V$ on $X_*:=C(V, L_*)$. Then $V$ is K-semistable.
\end{thm}

As explained in Section \ref{sec-results}, this combined with Theorem \ref{ks2mv} will complete the proof of Theorem \ref{main}. Indeed, Theorem \ref{ks2mv} shows $(1)\Rightarrow (3)$ and $(3)\Rightarrow (2)$ trivially.

%\subsection{Proof of volume minimizing $\Rightarrow$ K-semistable}

The rest of this section is devoted to the proof of this theorem. By \cite{LX14}, to prove the K-semistability, we only need to consider special test configurations. As we will see, this reduction simplifies the calculations and arguments in a significant way. So from now on, we assume that there is a special test configuration $(\cV, \cL)\rightarrow \bC$ such that $\cL= -r^{-1} K_{\cV/\bC}$ is a relatively very ample line bundle. Notice that $r$ is in general not the same as $r_*$. However, for the sake of testing K-semistability which is of asymptotic nature (see \cite{Tia97, Don02, LX14}), we can choose $0<r=\frac{1}{N}$ for $N\gg 1$ sufficiently divisible such that $L$ is an integral multiple of $L_*$ . For later convenience we will define this constant:
\begin{equation}
\sigma:=\frac{r^{-1}}{r_*^{-1}}=\frac{r_*}{r} \in \bZ_{>0}, \text{ such that } \quad L=\sigma L_*.
\end{equation}
The central fibre of $\cV\rightarrow \bC$ is a $\bQ$-Fano variety, denoted by $\cV_0$ or $E$. Because the polarized family $(\cV, \cL)\rightarrow \bC$ is flat, we can choose $r^{-1}=N\gg 1$ sufficiently divisible such that $\cL=-r^{-1}K_{\cV/\bC}=-N K_{\cV/\bC}$ is a relatively very ample line bundle over $\bC$ and $H^0(V, mL)=H^0(\cV_0, m\cL_0)$ for every $m\in \bZ_{\ge0}$. By taking affine cones, we then get a flat family $\cX\rightarrow \bC$ with general fibre $X=C(V, L)$ and the central fibre $\cX_0=C(\cV_0, \cL_0)$ (see \cite[Aside on page 98]{Kol13}). 
Notice that we can write:
\[
\def\arraystretch{2}
\begin{array}{llll}
X_*&= &{\rm Spec}\bigoplus_{i=0}^{+\infty} H^0(V, iL_*)&=:{\rm Spec}\; R_*; \\
X&= &{\rm Spec}\bigoplus_{i=0}^{+\infty} H^0(V, iL)&=:{\rm Spec}\;  R; \\
\cX_0&=&{\rm Spec}\bigoplus_{i=0}^{+\infty}H^0(\cV_0, i\cL_0)&=:{\rm Spec}\; S;\\
\cX&=&{\rm Spec}\bigoplus_{i=0}^{+\infty}H^0(\cV, i\cL)&=:{\rm Spec}\;\cR.
\end{array}
\]
Our plan is to define families of valuations $\wt_\beta$ on $\cX_0$, $w_\beta$ on $X$ and $w_{*\beta}$ on $X_*$, and then 
study the relation between their normalized volumes. 

\subsection{Valuation $\wt_\beta$ on $\cX_0$}\label{sec-wtbeta}

By the definition of special test configuration, there is an equivariant $T_1=\bC^*$ action on $(\cV, \cL)\rightarrow \bC$ which fixes the central fibre $\cV_0$. This naturally induces an equivariant $T_1=:e^{\bC \eta}$ action on $\cX\rightarrow \bC$ which fixes $\cX_0$. There is also another natural $T_0:=e^{\bC\xi_0}\cong \bC^*$-action which fixes each point on the base $\bC$ and rescales the fibre. These two actions commute and generate a $\bT=(\bC^*)^2$ action on $\cX_0$. 
For any $\tau=(\tau_0, \tau_1)\in \bZ^2$ and $t=(t_0, t_1)\in (\bC^*)^2$, we will denote $t^{\tau}=t_0^{\tau_0}t_1^{\tau_1}$. Consider the weight decomposition of $S$ under the $\bT$ action:
\[
S=\bigoplus_{\tau\in \Gamma} S_{\tau},
\]
where
%we denoted:
%\begin{itemize}
%\item 
$S_{\tau}=\{f\in S; t\circ f=t^\tau\cdot f \text{ for all } t\in (\bC^*)^2\}$ for any $\tau\in \bZ^2$ and 
%\item 
$\Gamma=\{\tau \in \mathbb{Z}^2; S^\tau\neq \{0\} \}\subset \mathbb{Z}^2$. 
%\end{itemize}

Denote by $\mathfrak{t}_{\bC}$ the Lie algebra of $\bT$. The two generator $\{\xi_0, \eta\}$ allows us to identity $\mathfrak{t}_{\bC}=N\otimes_{\bZ} \mathbb{C}$ for a lattice $N\cong\mathbb{Z}^2$. 
%We also define $\mathfrak{t}_{\bQ}=N\otimes_{\bZ}\bQ$ and $\ft_{\bR}=N\otimes_{\bZ}\bR$. Following standard notations in toric geometry, we also define $M={\rm Hom}_{\bZ}(N, \bZ)$ and $M_{\bR}=M\otimes\bR$. Then there is a natural pairing on $N_{\bR}\times M_{\bR}$ and a canonical embedding $\Gamma\hookrightarrow M$.
Let $\mathfrak{t}_{\bR}=N\otimes_{\bZ}\bR$. Any element $\xi\in \mathfrak{t}_{\bR}$ is a holomorphic vector field of the form $\xi=a_0 \xi_0+ a_1 \eta$. Any $\xi=a_0\xi_0+a_1\eta\in \ft_{\bR}$ determines a valuation $\wt_{\xi}$ on $\bC(\cX_0)$:
\[
\wt_{\xi}(f)=\min \left\{\langle\tau,  \xi \rangle; f^\tau\neq 0 \right\} \text{ for } f=\sum_\tau f^\tau \in S.
\]
This is clearly a $\bT$-invariant valuation. In particular it is $\bC^*$-invariant, i.e. $T_0$-invariant. 
%The $T_1$ action generated by $\eta$ on $\cV_0$ determines a valuation $\wt_\eta^{(0)}$ on $\bC(\cV'_0)$ and $\wt_{\xi}$ can also be viewed as a $T_0=\bC^*$-invariant extension of $a_1\wt^{(0)}\eta$. 
\begin{defn}
The cone of positive vectors is defined to be:
\begin{equation}\label{eq-conepos}
\ft^{+}_{\bR}=\{ \xi=a_0\xi_0+a_1 \eta \in \ft_\bR; \langle \tau, \xi\rangle=a_0\tau_0+a_1 \tau_1 >0 \text{ for any } \tau \in \Gamma\setminus{(0,0)} \}.
\end{equation}
\end{defn}
%\begin{rem}
$\ft^{+}_{\bR}$ is essentially the same as the cone of Reeb vector fields considered in Sasaki geometry (see \cite{MSY08}). The following is a standard fact:
%\end{rem}
\begin{lem}\label{lem-opencone}
The cone $\ft^{+}_{\bR}$ contains an open neighborhood of $\xi_0$. 
\end{lem}
\begin{proof}
For a fixed $\tau_0\in \bZ_{>0}$, 
\[
H^0(\cV_0, \tau_0L_0)=S_{\tau_0}=\bigoplus_{\tau_1} S_{(\tau_0, \tau_1)} \text{ with } S_{(\tau_0, \tau_1)}\neq 0
\]
is the weight space decomposition with respect to the action $T_1=e^{\bC\eta}$ on $(\cV_0, \cL_0)$. Because $S$ is finitely generated, there exists $A>0\in\bR$ such that 
$-A \tau_0< \tau_1< A \tau_0$ for any $(\tau_0, \tau_1)$ appearing in the above decomposition (see \cite[Proof of Corollary 3.4]{BHJ15}). So for any $(a_0, a_1)\in \bR^2$ with $a_0>0$ and $|a_1 a_0^{-1}|<A^{-1}$, we have:
\[
 a_0 \tau_0+a_1 \tau_1\ge a_0 (\tau_0+a_0^{-1}a_1 \tau_1)\ge a_0(\tau_0-|a_0^{-1} a_1| A\tau_0)>0.
\]
The lemma follows immediately by choosing $a_0=1$.
%$\wt_{\xi_0}=\ord_{\cV_0}$ is the divisorial valuation where $\cV_0$ is identified with the blow up the vertex $o'\in \cX'_0$. 
\end{proof}
One sees directly that $\xi\in \ft_{\bR}^+$ if and only if the center of $\wt_{\xi}$ is the vertex $o'\in \cX_0$, i.e. $\wt_{\xi}\in \Val_{\cX_0,o'}$. 
For $\beta\in \mathbb{R}$, let $\xi_\beta=\xi_0+\beta \eta\in \ft_{\bR}$. The associated valuation $\wt_{\xi_\beta}$ will be denoted by $\wt_\beta$.
%Conversely we have $\Gamma=(\ft_{\bR}^{+})^{\vee}\cap M$.
%$M=N^{\vee}={\rm Hom}_{\bZ}(N, \bZ)\cong \bZ^2$, $M_{\bR}=M\otimes_{\bZ}\bR$ and $M_{\bC}=M\otimes_{\bZ}\bC$.
It's clear that $\xi_0\in \ft_{\bR}^{+}$. By Lemma \ref{lem-opencone}, for $0\le \beta\ll 1$, $\xi_\beta\in \ft^{+}_{\bR}$ and hence the valuation $\wt_\beta$ is in $\Val_{\cX_0, o'}$. Moreover, from the definition, we know that 
\begin{equation}\label{eq-wtbeta}
\wt_{\beta}(f)=k + \beta\cdot \wt_{\eta}(f) \text{ for any } f\in S_k=H^0(\cV_0, k\cL_0).
\end{equation}

%The main point is that the associated valuation $\wt_\beta$ corresponds to the valuation $v^{(1)}_E$ on $\bC(X)$ that we are going to discuss.
% When $m\gg 1$, by flatness, we have a $(\bC^*)^2$-equivariant vector bundle $\mathcal{E}_m$ over $\bC^1$, whose space of global sections is given by $H^0(\cV, m\cL)$. By raising $\cL$ to higher powers, we can assume that this holds for any $m\ge 1$.

\subsection{Valuation $w_\beta$ on $X$}\label{sec-v1E}

%\subsection{Valuation $v^{(1)}_E$ on $\bC(X)$}\label{sec-v1E}

For any $f\in H^0(V, m L)$, $f$ determines 
a $T_1$-invariant section $\bar{f}\in H^0(V\times \bC^*, \pi_1^*(mL))$. Since we have a $T_1$-equivariant isomorphism $(\cV\backslash \cV_0, \cL)\cong (V\times\bC^*, \pi_1^*L)$, $\bar{f}$ can be seen as a meromorphic section of $H^0(\cV, m\cL)$. Recall that $\cV_0$ will also be denoted by $E$. We will define a valuation $v^{(1)}_{E}$ on $\bC(X)$ satisfying
$v^{(1)}_{E}(f)=\ord_{\cV_0}(\bar{f})$ for any $f\in H^0(V, mL)$ and another valuation $v^{(1)}_{*E}$ on $\bC(X_*)$ such that $v^{(1)}_E$ is the restriction of $v^{(1)}_{*E}$ under inclusion $\bC(X)\hookrightarrow \bC(X_*)$ that is induced by the natural embedding:
\[
R=\bigoplus_{k=0}^{+\infty} H^0(V, k L)=\bigoplus_{k=0}^{+\infty} H^0(V, \sigma k L_*) \hookrightarrow R_*=\bigoplus_{m=0}^{+\infty} H^0(V, mL_*).
\]
To get a precise formula for $v^{(1)}_{E}$ and $v^{(1)}_{*E}$, we need some preparations. Denote by $v_{\cV_0}$ or $v_E$ the restriction of the valuation $\ord_{\cV_0}$ on $\bC(\cV)\cong \bC(V\times\bC)$ to $\bC(V)$. By \cite[Lemma 4.1]{BHJ15}, if $\cV_0$ is not the strict transform of $V\times \{0\}$, then
\begin{equation}\label{ordF}
v_{\cV_0}=v_E=\ord_{\cV_0}|_{\bC(V)}=q\cdot \ord_{F}
\end{equation}
for a prime divisor $F$ over $V$ and a rational $q>0$ (called multiplicity of $v$). Otherwise $v_{\cV_0}=v_E$ is trivial. 
\begin{rem}
A product test configuration, i.e. a test configuration induced from a $\bC^*$ action on $V$, does not necessarily give rise to a trivial valuation. See \cite[Example 3]{LX14} (see also \cite[Example 2.8]{BHJ15}).
\end{rem}
$v_{E}$ induces a valuation $v^{(0)}_E$ on $\bC(X)$ and a valuation $v^{(0)}_{*E}$ on $\bC(X_*)$ in the following way. Choose an affine open neighborhood $U$ of ${\rm center}_V(v_E)$ and a local trivializing section $s$ of $L$ on $U$. Any $f\in H^0(V, mL)$ can be locally written as $f=h s^m$ on $U$. Define:
\[
v^{(0)}_{E}(f)=v_{E}(h).
\]
It's clear that $v^{(0)}_{E}(f)\ge 0$ for any $f\in H^0(V, mL)$ and the equality holds if and only if $f$ does not vanish on the center of $v^{(0)}_{E}$ over $V$. In particular, we see that $v^{(0)}_E$ is trivial if and only if $v_E$ is trivial.  If $v_E=q\cdot\ord_F$ is not trivial and hence a divisorial valuation. We claim that 
\begin{lem}\label{lem-v0Ediv}
$v^{(0)}_E$ is a divisorial valuation. 
\end{lem}
\begin{proof}
We can find a birational morphism $\mu: V'\rightarrow V$ such that $F$ is a prime divisor on $V'$. 
$Y'$ denotes the total space of the pull back line bundle $\pi': (\mu')^*L^{-1}\rightarrow V'$. Define $\bF=\pi'^{-1}(F)$, which is a prime divisor on $Y'$. 
Then we have a sequence of birational morphisms $Y'\stackrel{\pi_2}{\rightarrow} Y=Bl_oX\stackrel{\pi_1}{\rightarrow} X$, and by definition it's easy to see that $v^{(0)}_E=q\cdot \ord_{\bF}$. 
\end{proof}
The above construction can be carried out for $X_*=C(V, L_*)$ by using $L_*$ instead of $L$ and we get a valuation denoted by $v^{(0)}_{*E}$, a divisor $\bF_*$ over $X_*$ such that $v^{(0)}_{*E}=q\cdot \ord_{\bF_*}$. One verifies that $v^{(0)}_E$ is the restriction of $v^{(0)}_{*E}$ under the natural embedding $R\hookrightarrow R_*$. Later we will need the following log discrepancies.
\begin{lem}\label{ldbF}
We have identities of log discrepancies:
\begin{equation}
A_X(\bF)=A_V(F)=A_{X_*}(\bF_*).
\end{equation}
\end{lem}
\begin{proof}
Let $Y=Bl_oX\rightarrow X$ be the blow up with the exceptional divisor $V$. Then we have:
\[
A_X(\ord_{\bF})=A_{(Y, (1-r)V)}(\ord_{\bF})=A_{Y}(\ord_{\bF})+(r-1) \ord_{\bF}(V)=A_{Y}(\ord_{\bF}).
\]
The last identity is because that $\ord_{\bF}(V)=0$ by the definition of $\ord_{\bF}=v^{(0)}_E$. Now let $\mu: V'\rightarrow V$ and $\pi_2: Y'\rightarrow Y$
be the same as in the above paragraph. Assume that the support of exceptional divisors of $\mu$ is decomposed into irreducible components as $F+D_1+\dots+D_k=:F+D$. Then correspondingly the support of exceptional divisor of $\pi_2$ can be decomposed into $\bF+\mathbb{D}$ where $\mathbb{D}=\pi'^{-1}D$. So we can write:
\[
K_{Y'}+\bF=\pi_2^*K_Y+A_Y(\bF)\bF+\Delta, \quad \pi_2^*V=V',
\]
where $\Delta$ has the support contained in $D$.
By adjunction, we get $K_{V'}+F=\mu^*K_V+A_Y(\bF)F+\Delta|_{V'}$, and hence $A_V(F)=A_Y(\bF)$. The same argument can be applied to $(X_*, \bF_*)$ and we get the second identity.
%\begin{eqnarray*}
%A_X(v^{(0)}_E)&=&\\
%&=&A_V(\ord_{F})=A_{V\times\bC}(\cV_0)-1=-c_2 r,
%\end{eqnarray*}
\end{proof}
%we denote $\bF=\overline{(\pi_{\times})^{-1}(F)}$ where $\pi_{\times}: X\backslash o\rightarrow \bV$ is the natural $\bC^*$-bundle. Then it's easy to see that $v^{(0)}_{E}=\ord_{\mathbb{F}}$. 

Analogous to the discussion in \cite{BHJ15}, there are $\bC^*$-invariant extensions of $v_E$ from $\bC(V)$ to $\bC(X)$ defined as follows. For any $f\in \bC(X)$, we can use $T_0$-action to get Laurent typed series $f=\sum_{m\in\bZ}f_m$ with $f_m$ a meromorphic section of $mL$. Then we define a family of extensions of $v_E$ with parameter $\lambda$:
\begin{equation}
\bV_\lambda(f):=\min_{m\in \bZ}\left\{v^{(0)}_E(f_m)+\lambda m \right\}.
\end{equation}
%Since $v^{(0)}_E=q\cdot \ord_{\bF}$, equivalently $\bV_\lambda$ is defined as the quasi-monomial valuation of $(V+\bF)$ with weight $(1,q \lambda)$. 
It's clear that $\bV_\lambda$ is $\bC^*$-invariant, i.e. $T_0$-invariant.
In the same way, we define the extensions of $v_{E}$ from $\bC(V)$ to $\bC(X_*)$. For any $g_*=\sum_{k\in \bZ} g_{*k}\in \bC(X_*)$ with $g_{*k}$ a meromorphic section of $kL_*$, we define:
\[
\bV_{*\mu}(g_*):=\min_{k\in \bZ} \left\{v^{(0)}_{*E}(g_{*k})+\mu k \right\}.
\]
Because $H^0(V, mL)=H^0(V, m\sigma L_*)$, it's easy to verify that $\bV_\lambda$ is
the restriction of $\bV_{*\lambda/\sigma}$ under the embedding $R\subset R_*$ and $\bC(X)\subset \bC(X_*)$.

Let $A_{V\times\bC}(\cV_0)$ be the log discrepancy of $\cV_0$ over $V\times\bC$, and define a constant 
\begin{equation}\label{defc2}
a_1:=-\frac{A_{V\times\bC}(\cV_0)-1}{r}.
\end{equation}
\begin{rem}\label{rema1}
By \cite[Proposition 4.11]{BHJ15}, we have $a_1=-q\cdot A_V(F)/r$.
\end{rem}
%Now we can precisely define the valuations we wanted:
\begin{defn}
Define a valuation on $\bC(X)$ and a valuation on $\bC(X_*)$ as follows:
\begin{equation}
v^{(1)}_E=\bV_{a_1} \text{ on } R, \quad v^{(1)}_{*E}=\bV_{*(a_1/\sigma)} \text{ on } R_*.
\end{equation}
\end{defn}
By the following lemma, this is exactly what we want.
\begin{lem}\label{v01form}
If $v_0$ denotes the divisorial valuation $\ord_{V}$, then
\begin{enumerate}\item
For any $f\in H^0(V, mL)$, we have:
\[
\ord_{\cV_0}(\bar{f})=v^{(0)}_{E}(f)+a_1 m=v^{(0)}_E(f)+a_1 v_0(f)=\bV_{a_1}(f).
\]
\item We have a sharp lower bound:
\[
v^{(1)}_{E}(V)=\inf_{\fm}\frac{v^{(1)}_{E}}{v_0}=a_1.
\]
\end{enumerate}

\end{lem}
In the following, we will denote by $v_{*0}$ the divisorial valuation $\ord_{V}$ on $R_*$. If we want to emphasize $\ord_V$ as a valuation on $\bC(X_*)$, we will write $\ord_{V_*}$.  
Notice that $\ord_{V}=v_{0}=\frac{v_{*0}}{\sigma}=\frac{\ord_{V_*}}{\sigma}$ 
under the embedding $R\hookrightarrow R_*$. The following corollary is clearly a consequence of the above lemma.
\begin{cor}\label{v*01form}
For any $f\in H^0(V, \sigma m L_*)=H^0(V, m L)$, we have:
\[
\ord_{\cV_0}(\bar{f})=v^{(0)}_{E}(f)+a_1 m=v^{(0)}_{*E}(f)+\frac{a_1}{\sigma} v_{*0}(f)=\bV_{*(a_1/\sigma)}(f),
\] 
and there is a sharp lower bound:
\[
v^{(1)}_{*E}(V)=\inf_{\mathfrak{m}} \frac{v^{(1)}_{*E}}{v_{*0}}=\frac{a_1}{\sigma}.
\]
Moreover, $\bV_{a_1}$ is the restriction of $\bV_{*(a_1/\sigma)}$ under the embedding $R\hookrightarrow R_*$.
\end{cor}
\begin{proof}[Proof of Lemma \ref{v01form}]
The proof is motivated by that of \cite[Lemma 5.17, 5.18]{BHJ15}.
Let $\cW$ be a normal test configuration dominating both $V\times\bC$ and $\cV$ by the morphisms $\rho_1: \cW\rightarrow V\times \bC$ and $\rho_2: \cW\rightarrow \cV$. We can write down the identities:
\begin{equation}\label{comptc1}
K_{\cW}=\rho_1^* K_{V\times\bC}+a\widehat{\cV}_0+\sum_{k=1}^{N} d_k F_k=\rho_2^*K_{\cV}+b \hV+\sum_{k=1}^N e_k F_k
\end{equation}
By the definition of log discrepancy, we see that $A_{V\times \bC}(\cV_0)=a+1$. 
By \eqref{comptc1}, we also get:
\begin{equation}\label{comptc2}
\rho_2^*K_{\cV}^{-m/r}=\rho_1^*K_{V\times \bC}^{-m/r}-\frac{a m}{r}\widehat{\cV}_0+\frac{bm}{r}\hV-\sum_{k=1}^N\frac{m(d_k-e_k)}{r}F_k.
\end{equation}
So by the above identity, for any $f\in H^0(V, mL)$, we have (cf. \cite[Proof of Lemma 5.17]{BHJ15})
\[
\ord_{\cV_0}(\bar{f})=v^{(0)}_{\cV_0}(f)-\frac{am}{r}=v^{(0)}_E(f)-\frac{A_{V\times\bC}(\cV_0)-1}{r}m. 
\]
Because $v^{(0)}_{E}(f)=\ord_{F}(f)\ge 0$, %is strictly positive if $f$ does not vanish identically on the center of $\ord_{F}$ on $V$. 
we have
\begin{equation}\label{ratecV}
\frac{v^{(1)}_{E}(f)}{\ord_{V}(f)}=\frac{\ord_{\cV_0}(\bar{f})}{m}\ge -\frac{a}{r}.
\end{equation}
Because $mL$ is globally generated when $m\gg 1$, we can find $f'\in H^0(V, mL)$ which is nonzero on the center of $v_{\cV_0}=v_E$, so that $\ord_{\cV_0}(\bar{f'})=-\frac{a}{r}m$. So we see that the lower bound in \eqref{ratecV} can be obtained, and the second statement follows from the proof of Lemma \ref{lemdefc1}.
\end{proof}

%\subsection{Valuations $w_\beta$ on $X$}
Although the valuation $v^{(1)}_{E}$ is not necessarily positive over $R$, we can still consider the graded filtration of $v^{(1)}_{E}$ on $R$: 
\begin{eqnarray*}
\cF^\lambda_{v^{(1)}_{E}} H^0(V, kL)&=&\{f\in H^0(V, kL); v^{(1)}_E(f)=\ord_E(\bar{f})\ge \lambda \}\\
&=&\{f\in H^0(V, kL); t^{-\lambda}\bar{f}\in H^0(\cV, k\cL)\}.
\end{eqnarray*}
Notice that this filtration was studied in \cite{BHJ15} and \cite{WN12}. This is a decreasing, left-continuous, multiplicative filtration. Moreover, it's linearly bounded from below, which means that:
\[
e_{\min}=e_{\min}\left(R_\bull, \cF_{v^{(1)}_E}\right)=\liminf_{k\rightarrow +\infty}\frac{\inf\left\{\lambda; \cF^{\lambda}_{v^{(1)}_E}H^0(V, kL)\neq H^0(V, kL)\right\}}{k}>-\infty.
\]
Note that this follows from Lemma \ref{v01form}, by which we have $e_{\min}\ge -\frac{A_{V\times\bC}(\cV_0)-1}{r}$, and is clearly equivalent to the following statement.
\begin{lem}
There exists $M_0\gg 1\in \bZ$, such that for any $f\in H^0(V, kL)$ and any $k\ge 0$, $t^{M_0 k}\bar{f}$ extends to become a holomorphic section of $H^0(\cV, k\cL)$.
\end{lem}
As a consequence, for any $M\in \bZ$ with $M\ge M_0$ we have an embedding of algebras $\Psi_M: R\rightarrow \cR$, which maps $f\in H^0(V, kL)$ to $t^{Mk} \bar{f}\in H^0(\cV, k\cL)$. Now we can define a family of valuations on $\bC(X)$ which are finite over $R=\bigoplus_{k=0}^{+\infty} H^0(V, kL)$ with the center being the vertex $o$. For any $\beta=M^{-1}$ with $M\in\bZ$ and $\beta^{-1}=M\ge M_0$ and any $f\in H^0(V, kL)$, define
\begin{eqnarray}\label{defwbeta}
w_{\beta} (f)&:=&\beta\cdot \ord_{\cV_0}(\Psi_{\beta^{-1}}(f))=\beta\cdot \ord_{\cV_0}\left(t^{\beta^{-1}k}\bar{f}\right)\nonumber\\
&=&k+\beta\cdot\ord_{\cV_0}(\bar{f})=\ord_V(f)+\beta\cdot v^{(1)}_{E}(f).
\end{eqnarray}
In other words, $w_\beta$ is a family of valuations starting from $v_0=\ord_{V}$ ``in the direction of $v^{(1)}_{E}$". Notice that
by Lemma \ref{v01form}, we have $v^{(1)}_E(f)=v^{(0)}_E(f)+a_1 v_0(f)$, where $a_1=v^{(1)}_E(V)=-(A_{V\times\bC}(\cV_0)-1)/r$. So we can rewrite 
$w_\beta(f)$ as:
\begin{eqnarray}\label{wbetaquasi}
w_\beta(f)=v_0(f)+\beta\cdot \left(v^{(0)}_E(f)+a_1 v_0(f)\right)=\left[(1+\beta a_1)v_0+\beta v^{(0)}_E\right] (f).
\end{eqnarray}
From the above formula, we see that $w_\beta$ is nothing but the quasi-monomial valuation of $(Y', V+\bF)$ with weight $(1+\beta a_1, \beta q)$, where $(Y', V+\bF)$ is the same pair as in the proof of Lemma \ref{lem-v0Ediv} . Note that $(1+\beta a_1, \beta q)$ is an admissible weight (i.e. having non-negative components) when $0\le \beta\ll 1$. If $v_E$ is trivial, then $w_\beta\equiv v_0=\ord_V$. Notice that the description of $w_\beta$ as quasi-monomial valuations does not require $\beta^{-1}=M$ to be an integer. So from now on we define
\begin{defn}\label{def-wbeta}
Using the above notation, for any $0\le \beta\ll 1$, the valuation $w_\beta$ is defined to be the quasi-monomial valuation on the model $(Y', V+\bF)$ with the weight $(1+\beta a_1, \beta q)$. Equivalently, for any $f\in R=\bigoplus_k H^0(V, kL)$, we define:
\[
w_\beta(f)=\min\left\{k+ \beta v^{(1)}_E(f); f=\sum_k f_k \text{ with } f_k\neq 0 \right\}.
\]
\end{defn}

Correspondingly, on $X_*$ we define $w_{*\beta}$ to be the quasi-monomial valuation of $V_*+\bF_*$ with the weight $(\frac{1}{\sigma}+\beta \frac{a_1}{\sigma}, \beta q)$ if $v_E$ is not trivial, and otherwise $w_{*\beta}\equiv \frac{v_{*0}}{\sigma}=\frac{1}{\sigma}\ord_{V_*}$, where $V_*=V$ is considered as a divisor over $X_*$. By the above discussion, $w_\beta$ is the restriction of $w_{*\beta}$ under the embedding $R\hookrightarrow R_*$.

In the following two lemmas, we compare the volumes and log discrepancies of $w_\beta$ and $w_{*\beta}$.
\begin{lem}\label{mulvol}
We have an identity of volumes:
\begin{equation}\label{volww*}
\vol(w_\beta)=\frac{1}{\sigma}\vol(w_{*\beta}).
\end{equation}
\end{lem}
\begin{proof}
%Indeed, because the valuations are all $\bC^*$-invariant, we can use Corollary \ref{lemsat} to first get:
%\[
%\cF^m R_k=\fa_m(w_\beta)\bigcap R_k, \quad \cF_*^m R_{*\sigma k}=\fa_m(w_{*\beta})\bigcap R_{*\sigma k}.
%\]
For any fixed $\beta$, we use $\cF R_k$ (resp. $\cF_{*}R_{*}$) to denote the filtrations determined by $w_\beta$ (resp. $w_{*\beta}$) on $R$ (resp. $R_*$).
 For any $f\in H^0(k L)$, $f\in\cF^m R_k$ if and only $w_\beta (f)\ge m$.
Because $w_\beta=w_{*\beta}$ when both are restricted to $H^0(kL)=H^0(\sigma k L_*)$, this holds if and only if $f\in \cF_*^{m}R_{* \sigma k}$. So we see that $\cF^{m} R_k=\cF^m R_{*\sigma k}$. As
a consequence we can calculate:
\begin{eqnarray*}
\vol(\cF R^{(t)})&=&\lim_{k\rightarrow+\infty}\frac{\dim_{\bC}\cF^{kt} H^0(kL)}{k^{n-1}/(n-1)!}\\
&=&\lim_{k\rightarrow+\infty}\frac{\dim_{\bC}\cF_*^{\sigma k (t/\sigma)} H^0(\sigma k L_*)}{(\sigma k)^{n-1}/(n-1)}\sigma^{n-1}\\
&=&\sigma^{n-1}\vol\left(\cF_* R^{(t/\sigma)}_*\right).
\end{eqnarray*}
By the volume formula obtained in \eqref{volv1b}, we get the identity \eqref{volww*} by change of variables in the
integral formula. Notice the change of integration limits is valid by the sharp lower bounds obtained in Lemma \ref{v01form} and Corollary \ref{v*01form}.
\end{proof}
\begin{lem}
We have the identities of log discrepancies:
\begin{equation}
A_{X}(w_{\beta})\equiv r\equiv A_{X_*}(w_{*\beta}).
\end{equation}
\end{lem}
\begin{proof}
If $v_E$ is trivial, then $w_\beta=v_0$ and $A_X(w_\beta)=A_X(v_0)=r$. Otherwise, by \eqref{wbetaquasi} $w_\beta$ is a quasi monomial valuation of $(V+\bF)$ with weight $(1+\beta a_1, \beta q)$. So we have:
\begin{eqnarray}\label{AXwbeta}
A_X(w_\beta)&=&(1+\beta a_1) A_X(\ord_V)+\beta q A_X(\ord_{\bF})=(1+\beta a_1) r+\beta q A_X(\bF)\nonumber\\
&=&(1+\beta a_1)r+\beta q A_V(F),
\end{eqnarray}
where we used $A_X(\bF)=A_V(F)$ by Lemma \ref{ldbF}.
%recall that $v^{(0)}_E=\ord_{\bF}$ and $v^{(0)}_E(V)=0$. As before, let $Y\rightarrow X$ be the blow up of $o$ with the exceptional divisor $V$. Then by adjunction, 
%we have $A_Y(v^{(0)}_E)-v^{(0)}_E(V)=A_V(\ord_{F})$, where $\ord_F=v_E$ is the restriction of $\ord_{\cV_0}$ to $\bC(V)$. So we have
%Now by \cite[Proposition 4.11]{BHJ15}, we have $A_{V\times \bC}(\cV_0)=1+A_V(v^{(0)}_E)=1+q A_V(F)$. So we get 
%\[
%q\cdot A_V(F)=\frac{A_{V\times\bC}(\cV_0)-1}{r}r=-a_1 r,
%\]
%by the definition of $a_1$ in \eqref{defc2}.
%The last identity is by Lemma \ref{v01form}. 
By Remark \ref{rema1}, $q\cdot A_V(F)=-a_1 r$.
So we see that $A_X(w_\beta)\equiv r$ by \eqref{AXwbeta}.

For $A_{X_*}(w_{*\beta})$, we notice that:
\begin{eqnarray*}
A_{X_*}(w_{*\beta})&=&\frac{1+\beta a_1}{\sigma}A_{X_*}(\ord_{V_*})+\beta q A_{X_*}(\ord_{\bF_*})\\
&=&(1+\beta a_1)r+\beta q A_V(F)=A_X(w_\beta).
\end{eqnarray*} 
\end{proof}
As a corollary of above two lemmas, we have:
\begin{cor}\label{corww*}
We have the identity of normalized volumes:
\[
\hvol(w_\beta)=\frac{1}{\sigma}\cdot \hvol(w_{*\beta}).
\]
\end{cor}

\subsection{$w_\beta$ vesus $\wt_\beta$}\label{sec-wvwbeta}

First recall the definition of $\wt_\beta$ from Section \ref{sec-wtbeta}. If we denote $\wt^{(1)}=\wt_\eta$, then for any $f\in H^0(\cV_0, k\cL_0)$, we have (see \eqref{eq-wtbeta})
% we consider the subtorus
%$G_\beta$ generated by the element $\xi_\beta=\xi_0+\beta\eta \in \bC\xi_0\bigoplus\bC\eta$. It's easy to see that for $0\le \beta\ll 1$, $G_\beta$ has nonnegative weights on $S=\cR/(t)\cR$ and has strictly positive weights on $S_{+}:=\bigoplus_{i>0}H^0(\cV_0, i\cL_0)$. As a consequence, $G_\beta$ determines a valuation $\wt_\beta$ which is finite over $S$ and is centered at the vertex $o\in \cX_0$. If we denote by $\wt^{(1)}$ the weight function associated to the $T_1$-action, then we have:
\begin{equation}\label{wtbetacomp}
\wt_\beta(f)=\wt_0(f)+\beta\cdot \wt^{(1)}(f):=\ord_{\cV_0}(f)+\beta\cdot \wt^{(1)}(f).
\end{equation}
Note that this decomposition corresponds exactly to \eqref{defwbeta}. Our main goal is to understand the following correspondence:
\[
\wt_\beta  \longleftrightarrow  w_\beta, \quad
\wt^{(1)}  \longleftrightarrow  v^{(1)}_{E}.
\]
\begin{prop}\label{wwbeta}
For any $0\le \beta\ll 1$, we have the following relation between $w_\beta$ and $\wt_\beta$.
\begin{enumerate}
\item $\vol(\wt_\beta)=\vol(w_\beta)$.
\item $A_{\cX_0}(\wt_\beta)\equiv r \equiv A_{\cX_0}(\ord_{\cV_0})$.
\end{enumerate}
\end{prop}

%\subsection{Proof of Proposition \ref{wwbeta}}
\begin{proof}
%The rest of this section is devoted to the proof of Proposition \ref{wwbeta}. 
Notice that if $\beta=0$, then 
$w_0=\ord_{V}=v^X_{0}$ and $\wt_0=\ord_{\cV_0}=:v^{\cX_0}_{0}$. So we have:
\[
\vol(v^X_{0})=L^{n-1}=\vol(v^{\cX_0}_{0}), \quad A_X(v^X_{0})=r=A_{\cX_0}(v^{\cX_0}_{0}).
\]
For simplicity of notations, when it's clear, we will denote both $\ord_V$ and $\ord_{\cV_0}$ by $v_0$. So Proposition 1 holds for $\beta=0$. From now on we assume $\beta\neq 0$.
\begin{enumerate}
\item
We first prove item 1. 
%We will construct a linear isomorphism between the $\bC$-vector space $\cF^\lambda_{\wt_\beta} H^0(\cV_0, m\cL_0)$ and $\cF^\lambda_{w_\beta} H^0(V, mL)$ for any $\lambda\in \bR$ and $m\gg 1$.
We will prove 
\begin{equation}\label{equaldim1}
\dim_{\bC}\cF^\lambda_{\wt_\beta} H^0(\cV_0, m\cL_0)=\dim_{\bC}\cF^\lambda_{w_\beta} H^0(V, mL)
\end{equation} 
for any $\lambda\in\bR$ and $m\gg 1$.
Notice that, by the definition of $w_\beta$ and $\wt_\beta$, we have that, (when $\beta\neq 0$):
\begin{eqnarray*}
&&\cF^\lambda_{\wt_\beta} H^0(\cV_0, m\cL_0)=\cF^{(\lambda-m)/\beta}_{\wt^{(1)}}H^0(\cV_0, m\cL_0);\\ 
&&\cF^\lambda_{w_\beta} H^0(V, mL)=\cF^{(\lambda-m)/\beta}_{v_E^{(1)}} H^0(V, mL).
\end{eqnarray*}
So we just need to show 
\begin{equation}\label{equaldim2}
\dim_{\bC}\cF^\lambda_{v^{(1)}_{E}}H^0(X, mL)=\dim_{\bC}\cF^\lambda_{\wt^{(1)}}H^0(\cV_0, m\cL_0)
\end{equation} 
for any $\lambda\in \bR$ and $m\gg 1$. This actually follows from \cite[Section 2.5]{BHJ15}. We give a separate proof for the convenience of the reader.

When $m\gg 1$, by flatness, we have a $T_1$-equivariant vector bundle $\mathcal{E}_m$ over $\bC^1$, whose space of global sections is $H^0(\cV, m\cL)=:\mathbb{E}_m$. Notice that $\bE_m$ is a natural $\bC[t]$ module satisfying $\bE_m/(t)\bE_m\cong H^0(\cV_0, m\cL_0)=:\bE_m^{(0)}$. The $T_1$-action is compatible under this quotient map. For any $g\neq 0\in \bE^{(0)}_m$ of weight $\lambda$, we can extend it to become $\cG\neq 0\in \bE_m$ of weight $\lambda$. In this way, we get a $T_1$-equivariant splitting morphism $\tau_m: \bE_m^{(0)}\rightarrow \bE_m$. We denote by $\Phi_m: \bE_m^{(0)}\rightarrow \bE_m^{(1)}$ the composition of $\tau_m$ with the quotient $\bE_m\rightarrow \bE_m/(t-1)\bE_m\cong H^0(V, mL)=:\bE_m^{(1)}$. From the construction, if $g$ has weight $\lambda $, then $t^{\lambda}\cG=:\bar{f}$ is $T_1$-invariant with its restriction on $\cV_1\cong V$ denoted by $f$. Then we have $v_{\cV_0}(f)=\ord_{\cV_0}(\bar{f})=\ord_{\cV_0}(t^{\lambda}g)=\lambda$. So we see that $\Phi_m\left(\cF_{\wt^{(1)}}^{\lambda}H^0(\cV_0, m\cL_0)\right)\subseteq \cF^\lambda_{v^{(1)}_{E}} H^0(V, mL)$. We claim that $\Phi_m$ is injective. For any $g=g_{\lambda_1}+\dots+g_{\lambda_k}\neq 0$ with each nonzero $g_{\lambda_j}$ having weight $\lambda_j$ and $\lambda_1<\lambda_2<\dots<\lambda_k$, we have that $\widetilde{\cG}_{\lambda_j}:=t^{\lambda_j}\cG_{\lambda_j}$ are $T_1$-invariant sections of $\cL\rightarrow \cV$. So $\left.\widetilde{G}_{\lambda_j}\right|_{\cV\backslash\cV_0}\cong: \left.\left(\pi_1^*f_{\lambda_j}\right)\right|_{V\times\bC^*}$ where $\pi_1: V\times\bC\rightarrow V$ is the projection and we used the isomorphism $(\cV\backslash\cV_0, \cL)\cong (V\times\bC^*, \pi_1^*L)$. $\widetilde{\cG}$ are linearly independent because they have different vanishing orders along $\cV_0$.  Hence $f_{\lambda_j}$ are also linearly independent. So $\Phi_m(g)=f_{\lambda_1}+\dots+f_{\lambda_k}\neq 0$ if $g\neq 0$. So we get 
%\begin{equation}\label{eq-eqwtdim}
$\dim_{\bC}\cF^{\lambda}_{\wt^{(1)}}H^0(\cV_0, m\cL_0)\le \dim_{\bC}\cF^{\lambda}_{v^{(1)}_E} H^0(V, mL)$.
%\end{equation}
%By flatness of the family $(\cV, \cL)\rightarrow \bC$, we get the identity $\dim_\bC H^0(\cV_0, m\cL_0)=\dim_\bC H^0(V, mL)$ for $m\gg 1$. Combined with \eqref{eq-eqwtdim},  

The above argument also shows that for any $\lambda\in \mathbb{R}$, we have:
\begin{equation}\label{ineq-qdim}
\dim_{\bC}\frac{\cF^\lambda_{\wt^{(1)}} H^0(\cV_0, m \cL_0)}{\cF^{>\lambda}_{\wt^{(1)}}H^0(\cV_0, m\cL_0)}\le \dim_{\bC}\frac{\cF^\lambda_{v^{(1)}}H^0(V, mL)}{\cF^{>\lambda}_{v^{(1)}}H^0(V, mL)}.
\end{equation}
On the other hand, we have the equality:
\begin{eqnarray*}
\sum_{\lambda} \dim_{\bC}\frac{\cF^\lambda_{\wt^{(1)}} H^0(\cV_0, m \cL_0)}{\cF^{>\lambda}_{\wt^{(1)}}H^0(\cV_0, m\cL_0)}&=&\dim_{\bC} H^0(\cV_0, m\cL_0)\\
&=& \dim_{\bC} H^0(V, mL)=\sum_{\lambda} \dim_{\bC}\frac{\cF^\lambda_{v^{(1)}}H^0(V, mL)}{\cF^{>\lambda}_{v^{(1)}}H^0(V, mL)}.
\end{eqnarray*}
So in fact \eqref{ineq-qdim} is an equality and consequently \eqref{equaldim1} holds for any $\lambda \in \bR$ and $m\gg 1$.
%Reversely any $f\neq 0\in H^0(V, mL)$ can be extended to a unique $T_1$-invariant section $F\in H^0(\cV\backslash \cV_0, m\cL)$ by using the $T_1$-equivariant isomorphism $(\cV\backslash \cV_0, \cL)\cong (V\times\bC^*, \pi_1^*L)$ where $\pi_1: V\times\bC\rightarrow V$ is the projection. Let $\lambda=\ord_{\cV_0}(\bar{f})$. Then $\cG':=t^{-\lambda}\bar{f}$ extends to become a nonzero holomorphic section of $H^0(\cV, m\cL)$ with weight $\lambda$. The restriction $g:=\cG'|_{\cV_0}$ is nonzero in $H^0(\cV_0, m\cL_0)$. % we have $\cG'-\tau_m(g)\in (t)\bE$. 
%In this way we get a linear morphism $\Psi_m: \cF^\lambda_{v^{(1)}_{E}}H^0(V, mL)\rightarrow \cF^\lambda_{\wt^{(1)}}H^0(\cV_0, m\cL_0)$. We claim that $\Psi_m$ is injective. Indeed if $0\neq f=f_{\lambda_1}+\cdots+f_{\lambda_k}$ such that $\ord_{\cV_0}\bar{f}_{\lambda_j}=\lambda_j$ and $-\infty<\lambda_1<\lambda_2<\dots<\lambda_k<+\infty$. Then $\cG'_{\lambda_j}=t^{-\lambda_j}\bar{f}_{\lambda_j}$ have different weights and hence are linearly independent. So we get the inequality in the other direction $\dim_{\bC}\cF^{\lambda}_{\wt^{(1)}}H^0(\cV_0, m\cL_0)\ge \dim_{\bC}\cF^{\lambda}_{v^{(1)}_{E}}H^0(V, mL)$.
%we have $\Psi_m^{-1}\left(\cF^\lambda_{\wt^{(1)}} H^0(V, mL)\right)\subset \cF^{\lambda}_{\ord_{\cV_0}}H^0(\cV_0, m\cL_0)$.
%\[
%\ord_{\cV_0}(f)=\WT(\tau_m^{-1}(f)).
%\]
%Now $g=\cG|_V$ such that $t^{j_l}\bar{g}_{j_l}=\cG_{j_l}$. So
%$v_{\cV_0}(g)=j_l=\wt^0(g^{(0)})$.
%$v_{\cV_0}(g)=\ord_{\cV_0}(\bar{g})$.

Combining the above discussions, we see that \eqref{equaldim1} hold for any $\lambda\in \bR$ and $m\gg 1$. So we can compare the volumes by calculating:
\begin{eqnarray*}
\dim_{\bC}R/\fa_\lambda(w_\beta)&=&\sum_{m=0}^{+\infty}\dim_{\bC}H^0(V,mL)/\cF^\lambda_{w_\beta} H^0(V,mL)\\
&=&\sum_{m=0}^{+\infty}\dim_{\bC}H^0(\cV_0, m\cL_0)/\cF^\lambda_{\wt_\beta} H^0(\cV_0, m\cL_0)+O(1)\\
&=&\dim_{\bC}S/\fa_\lambda(\wt_\beta)+O(1).
\end{eqnarray*}
Now the identity on volumes follow by dividing both sides by $\lambda^n/n!$ and taking limits as $\lambda\rightarrow +\infty$.

\item
Next we calculate the log discrepancies of $\wt_\beta$. For this purpose we will use the interpretation of log
discrepancy of $\wt_\beta$ as the weight on the corresponding $\bC^*$-action on an equivariant nonzero pluricanonical-form (see \cite{Li15a}). Choose $m$ sufficiently divisible so that $mr$ is an integer and $mK_{\cX_0}=m r\cL^{-1}_0$ is Cartier. Then on $\cX_0=C(\cV_0, \cL_0)$, there is a nowhere zero $m$-canonical section $\Omega\in H^0(\cX_0, m K_{\cX_0})$ (alternatively we could think of $\Omega^{1/m}$ as a multi-section of $K_{\cX_0}$). This is a well known construction showing that $\cX_0$ is $\mathbb{Q}$-Gorenstein in our setting (see \cite[Proposition 3.14.(4)]{Kol13}). For the convenience of the reader, we give an explicit description.

To define $\Omega$ first choose an affine covering $\{U_\alpha\}$ of $\cV_0$ that induces an open covering $\{\pi^{-1}U_\alpha\}$ of $\cX_0\setminus \{o\}$. Notice that there is a natural identification $\cX_0\setminus \{o\}\cong \cL_0^{-1}\setminus \cV_0$ where $\cV_0$ is considered as the zero section of the total space $\cL_0^{-1}$ of the line bundle $\cL_0^{-1}\rightarrow \cV_0$. So we have a canonical projection $\pi: \cX_0\setminus \{o\}\rightarrow \cV_0$. 

For any affine open set $U=U_\alpha \subset \cV_0$, choose a local generator $s=s(p)$ of $\mathcal{O}_{\cV_0}(\cL_0^{-1})(U)$. Then 
\[
\pi^{-1}(U)=\{h\cdot s(p); p\in U, h\neq 0 \}\subset \cX_0\setminus \{o\} .
\]
In other words, any point $P\in \cX_0\setminus \{o\}$ can be represented as $h\cdot s$ over an affine open set $U$ of $\cV_0$ containing $p=\pi(P)$. Since $mK_{\cV_0}$ is Cartier and $K_{\cV_0}= r\cL_0^{-1}$, $s^{\otimes m r}$ is a local generator of $\mathcal{O}_{\cV_0}(m K_{\cV_0})$. Moreover if $p=\pi(P)$ is a smooth point, we can choose local  coordinates $z=\{z_1, \dots, z_{n-1}\}$ around $p$ such that $s^{\otimes mr}=(dz_1\wedge \cdots\wedge dz_{n-1})^{\otimes m}=:(dz)^{\otimes m}$. Then $\{z_1, \dots, z_{n-1}, h\}$ are local holomorphic coordinates on a neighborhood $U'$ of $P\in \pi^{-1}(U_{\rm reg})\subset \cX_0\setminus \{o\}$. We define 
\begin{equation}\label{eq-Omega}
\Omega=(d h^r)^{\otimes m} \wedge (dz)^{\otimes m}:=\pm\;  r^m h^{rm-m} (dh\wedge dz)^{\otimes m} \in H^0(U', m K_{\cX_0}).
\end{equation}
\begin{lem}
$\Omega$ defines a global section in $H^0(\cX_0, m K_{\cX_0})$. Moreover $\Omega$ is invariant under the canonical lifting (given by pull-back of pluri-canonical forms) of $T_1$-action on $mK_{\cX_0}$. 
\end{lem}
\begin{proof}
Because $\cX_0$ is normal and $mK_{\cX_0}$ is Cartier, we just need to verify the statement on the regular locus of $\cX_{0}$. To verify $\Omega$ is globally defined, we need to show that
$\Omega$ does not depend on the choices of local holomorphic coordinates. 

If $\tilde{h}^{mr}\cdot (d\tilde{z})^{\otimes m}=h^{mr} \cdot (dz)^{\otimes m}$, then $\tilde{h}^{mr}=h^{mr}{\rm Jac}\left(\frac{\partial z}{\partial \tilde{z}}\right)^{m}$. 
So 
\begin{eqnarray*}
mr \tilde{h}^{mr-1} d\tilde{h}\wedge d\tilde{z}&=& mr h^{mr-1} (dh)  {\rm Jac}(\partial z/\partial \tilde{z})^m\wedge {\rm Jac}(\partial \tilde{z}/\partial z)dz\\
&=&m r h^{mr-1}\tilde{h}^{mr}/h^{mr} {\rm Jac}(\partial \tilde{z}/\partial z)dh\wedge dz.
\end{eqnarray*}
So we get $\tilde{h}^{-1}d\tilde{h}\wedge d\tilde{z}={\rm Jac}(\partial\tilde{z}/{\partial z}) \cdot h^{-1} dh\wedge dz$. From this identity we easily get:
\begin{eqnarray}\label{eq-cvf}
\tilde{h}^{rm-m}(d\tilde{h}\wedge d\tilde{z})^{\otimes m}&=& \tilde{h}^{rm} (\tilde{h}^{-1}d\tilde{h}\wedge d\tilde{z})^{\otimes m}\nonumber\\
&=&h^{rm}(h^{-1} dh\wedge dz)^{\otimes m}
=h^{rm-m}(dh\wedge dz)^{\otimes m}.
\end{eqnarray}
So $\Omega$ is indeed globally defined. 

To see the invariance under $T_1$-action, we just need to verify the invariance of $\Omega$ on the regular locus. Recall that $T_1$ action on $\cX_0$ is induced by the canonical lifting of $\bC^*$-action on $\cL_0^{-mr}=K_{\cV_0}^m$ that is given by the pull back of $m$-pluri-canonical-forms (see Remark \ref{remaction}). For any 
smooth point $p\in \cV_0$ and $t\in T_1$, $t\circ p=\tilde{p}$ is also a smooth point of $\cV_0$. By the above change of variable formula in \eqref{eq-cvf}, we get $t^*\Omega=\Omega$ for any $t\in T_1=e^{\bC\eta}$. So $\Omega$ is indeed invariant under the action of $T_1$ on $\cX_0$. 

\end{proof}

On the other hand, $\Omega$ has weight $m r$ under the $T_0$-action generated by $\xi_0=h\partial_h$:
\[
\mathfrak{L}_{\xi_0}\Omega=\mathfrak{L}_{h\partial_h} (\pm r^m h^{rm-m}(dh\wedge dz)^{\otimes m})=mr \cdot \Omega, 
\]
where $\mathfrak{L}$ denotes the Lie derivative with respect to the generating holomorphic vector field. 
So we get
\begin{equation}\label{eq-Lxi}
\mathfrak{L}_{\xi_\beta}\Omega=\mathfrak{L}_{\xi_0}\Omega+\beta\mathfrak{L}_{\eta}\Omega=m r \cdot \Omega. 
\end{equation}
We claim that \eqref{eq-Lxi} implies $A_{\cX_0}(\wt_\beta)\equiv r$ which is independent of $\beta$. In the case where $\beta\in \bQ$ and hence $\xi_\beta$ generates a $\bC^*$-action, this was observed in \cite{Li15a}, which in essence depends on Dolgachev-Pinkham-Demazure's description of normal graded rings as expounded and developed by Koll\'{a}r. For general $\beta\in \bR$ (with $0<\beta\ll 1$), the claim follows from the case where $\beta\in \bQ$ and the fact that we can realize $\wt_\beta$ as quasi-monomial valuations on a common log smooth model. One way to see this fact is to use the same construction as used for $w_\beta$ in Section \ref{sec-v1E} which lead to the Definition \ref{def-wbeta}. The other way, which works for general torus invariant valuations, is to use structure results of normal $T$-varieties from \cite{AH06, AIP11, LS13}. The latter papers indeed generalized Dolgachev-Pinkham-Demazure's construction to varieties with torus actions of higher ranks. For the reader's convenience, we briefly explain how the arguments fit together to work. 
\begin{enumerate}
\item 
In the case where $\beta\in \bQ$ and $0\le \beta\ll 1$, the vector field $\xi_\beta=\xi_0+\beta \eta$ generates a $\bC^*$-action $\lambda_\beta: \bC^*\rightarrow {\rm Aut}(\cX_0)$. The GIT quotient $\cX_0/\lambda_\beta$ is an orbifold (or a stack) $(S_\beta, \Delta_\beta)$ and $\cX_0\setminus\{o'\}$ is a Seifert $\bC^*$-bundle over $(S_\beta, \Delta_\beta)$ (see \cite{Kol04}). There is a natural compactification $\tilde{\cX}_0$ of $\cX_0\setminus \{o'\}$ such that the following conditions hold.
\begin{enumerate}
\item[(i)]
$\tilde{\cX}_0$ is the total space of an orbifold line bundle $L_\beta$ over $(S_\beta, \Delta_\beta)$. In other words, if we denote $A=\bigoplus_{m=0}^{+\infty} H^0(S_\beta, kL_\beta)$ where $L_\beta$ is viewed as a $\bQ$-Weil divisor, then $\cX_0={\rm Spec}_{\bC} (A)$ and $\tilde{\cX}_0={\rm Spec}_{S_\beta} (A)$. Moverover $-(K_{S_\beta}+\Delta_\beta)=r_\beta L_\beta$ for some $r_\beta>0\in \bQ$ (see \cite{Wat81}, \cite[40-42]{Kol04} and \cite[3.1]{Kol13}). 
\item[(ii)]
The vector field $\xi_\beta$ lifts to become a vector field $\xi_\beta$ on $\tilde{\cX}_0$, which generates a $\bC^*$-action and is a real positive multiple of the natural rescaling vector field $\xi_\beta'=u\partial_u$ along the fibre of $\tilde{\cX}_0\rightarrow (S_\beta, \Delta_\beta)$ where $u$ is an affine coordinate along the fibre of $L_\beta\rightarrow S_\beta$. Notice that $\xi'_\beta$ is well-defined, independent of the choice of $u$ and invariant under finite stabilizers over the orbifold locus $\Delta_\beta$.
\item[(iii)]
There is a $\lambda_\beta$-equivariant nowhere vanishing holomorphic section $\Omega'\in H^0(\cX_0, m K_{\cX_0})$ which can be constructed by using the structure of the orbifold line bundle $L_\beta\rightarrow (S_\beta, \Delta_\beta)$ in the same way that $\Omega$ was constructed. By the same calculation as above, we know that $\cL_{\xi_\beta'}\Omega'=m r_\beta \Omega'$. 
\end{enumerate}
From (i) and by using the adjunction formula, it's easy to get $A_{\cX_0}(\wt_{\xi'_\beta})=A_{\cX_0}(\ord_{S_\beta})=r_\beta$ (see \cite[3.1]{Kol13}).   Because both $\Omega'$ and $\Omega$ are equivariant under the $\bC^*$-action $\lambda_\beta$, their ratio $f=\Omega'/\Omega$, which is a nowhere vanishing regular function over $\cX_0$, is also $\lambda_\beta$-equivariant. In other words, there exists $\kappa \in \bZ$ such that $f(t\cdot p)=t^{\kappa}\cdot f(p)$ for any $p\in \cX_0$ and $t\in \bC^*$. Letting $t\rightarrow 0$ and using that $|f(o')|$ is bounded away from $0$ and $+\infty$, we see that $\kappa=0$, i.e. $\Omega'/\Omega$ is $\lambda_\beta$-invariant. So $f$ is a pull back of a regular function that is globally defined on $S_\beta$. This implies that $f$ has to be a constant. So we get from (iii) that $\cL_{\xi'_\beta}\Omega=m r_\beta \Omega$. By (ii) above, $\xi_\beta=a \xi_\beta'$ for some $a\in \bR$. So $\cL_{\xi_\beta} \Omega=a\cdot \cL_{\xi'_\beta}\Omega=a m r_\beta \Omega$. 
Comparing this with \eqref{eq-Lxi}, we see that $a m r_\beta=m r$ and hence $A_{\cX_0}(\wt_{\xi_\beta})=a A_{\cX_0}(\wt_{\xi'_\beta})=a r_\beta= r$. 
\item
If $\beta\not\in \bQ$, then $\wt_{\xi_\beta}$ is quasi-monomial of rational rank 2. Because $\cX_0$ is a normal variety with a $(\bC^*)^2$-action, by \cite[Theorem 3.4]{AH06} there exists a normal semi-projective variety $Y$, a birational morphism $\mu: \tilde{\cX}_0\rightarrow \cX_0$ and a projection $\pi: \tilde{\cX}_0\rightarrow Y$ such that the generic fibre of $\pi$ is a normal toric variety of dimension $2$. This toric variety, denoted by $Z$, is associated to the polyhedral cone $\sigma\subset N\otimes_{\bZ}\bR\cong \bR^2$ where $N={\rm Hom}(\bC^*, (\bC^*)^2)\cong \bZ^2$. Each valuation $\wt_{\xi_\beta}$ (for $0\le \beta\ll 1$) then corresponds to a vector contained in the interior of $\sigma$ (see \cite[11]{AIP11}). Let $\tilde{Z}\rightarrow Z$ be a toric resolution of singularities of $Z$ and choose a Zariski open set $U$ of $Y$ such that the fibre $\pi^{-1}(p)$ of any point $p\in U$ is isomorphic to $Z$. Then there exists a smooth model $\mathfrak{X}\rightarrow \cX_0$ that dominates both $\tilde{\cX}_0$ and $U\times \tilde{Z}$. In fact one can choose $\mathfrak{X}$ to be a toroidal desingularization of $\tilde{\cX}_0$ as constructed in \cite[Section 2]{LS13}. On the model $\mathfrak{X}$ we can realize $\wt_{\xi_\beta}$, for all $0\le \beta\ll 1$, as monomial valuations around the same regular point $p\in Y$ but with different non-negative weights.\footnote{In the case where $Y=\{\rm pt\}$ we are just using toric resolution of singularities. This kind of construction can be considered as a special case of the construction in \cite[Proof of Lemma 3.6.(ii)]{JM10}.} 
By the linearity and hence the continuity of log discrepancy function with respect to the weights (see \cite[5.1]{JM10}), we use (a) and the denseness of rational weights to finally conclude that $A_{\cX_0}(\wt_{\xi_\beta})\equiv r$ for any $\beta\in \bR$ with $0\le \beta\ll 1$. \end{enumerate}

\end{enumerate}

\end{proof}

We provide a simple example illustrating the proof of $A(\wt_\beta)\equiv r$ in Proposition \ref{wwbeta}.
\begin{examp}
Choose $V=\mathbb{CP}^{n-1}$ and $L=n^{-1} K_{V}^{-1}$ the hyperplane bundle over $\bC\bP^{n-1}$. Then $X=C(V, H)=\bC^n$. Consider a holomorphic vector field $\eta$ on $\bC\bP^{n-1}$ that 
is induced by the holomorphic vector field $\sum_{i=0}^{n-1} \lambda_i Z_i\frac{\partial}{\partial Z_i}$ on $\bC^n$ under the quotient $\pi: \bC^n-\{0\}\rightarrow\bC\bP^{n-1}$, where $\lambda_i\in \bZ$ and $Z=[Z_0, \cdots, Z_{n-1}]$ are the homogeneous coordinates of $\bC\bP^{n-1}$ which are also considered as coordinates on $\bC^n$. Using the same notations as in the above proof of item 2 of Proposition \ref{wwbeta} (in particular $m=1$ and $r=n$), we are going to show:
\begin{lem}
$\Omega=n dZ_0\wedge\cdots\wedge dZ_{n-1}$. The canonical lifting of $\eta$ to $\bC^n$ is given by
\[
\eta=\sum_{i=0}^{n-1} (\lambda_i-\bar{\lambda}) Z_i \frac{\partial}{\partial Z_i}
\]
where $\bar{\lambda}=\frac{1}{n}\sum_{p=0}^{n-1}\lambda_p$.
\end{lem}
\begin{proof}
We will do the calculation on the affine chart $U_0=\{Z_0\neq 0\}$. Exactly the same calculation can be done on other $U_i=\{Z_i\neq 0\}$ and the result will turn out to be symmetric with respect to the index $i$. First choose $s_0$ to be the canonical trivializing section of $L^{-1}$ over $U_0=\{Z_0\neq 0\}$ given by $s_0(Z)=(1, Z_1/Z_0, \cdots, Z_{n-1}/Z_0)$.  
%\[
%(Z_0, \cdots, Z_{n-1})=Z_0(1, z_1, \cdots, z_{n-1}).
%\]
 Any point $Z=(Z_i)\in \bC^n-\{0\}$ can be written as:
\[
Z=Z_0(1, Z_1/Z_0, \cdots, Z_{n-1}/Z_0)=Z_0\cdot s_0.
\]
So the globally defined holomorphic $n$-form from \eqref{eq-Omega} is equal to
\begin{equation}
\Omega=d (Z_0^n)\wedge (dz_1\wedge\cdots \wedge d z_{n-1})=n dZ_0\wedge dZ_1\wedge \cdots \wedge d Z_{n-1}.
\end{equation}
On the other hand, the natural trivializing section of the holomorphic line bundle $K_V$ over $U_0$ is
$dz=d {z_1}\wedge \cdots \wedge d {z_{n-1}}$. Under the isomorphism $K_V=n L^{-1}$, we have $s_0^{\otimes n}=dz$ on $U_0$ (by looking at the transition functions).

In the affine coordinate chart $U_0=\{Z_0\neq 0\}$, we have $\eta=\sum_{i=1}^{n-1}(\lambda_i-\lambda_0)z_i\frac{\partial}{\partial z_i}$ where $\{z_i=Z_i/Z_0\}$ are inhomogeneous coordinates on $U_0\cong \bC^{n-1}$. So we have $t^* dz=t^{\left(\sum_{i=1}^{n-1}(\lambda_i-\lambda_0)\right)} dz$ and hence $t\circ s_0=t^{\left(\sum_{i=1}^{n-1}(\lambda_i-\lambda_0)/n\right)}s_0$ for any 
$t\in e^{\bC\eta}\cong \bC^*$. So the canonical lifting of $\eta$ on $\pi^{-1}(U_0)$, still denoted by $\eta$, is equal to (under the coordinate 
$\{Z_0, z_1, \cdots, z_{n-1}\}$):
\begin{equation}\label{eq-eta1}
\eta=-\frac{1}{n} \sum_{i=1}^{n-1}(\lambda_i-\lambda_0) Z_0 \frac{\partial}{\partial Z_0}+ \sum_{i=1}^{n-1}(\lambda_i-\lambda_0)z_i\frac{\partial}{\partial z_i}.
\end{equation}
Now the coordinate change from $\{Z_0, z_1, \cdots, z_{n-1}\}$ to $\{Z_0, Z_1, \cdots, Z_{n-1}\}$ is given by:
\begin{equation*}
Z_0=Z_0, \quad Z_i=z_i\cdot Z_0,\quad i=1, \cdots, n-1; 
\end{equation*}
which implies the change of basis formula:
\begin{equation*}
 \frac{\partial}{\partial Z_0}=\frac{\partial}{\partial Z_0}+\sum_{j=1}^{n-1}\frac{Z_j}{Z_0}\frac{\partial}{\partial Z_j}, \quad \frac{\partial}{\partial z_i}=Z_0 \frac{\partial}{\partial Z_i}, \quad i=1,\cdots, n-1.
\end{equation*}
Substituting these into \eqref{eq-eta1} we get:
\begin{eqnarray*}
\eta&=& \frac{1}{n} \sum_{i=1}^{n-1} (\lambda_0-\lambda_i) Z_0\left(\frac{\partial}{\partial Z_0}+\sum_{j=1}^{n-1} \frac{Z_j}{Z_0}\frac{\partial}{\partial Z_j}\right)+\sum_{i=1}^{n-1}(\lambda_i-\lambda_0)\frac{Z_i}{Z_0}Z_0\frac{\partial}{\partial Z_i}\\
&=&\sum_{p=0}^{n-1}(\lambda_p-\bar{\lambda})Z_p\frac{\partial}{\partial Z_p}.
\end{eqnarray*}
Indeed, this follows from easy manipulation:
\[
\frac{1}{n}\sum_{i=1}^{n-1}(\lambda_0-\lambda_i)=\lambda_0-\bar{\lambda}, \quad \left(\frac{1}{n}\sum_{j=1}^{n-1}(\lambda_0-\lambda_j)\right)+(\lambda_i-\lambda_0)=\lambda_i-\bar{\lambda}.
\]
\end{proof}
We let $\xi_0=\sum_{i=0}^{n-1}Z_i \frac{\partial}{\partial Z_i}$ as in the proof of Proposition \ref{wwbeta}. Then $\xi_\beta=\sum_{i=0}^{n-1}(1+\beta(\lambda_i-\bar{\lambda}))Z_i\frac{\partial}{\partial Z_i}$. It's clear that, for $0\le \beta\ll 1$, the valuation $\wt_\beta$ is the monomial valuation centered at $0\in \bC^n$  and its log discrepancy is given by:
\[
A_{\bC^n}(\wt_\beta)=\sum_{i=0}^{n-1} (1+\beta(\lambda_i-\bar{\lambda}))\equiv n.
\]
\end{examp}

\subsection{Completion of the proof of Theorem \ref{mv2ks}}
Now we can finish the proof of Theorem \ref{mv2ks}. Indeed, it's well known now that the derivative of
$\hvol(\wt_\beta)=r^n \vol(\wt_\beta)$ on $\cX_0$ is given by a positive multiple of the Futaki invariant on $\cV_0$ defined in \cite{Fut83, DT92}, or equivalently the CM weight in \eqref{CMstc}. This was first discovered in \cite{MSY08} when $V$ is smooth, and generalized to the mildly singular case in \cite{DS15} (see also \cite{CS12}). For the convenience of the reader, we give a different proof in Lemma \ref{lemMSY} following the spirit of calculations in previous sections and using results on Duistermaat-Heckman measures in \cite{BC11, BHJ15}.\footnote{The generalized convexity of the volume in Lemma \ref{lemMSY} has also been notice by Collins-Sz\'{e}kelyhidi in \cite[Page 6]{CS15}.}

Under the assumption that $\hvol$ obtains the minimum at $\ord_{V}$ on $(X_*, o_*)$, by the above Proposition and Corollary \ref{corww*}, we have
\[
\hvol(\wt_\beta)=\hvol(w_\beta)=\sigma^{-1}\hvol(w_{*\beta})\ge\sigma^{-1}\hvol(w_{*0})=\hvol(w_0)=\hvol(\wt_0).
\]
for any $0\le \beta\ll1$. So we must have $\left.\frac{d}{d\beta}\hvol(\wt_\beta)\right|_{\beta=0}\ge 0$ which is equivalent to ${\rm CM}(\cV, \cL) \ge 0$. Since
this non-negativity  holds for any special test configuration $(\cV, \cL)$, we are done.

\begin{lem}\label{lemMSY}
The following properties hold. 
\begin{enumerate}
\item The derivative of $\hvol(\wt_\beta)=r^n \cdot \vol(\wt_\beta)$ at $\beta=0$ is a positive multiple of the Futaki invariant on the central fibre $\cV_0$ for the $T_1$-action.
More precisely, we have the equality:
\begin{equation}\label{derCM}
\left.\frac{d}{d\beta}\hvol(\wt_\beta)\right|_{\beta=0}=-\left(-K_{\bar{\cV}/\bP^1}\right)^{n}=(n (-K_V)^{n-1})\cdot {\CM}(\cV,\cL).
\end{equation}

\item For any $\xi_1\in \ft_{\bR}^{+}$ (see \eqref{eq-conepos}) with associated weight function $\wt_1:=\wt_{\xi_1}$, $\Phi(t)=\vol((1-t)\wt_0+t\; \wt_1)$ is convex with respect to $t\in [0,1]$.
\end{enumerate}
\end{lem}
\begin{proof}
First notice that, by definition, $\wt_\beta=\wt_0+\beta\wt^{(1)}$ where $\wt_0=\ord_{\cV_0}$ is the canonical valuation. So we see that
\[
\fa_{\lambda}(\wt_\beta)\bigcap S_m=\cF^\lambda_{\wt_\beta} S_m=\cF^{(\lambda-m)/\beta}_{\wt^{(1)}} H^0(\cV_0, m\cL_0).
\]
%When $0\le \beta\ll 1$, $\wt_\beta$ has positive weights on $H^0(\cV_0, m\cL_0)$. In other words, $\cF^x_{\wt_\beta}=0$ if $x\le 0$.
Similar as before, we define $c_3:=\inf_{\fm}\wt^{(1)}/\wt_0>-\infty$ and calculate:
\begin{eqnarray*}
&&n!\dim_{\bC} S/\fa_\lambda(\wt_\beta)%&=&\sum_{m=0}^{+\infty} \dim_{\bC} S_m-\dim_{\bC} \cF^{(x-m)/\beta}_{\wt^{(1)}}S_m\\
=n!\sum_{m=0}^{\lfloor \lambda/(1+c_3\beta)\rfloor}\left( \dim_{\bC} S_m-\dim_{\bC} \cF^{(\lambda-m)/\beta}_{\wt^{(1)}}S_m\right)\\
%&=&\frac{m^n}{(1+\beta c_3)^n}L^{n-1}-\sum_{i=0}^{\infty}\vol\left(\cF^{\frac{m-i}{\alpha i}}\right) i^{n-1}\\
&=&\frac{\lambda^n\cL_0^{n-1}}{(1+\beta c_3)^n}-n \lambda^n \int^{+\infty}_{c_3}\vol\left(\cF_{\wt^{(1)}}S^{(x)}\right)\frac{\beta dx}{(1+\beta  x)^{n+1}}+O(\lambda^{n-1}).
\end{eqnarray*}
For simplicity of notations, in the rest part of the argument we will denote $\cF_{\wt^{(1)}}S^{(x)}$ simply by $\cF_{\wt^{(1)}}^{(x)}$.
This gives the following formula (compare \eqref{eqvolvt1} and \eqref{eqvolvt2}):
\begin{eqnarray}
\vol(\wt_\beta)&=&\frac{\cL_0^{n-1}}{(1+\beta c_3)^n}-n\int^{+\infty}_{c_3}\vol\left(\cF_{\wt^{(1)}}^{(x)}\right)\frac{\beta dx}{(1+\beta x)^{n+1}}\\
&=&-\int^{+\infty}_{c_3}\frac{d \vol\left(\cF_{\wt^{(1)}}^{(x)}\right)}{(1+\beta x)^n}.
\end{eqnarray}
The second identity follows from integration by parts as in \eqref{eqvolvt2}. So the normalized volume is equal to (compare \eqref{eqvolvt1}):
\begin{eqnarray*}
\hvol(\wt_\beta)&=&r^n\left(\frac{\cL_0^{n-1}}{(1+\beta c_3)^n}-n\beta\int^{+\infty}_{c_3}\vol\left(\cF_{\wt^{(1)}}^{(x)}\right)\frac{dx}{(1+\beta x)^{n+1}}\right)=:\Phi(\beta).
\end{eqnarray*}
Again $\Phi(0)=r^n \cL_0^{n-1}=\hvol(\wt_0)$ and its derivative at $\beta=0$ is equal to:
\begin{eqnarray*}
\Phi'(0)&=& -n c_3 r^n \cL_0^{n-1} -n r^n  \int_{c_3}^{+\infty} \vol\left(\cF_{\wt^{(1)}}^{(x)}\right)d x\\
%&=&-n c_3 r^n \cL_0^{n-1}-n r^n \int_{c_3}^{+\infty}\vol(\cF^{x}_{\wt^{(1)}})dx\\
\end{eqnarray*}
Using integration by parts and noticing that $\vol(\cF_{\wt^{(1)}}^{(x)})=\cL_0^{n-1}$ if $x\le c_3$, we get:
\begin{eqnarray*}
\Phi'(0)&=&-n c_3 r^n \cL_0^{n-1}-n r^n \lim_{M\rightarrow +\infty}\left.x \vol(\cF^{(x)}_{\wt^{(1)}})\right|_{c_3}^{M}+n r^n \int_{c_3}^{+\infty} x d\vol(\cF^{(x)}_{\wt^{(1)}})\\
&=&n r^n \int_{c_3}^{+\infty} x \cdot d\vol\left(\cF^{(x)}_{\wt^{(1)}}\right).
\end{eqnarray*}
Now by \cite{BC11} (see \cite[Theorem 5.3]{BHJ15}), we know that
\begin{equation}
\int_{c_3}^{+\infty}x\cdot d\vol\left(\cF^{(x)}_{\wt^{(1)}}\right)=-\lim_{m \rightarrow+\infty}\frac{w_m}{m N_m}=-\frac{\bar{\cL}^n}{n},
\end{equation}
where $w_m$ is the weight of $T_1$-action on $H^0(\cV_0, m\cL_0)$ and $N_m=\dim_{\bC}H^0(\cV_0, m\cL_0)$. 
So $\frac{\Phi'(0)}{n (-K_V)^{n-1}}=-\frac{r^n\bar{\cL}^{n}}{n (-K_V)^{n-1}}$ is nothing but the Futaki invariant on $\cV_0$ or equivalently the CM weight in \eqref{CMstc}. 

To see the second item, it's clear that the same calculation gives the formula as in \eqref{eqvolvt2}:
\[
\vol((1-t)\wt_0+t\wt_1)=\int^{+\infty}_{c_1}\frac{-d\vol\left(\cF_{\wt_1}S^{(x)}\right)}{((1-t)+t x)^{n}},
\]
where $c_1=\inf_{\fm}\wt_1/\wt_0$. Because $\xi_1\in \ft_{\bR}^{+}$, $\wt_1\in \left(\Val_{\cX_0,o'}\right)^{\bC^*}$ (see Section \ref{sec-wtbeta}). In particular $c_1>0$. The convexity with respect to $t$ follows from the fact that $ t \mapsto 1/((1-t)+tx)^n$ is convex for $t\in [0,1]$ (cf. proof of Lemma \ref{lemintfn}).
\end{proof}
%Notice that $v_{\cV_0}(f):=\ord_{\cV_0}(\bar{f})$ is an example of valuations studied in \cite{BHJ15}. However,
%here we consider it as a valuation on the coordinate ring $R$ of the cone $X=C(V,L)$. By \cite[Proposition 4.11]{BHJ15}, $v_{\cV_0}$ is divisorial or trivial. $w_{\beta}$ is nothing but a linear interpolation of $v_0$ and $v_{\cV_0}$.

\begin{rem}
As pointed by a referee, one could simplify the proof of Theorem \ref{mv2ks} by directly calculating the normalized volume $w_\beta$ and its derivative at $\beta=0$ by using the Duistermaat-Heckman measure of $\cF^x_{v^{(1)}_E} R$ without the help of $\wt_\beta$. Indeed, based on the results in Section \ref{sec-v1E} and by the same calculation as in the proof of Lemma \ref{lemMSY} (see \eqref{eq-volvalpha} and \eqref{eq-hvolvalpha}), we can get the following formula of $\hvol(w_\beta)$:
\[
\hvol(w_\beta)=r^n\left(\frac{L^n}{(1+\beta a_1)^n}-n \int^{+\infty}_{a_1} \vol\left(\cF^{(x)}_{v^{(1)}_E}\right)\frac{\beta dx}{(1+\beta x)^{n+1}}\right)
\]
where $a_1=-\frac{A_{V\times\bC}(\cV_0)-1}{r}=-q\cdot A_V(F)/r$ (see \eqref{defc2} and Remark \ref{rema1}). Taking derivative on both sides at $\beta=0$ and then integrating by parts, we get:
\begin{eqnarray*}
\left.\frac{d}{d\beta}\right|_{\beta=0}\hvol(w_\beta)&=&- n r^n a_1 L^n -n r^n \int^{+\infty}_{a_1}\vol\left(\cF^{(x)}_{v^{(1)}_E}R\right)dx\\
&=& n r^n \int^{+\infty}_{a_1} x \cdot d \vol\left(\cF^{(x)}_{v^{(1)}_E} R\right).
\end{eqnarray*}
Now applying directly \cite[Proposition 3.12]{BHJ15} gives: 
\[
\left.\frac{d}{d\beta}\right|_{\beta=0}\hvol(w_\beta)=-r^n \bar{\cL}^n= n (-K_V)^{n-1}\cdot \CM(\cV, \cL).
\]
Arguing in the same way as at the beginning of this subsection, we then complete the proof of Theorem \ref{mv2ks}.
On the other hand, we keep the original argument because the correspondence between $w_\beta$ and $\wt_\beta$ is natural and will be important for our future work. 
\end{rem}

\subsection{An example of calculation}

Here we provide an example to illustrate our calculations. This is related to the example discussed in \cite[Section 5]{Li13}. We let $V$ be an $(n-1)$-dimensional Fano manifold and assume that $D$ is a smooth prime divisor
such that $D\sim_{\bQ}-\lambda K_V$ with $0<\lambda<1$. Then we can construct a special degeneration of $V$ by using deformation to the normal cone in the following way. First denote by 
$\Pi_1: \cW:=Bl_{D\times\{0\}}(V\times\bC)\rightarrow V\times\bC$ the blow up of $V\times\bC$ along the codimension 2 subvariety $D\times\{0\}$. Then $\cW\rightarrow \bC$ is a flat family and the central fibre $\cW_0$ is the union of two components $\hV\cup \hE$. Here the $\hV$ component is the strict
transform of $V\times \{0\}$ which isomorphic to $V$ because $D\times \{0\}$ is of codimension one inside $V\times \{0\}$. $\hE$ denotes the exceptional divisor, which in this case is nothing but $\bP(N_D\oplus\bC)$, where $N_D$ is the normal bundle of $D\subset V$. For any $c>0\in \bQ$, there is a $\bQ$-line bundle $\mathcal{M}_c:=\Pi_1^*(-K_{V\times\bC})-c\hE$ on $\cW$. $\mathcal{M}_c$ is relatively ample over $\bC$ if and only if $c\in (0, \lambda^{-1})$. Moreover $\mathcal{M}_{\lambda^{-1}}$ is semi ample over $\bC$ and the linear system $|-p \mathcal{M}_{\lambda^{-1}}|$ for sufficiently divisible $p$ gives a birational morphism $\Pi_2: \cW\rightarrow \cV$ by contracting the component $V$ in the central fibre, and we have $\mathcal{M}_{\lambda^{-1}}^{p}\sim_{\bQ} \Pi_2^*(-pK_{\cV/\bC})$ over $\bC$ (see \eqref{eqld2}). We choose a $p \in \bZ_{>0}$ sufficiently divisible such that $\cL:=-p K_{\cV/\bC}$ is Cartier. Then the pair $(\cV, \cL)$ is a special degeneration of $(V, L)$ where $L= -p K_V$. Notice that the central fibre $\cV_0$, also denoted by $E$, is obtained from $\hE=\bP(N_D\oplus\bC)$ by contracting the infinity section $D_\infty$ of the $\bP^1$-bundle and hence in general has an isolated singularity. 

The divisorial valuation $v_E$ on $\bC(V)$ coincides with the divisorial valuation $\ord_D$. In other words, by using the notation of \eqref{ordF}, we have $F=D$ and $q=1$.  
The valuation $v^{(0)}_E$ on $X=C(V, L)$ is the divisorial valuation $\ord_{\bD}$ where $\bD$ is 
the closure of $\tau^{*}D$ where $\tau: L\setminus V\rightarrow V$ is the natural $\bC^*$-bundle. By Lemma \ref{ldbF}, $A_X(\bD)=A_V(D)=1$. Because 
$K_{\cW}=\Pi_1^* K_{V\times \bC}+\hE$ and $r=p^{-1}$, we have (see \eqref{defc2})
\[
A_{V\times\bC}(E)-1=1=A_V(v_E), \quad a_1:=-\frac{A_{V\times\bC}(E)-1}{r}=-p.
\]
\begin{lem}
We have the identity of $\bQ$-Cartier divisors:
\begin{equation}\label{eqld}
\Pi_2^*(-K_{\cV})=\Pi_1^*K_{V\times\bC}^{-1}-\hE+(\lambda^{-1}-1)\hV.
\end{equation}
\end{lem}
\begin{proof}
We have the identities:
\[
K_{\cW}=\Pi_2^*K_{\cV}+b \hV=\Pi_1^*K_{V\times \bC}+\hE.
\]
To determine the coefficient $b$, we use the adjunction formula:
\begin{eqnarray*}
K_{\hV}&=&(K_{\cW}+\hV)|_{\hV}=\Pi_2^*K_{\cV}|_{\hV}+(b+1)\hV|_{\hV}=(b+1)\hV|_{\hV}\\
&=&(b+1)(\hV+\hE)|_V-(b+1)\hE|_{\hV}=-(b+1)\hE|_{\hV}\\
&=&(b+1)\cO_{\hV}(-D_\infty)=\lambda(b+1)K_{\hV}.
\end{eqnarray*}
In the above we used $\cO_{\cV}(\hV+\hE)=\cO_{\cV}(\cW_0)$ is trivial over $\hV$. So $b=\lambda^{-1}-1$ and we obtain the equality.
\end{proof}
\begin{rem}
Since $\cW_0=\hV+\hE$, the equality \eqref{eqld} can also be written in a different form:
\begin{eqnarray}\label{eqld2}
\Pi_2^*(-K_\cV)&=&\Pi_1^*K_{V\times\bC}^{-1}-\lambda^{-1}\hE+(\lambda^{-1}-1)(\hV+\hE)\nonumber\\
&=&\mathcal{M}_{\lambda^{-1}}+(\lambda^{-1}-1)\cW_0.
\end{eqnarray}
\end{rem}
%For any $f\in H^0(V, kL)$, we have $
%\Pi_1^*f \in H^0(\cW,  k\Pi_1^* (-pK_V))$. If we denote by $s_{\hE}$ the defining section of $\cO_{\cW}(\hE)$, then we get a holomorphic section of $kp \mathcal{M}_{\lambda^{-1}}=kp(\Pi_1^*(-K_V)-\lambda^{-1} \hE)$:
%\[
%g:=\Pi_1^*f\cdot  s_{\hE}^{-k p \lambda^{-1}}.
%\]
%$g$ descends to become a holomorphic section $\bar{f}$ of $-kp K_{\cV}$. So we get:
%\[
%\ord_{E}(\bar{f})=v^{(0)}_E (f) - k p\lambda^{-1}=v^{(0)}_E(f)+c_2 k.
%\]
%\[
%\Pi_2^*(-p K_{\cV})%=\mathcal{M}^p_{\lambda^{-1}}
%=p\left(\Pi_1^*(K_{X\times\bC}^{-1})-\lambda^{-1}\hE\right)+p(\lambda^{-1}-1)\cW_0.
%\]
%\begin{eqnarray*}
%\Pi_2^*K_{\cV}&=&\Pi_1^*K_{X\times \bC}+\hE-b\hV=(\Pi_1^*K_{X\times\bC}+\lambda^{-1} E)-(b\hV+(\lambda^{-1}-1)\hE)\\
%&=&(\Pi_1^*K_{X\times\bC}+\lambda^{-1}E)-(\lambda^{-1}-1)\cW_0.
%\end{eqnarray*}
%\begin{equation}
%\Pi_2^* K_{\cV}^{-1}=\Pi_1^*(-K_{X\times\bC}-\lambda^{-1}E)+(\lambda^{-1}-1)\cW_0.
%\end{equation}
As above, let $w_\beta$ be the quasi-monomial valuation of $V+\bD$ with weight $(1+\beta a_1, \beta q)=(1-\beta p, \beta)$. We can calculate the volume of $w_\beta$ by applying the formula \eqref{volv1a} to get ($c_1=1+\beta a_1=1-\beta p$):
\begin{eqnarray}
\vol(w_\beta)&=&\frac{1}{(1-\beta p)^n}L^{n-1}- n\int_{1-\beta p}^{+\infty}\vol\left(\cF_{w_\beta} R^{(t)}\right)\frac{dt}{t^{n+1}}\nonumber.
\end{eqnarray}
To calculate $\vol\left(R^{(t)}\right)$, we notice that:
\[
\cF_{w_\beta}^x R_k=\cF^{(x-(1-\beta p) k)/\beta}_{v^{(0)}_E} R_k,
\]
so that we have the identity:
\begin{equation}\label{eq2vol}
\vol\left(\cF_{w_\beta} R^{(t)}\right)=\vol\left(\cF_{v^{(0)}_E}R^{\left(\frac{t-(1-\beta p)}{\beta}\right)}\right).
\end{equation}
So by changing of variables, we get another formula 
\begin{equation}\label{volbe2}
\vol(w_\beta)=\frac{L^{n-1}}{(1-\beta p)^n}-n\int^{+\infty}_0 \vol\left(\cF_{v^{(0)}_E}R^{(s)}\right)\frac{\beta ds}{(1-\beta p+\beta s)^{n+1}}.
\end{equation}
It's clear that the filtration defined by $v^{(0)}_E=\ord_D$ is given by
\[
\cF^{y}_{v^{(0)}_E} R_k=H^0(V, \cO_V(kL-\lceil y\rceil D))=H^0\left(V, \cO_V((k-\lambda p^{-1}\lceil y\rceil )L)\right).
\]
So it's straight-forward to calculate that:
\begin{equation}\label{volv0E}
\vol\left(\cF_{v^{(0)}_E}R^{(s)}\right)=
\left\{
\begin{array}{ll}
(1-\lambda p^{-1} s)^{n-1}L^{n-1} & \text{ if } s<p\lambda^{-1}; \\
0 & \text{ if } s\ge p\lambda^{-1}.
\end{array}
\right.
\end{equation}
Substituting \eqref{eq2vol} into \eqref{volbe2}, we get the volume of $w_\beta$. We can also calculate the derivative:
\begin{eqnarray*}
\left.\frac{d}{d\beta}\right|_{\beta=0}\vol(w_\beta)&=&pn L^{n-1}-n \int^{+\infty}_0\vol\left(\cF_{v^{(0)}_E}R^{(s)}\right)ds\\
&=&pn L^{n-1}-n L^{n-1}\int^{p\lambda^{-1}}_0(1-\lambda p^{-1} s)^{n-1} ds\\
&=&p(n-\lambda^{-1})L^{n-1}\ge 0.
\end{eqnarray*}
For the last inequality, we used the bound of Fano index. 
Indeed since $K_V^{-1}=\lambda^{-1} D$, it's well know that $\lambda^{-1}\le \dim V+1=n$. So $\vol(w_\beta)$ is increasing with respect to $\beta$
in this example. 
On the other hand, noticing that $\hV|_\hV=(\hV+\hE)|_\hV-\hE|_\hV=\cO_\hV(-D)=\lambda K_V$, it's an exercise to obtain:
\begin{eqnarray*}
(-K_{\bar{\cV}/\bP^1})^n&=&(\Pi_2^*(-K_{\bar{\cV}/\bP^1}))^n=\left(\Pi_1^* K_{V}^{-1}+\lambda^{-1}\hV-\cW_0 \right)^{n}\\
%&=&(\Pi_1^*K_{V\times\bP^1}^{-1}-\lambda^{-1}\hE)^n+n (\Pi_1^*K_{V\times\bP^1}^{-1}-\lambda^{-1}\hE)^{n-1}\cdot(\lambda^{-1}-1)\cW_0\\
%&=&\lambda^{-n}(-\hE)^n+n \lambda^{-(n-1)}(-\hE)^{n-1} \cdot (\lambda^{-1}-1)\hE
&=&(\lambda^{-1} - n)  \langle K_V^{-(n-1)}, [V]\rangle=p^{1-n}(\lambda^{-1}-n)\cdot L^{n-1}.
\end{eqnarray*}
So we obtain an identity illustrating the general formula \eqref{derCM}:
\begin{eqnarray}
\left.\frac{d}{d\beta}\right|_{\beta=0}\hvol(w_\beta)&=&p^{-n+1}(n-\lambda^{-1})L^{n-1}=-(-K_{\cV/\bP^1})^n.
\end{eqnarray}

\section{Proof of Theorem \ref{thmdiv}}\label{sec-thmdiv}

%From the above proof of Theorem \ref{main}, we see that to test K-semistability. We just need to compare the normalized volume of quasi-monomial valuations of $(V, \bF)$ where $\bF$ corresponds to a prime
%divisor $F$ over $V$. Using our volume formula, it's easy to calculate the normalized volumes and compare. 

For any divisorial valuation $\ord_F$ of $\bC(V)$, we consider the filtration:
\begin{equation}\label{divfil}
\cF_{F}^x R_k:=\{f\in H^0(V, kL)\; |\; \ord_F(f)\ge x\}.
\end{equation}
We know that this is a multiplicative, left-continuous and linearly bounded filtration by \cite{BKMS15} (see also \cite[Section 5.4]{BHJ15}). Similar to Lemma \ref{mulvol}, we have the following 
observation. 
\begin{lem}\label{intresc}
If $L=\sigma L_*$ and we denote $R_*=\bigoplus_{k=0}^{+\infty} H^0(V, kL_*)=:\bigoplus_{k=0}^{+\infty}R_{*k}$, then 
$\cF^x_F R_k=\cF^x_F R_{*\sigma k}$. As a consequence 
\[
\vol\left(\cF_F R^{(x)}\right)=\sigma^{n-1}\cdot \vol\left(\cF_F R_*^{(x/\sigma)}\right),
\]
and by change of variables we have:
\[
r^n\int_0^{+\infty}\vol\left(\cF_F R^{(x)}\right)dx=(r\sigma)^n \int_0^{+\infty}\vol\left(\cF_F R_*^{(x)}\right)dx.
\]
\end{lem}
To put Theorem \ref{thmdiv} in the setting of our discussions, we adopt the notation in Section \ref{secvol2semi} and consider the quasi-monomial $\bC^*$-invariant valuation $v_\alpha$ of $(V, \bF)$ with weight $(1, \alpha)$, similar to what we did in \cite{Li15a}. Then we have:
\begin{itemize}
\item $\inf_{\fm}\frac{v_\alpha}{v_0}=1$;
\item $\cF_{v_\alpha}^x R_k=\cF_F^{(x-k)/\alpha}R_k$. As a consequence:
\[
\vol\left(\cF_{v_\alpha} R^{(x)}\right)=\vol\left(\cF_FR^{(\frac{x-1}{\alpha})}\right).
\]
\end{itemize}
We can use the volume formula \eqref{volv1a} to get:
\begin{eqnarray}\label{eq-volvalpha}
\vol(v_\alpha)&=&L^{n-1}-n\int^{+\infty}_1 \vol\left(\cF_F R^{(\frac{x-1}{\alpha})}\right)\frac{dx}{x^{n+1}}\nonumber\\
&=&L^{n-1}-n\int^{+\infty}_0\vol\left(\cF_F R^{(t)}\right)\frac{\alpha dt}{(1+\alpha t)^{n+1}}.
\end{eqnarray}
The log discrepancy of $v_\alpha$ is given by
\[
A_X(v_\alpha)=r+\alpha A_X(\ord_{\bF})=r+\alpha A_V(F),
\]
where the last identity is by Lemma \ref{ldbF}. So we get the normalized volume of $v_\alpha$:
\begin{equation}\label{eq-hvolvalpha}
\hvol(v_\alpha)=(r+\alpha A_V(F))^n L^{n-1}- n A_X(v_\alpha)^n \int^{+\infty}_0 \vol\left(\cF_F R^{(t)}\right)\frac{\alpha dt}{(1+\alpha t)^{n+1}}.
\end{equation}
The derivative at $\alpha=0$ is equal to:
\begin{eqnarray}\label{eq-dvolalpha}
\left.\frac{d}{d\alpha}\right|_{\alpha=0}\hvol(v_\alpha)&=&n r^{n-1} A_V(F) L^{n-1}-n r^n \int^{+\infty}_0 \vol\left(\cF_F R^{(t)}\right)dt\nonumber\\
&=&n(-K_V)^{n-1}\left(A_V(F)-\frac{r^n}{(-K_V)^{n-1}}\int^{+\infty}_0\vol\left(\cF_F R^{(t)}\right)dt\right)\nonumber\\
&=&n(-K_V)^{n-1}\Theta_V(F).
\end{eqnarray}
We can apply these formula for $F$ coming  from any nontrivial special test configuration as considered in Section \ref{secvol2semi}. Using the notations in Section \ref{secvol2semi}, the valuation $w_\beta$ is associated to the divisor:
\[
(1+\beta a_1)V+ \beta q \bF=\frac{1}{1+\beta a_1}\left(V+\frac{\beta q}{1+\beta a_1} \bF\right).%=\frac{1}{1+\beta a_1}v_\alpha,
\]
So $w_\beta=v_\alpha/(1+\beta a_1)$ with $\alpha=\frac{\beta q}{1+\beta a_1}$. So we get $\hvol(w_\beta)=\hvol(v_\alpha)$ and 
\[
\left.\frac{d}{d\beta}\right|_{\beta=0}\hvol(w_\beta)=q \left.\frac{d}{d\alpha}\right|_{\alpha=0}\hvol(v_\alpha)=n(-K_V)^{n-1}\Theta_V(F).
\]
Note that $q>0$. By the proof in Section \ref{secvol2semi}, we know that $\Theta_V(F)$ is the Futaki invariant (or CM weight) of the special test configuration. 
\begin{rem}
For any special test configuration, one can also prove that $\Theta_V(F)$ is equal to the Futaki invariant by applying integration by parts to the formula in \cite[Theorem 3.12.(iii)]{BHJ15}. 
%and using Lemma \ref{v01form}.
\end{rem}
\begin{proof}[Proof of Theorem \ref{thmdiv}]
If $V$ is K-semistable or equivalently Ding-semistable, then by Theorem \ref{ks2mv} $\hvol(v_\alpha)\ge \hvol(v_0)$ for $0\le \alpha\ll 1$. So $\Theta_V(F)\ge 0$ by \eqref{eq-dvolalpha}. Actually the proof of Theorem \ref{ks2mv} in Section \ref{secC*inv} shows that one can apply Fujita's result directly for $\cF_F R_\bullet$ to get $\Theta_V(F)\ge 0$ (see also \cite{LL16}). 

Reversely, if $\Theta_V(F)\ge 0$ (resp. $>0$) for any divisorial valuation $\ord_F$ over $V$, then this holds for $F$ coming from any non-trivial special test configuration. By the above discussion this exactly means that the Futaki invariant of any nontrivial special test configuration is nonnegative (resp. positive). So $V$ is indeed K-semistable (resp. K-stable) in the sense of Tian (\cite{Tia97}). By using \cite{LX14} we conclude that $V$ is K-semistable (resp. K-stable) as in definition \ref{defKstable}.
\end{proof}

\section{Two examples of normalized volumes}\label{sec-exnv}

\begin{enumerate}
\item
Consider the quasi monomial valuation $v$ on $\bC(x,y)$ determined by $\gamma\in \bR^2_{>0}$:
\[
v\left(\sum_{e_1,e_2\ge 0} p_{e_1, e_2} x^{e_1}y^{e_2}\right)=\min\{e_1 \gamma_1+e_2\gamma_2 | p_{e_1,e_2}\neq 0 \}.
\] 
For any $m\in \bZ_{>0}$, $\dim_{\bC}(\bC[x,y]/\fa_m(v))$ is the number of positive integral solutions to the inequalities:
\[
\gamma_1 a+\gamma_2 b<m, a\ge 0, b\ge 0. 
\]
So it's easy to see that 
\[
\vol(v)=\lim_{m\rightarrow+\infty} \frac{\dim_{\bC}(\bC[x,y]/\fa_m(v))}{m^2/2!}=\frac{1}{\gamma_1 \gamma_2}.
\] 
On the other hand, the log discrepancy $A_{\bC^2}(v)=\gamma_1+\gamma_2$ so that
\[
\hvol(v)=\frac{(\gamma_1+\gamma_2)^2}{\gamma_1\gamma_2}\ge 2^2.
\]
The last equality holds if and only if $\gamma_1=\gamma_2$.

\item ({\bf Non-Abhyankar examples which are not $\bC^*$-invariant})
Here we calculate the normalized volumes of non-Abhyankar valuations in \cite[Example 3.15]{ELS03}. These valuations were constructed in \cite[pp. 102-104]{ZS60} and we follow the exposition in \cite[Example 3.15]{ELS03} to recall their constructions. Let $v(y)=1$ and set $v(x)=\beta_1>1$ some rational number (compared to \cite{ELS03}, we shift the subindex by $1$). Let $c_1$ 
be the smallest positive integer such that $c_1\beta_1 \in \bZ$. By \cite{ZS60}, there exists a valuation $v$ such that the polynomial
\[
q_2=x^{c_1}+y^{\beta_1 c_1}
\]
has value $\beta_2$ equal to any rational number greater than (or equal to) $\beta_1 c_1$. We choose this value so that $\beta_2=\frac{b_2}{c_2}>\beta_1 c_1$, where $b_2$ and $c_2$ are relatively 
prime positive integers with $c_2$ relatively prime to $c_1$. 
\begin{rem}
Notice that this already implies that $v$ is not $\bC^*$-invariant because for general $t\in \bC^*$, we have
\begin{eqnarray*}
v(t^{c_1} x^{c_1}+t^{\beta_1 c_1}y^{\beta_1 c_1})&=&v(t^{c_1} (x^{c_1}+y^{\beta_1 c_1})+(t^{\beta_1 c_1}-t^{c_1})y^{\beta_1 c_1})\\
&=&v(y^{\beta_1 c_1})=\beta_1 c_1\neq \beta_2=v(x^{c_1}+y^{\beta_1 c_1}).
\end{eqnarray*}
\end{rem}
This process can be repeated and in this way, we inductively construct a sequence of polynomials $q_i$, rational numbers $\beta_i=\frac{b_i}{c_i}$ such that:
\begin{enumerate}
\item
$c_i>1$ and $b_i$ are relatively prime positive integers with $c_i$ relatively prime to the product $c_1 c_2\cdots c_{i-1}$;
\item 
$q_{i+1}=q_i^{c_i}+y^{\beta_i c_i}$;
\item
$\beta_{i+1}>\beta_i c_i$.
\end{enumerate}
As shown in \cite{ZS60}, this uniquely defines a valuation $v$ on $\bC(x,y)$ such that $v(q_i)=\beta_i$ in the following way. Using Euclidian algorithm with respect to the power of $x$, and
setting $q_0=y$ and $q_1=x$, every polynomial has a unique expression as a sum of ``monomials" in the $q_i$:
\[
q^{\bf a}:= q_0^{a_0}q_1^{a_1}q_2^{a_2}\cdots q_k^{a_k},
\]
where $a_0$ can be arbitrary but the remaining exponents $a_i$ satisfy $a_i<c_i$. By the pair wisely prime properties of $c_i$, it's easy to see that these ``monomials" have distinct
values, i.e. $v(q^{\bf a})=a_0+\sum_{i=1}^k a_i \beta_i$ are all distinct. So the valuation ideals on $\bC[x,y]$ have the form:
\[
\fa_m=\left(\left\{q_0^{a_0}q_1^{a_1}q_2^{a_2}\cdots q_k^{a_k}|\; a_0+\sum_{i=1}^k a_i \beta_i \ge m, a_j\le c_j-1 \text{ for }  j\ge 1\right\}\right).
\]
In particular, the quotients $\bC[x,y]/\fa_m$ have vector space basis consisting of ``monomials"
\[
q_0^{a_0}q_1^{a_1}q_2^{a_2}\cdots q_k^{a_k} \text{ where } a_0+\sum_{i=1}^k a_i \beta_i <m \text{ and } a_j\le c_j-1 \text{ for } j\ge 1.
\]
So computing the volume of $v$ amounts to counting the number of ``monomials" in this basis and the result was given in \cite{ELS03}:
\begin{equation}\label{vola*}
\vol(v)=\lim_{k\rightarrow+\infty}\alpha_k^{-1}=:\alpha_*^{-1}, 
\end{equation}
where 
\[
\alpha_k^{-1}=\frac{c_1c_2\cdots c_{k-1}}{\beta_{k}}=\frac{1}{\beta_1}\left(\frac{c_1 \beta_1}{\beta_2}\right)\cdots\left(\frac{c_{k-1} \beta_{k-1}}{\beta_k}\right),
\]
is a decreasing sequence of rational numbers strictly less than 1. 
As pointed out in \cite{ELS03}, a key to show this is that:
\begin{equation}\label{a*vol}
\alpha_*^{-1}=\lim_{k\rightarrow+\infty }\frac{\dim_{\bC}(\bC[x,y]/\fa_{c_k\beta_k})}{(c_k\beta_k)^2/2!}\le \vol(v).
\end{equation}
Indeed, because $\beta_{k+1}>c_k\beta_k$, $\dim_{\bC}(\bC[x,y]/\fa_{c_k\beta_k})$ is equal to 
the number of integral solutions to the inequalities:
\begin{equation}\label{ineqexp}
a_0+\sum_{i=1}^k a_i \beta_i<c_k\beta_k, \text{ and } 0\le a_j\le c_j-1 \text{ for } j\ge 1.
\end{equation}
We notice that
\[
c_k\beta_k-\sum_{i=1}^{k} (c_i-1)\beta_i=\sum_{i=1}^k \beta_i-\sum_{i=1}^{k-1} c_i\beta_i=\sum_{i=1}^{k-1}(\beta_{i+1}-c_i\beta_i)+\beta_1>\beta_1>1,
\]
so that in particular 
\[
c_k\beta_k> \sum_{i=1}^k (c_i-1)\beta_i. 
\]
So the number of integral solutions to \eqref{ineqexp} is equal to:
\[
\sum_{a_1=0}^{c_1-1}\sum_{a_2=0}^{c_2-1}\cdots\sum_{a_k=0}^{c_k-1}\left\lceil c_k\beta_k-a_1\beta_1-a_2\beta_2-\cdots-a_k\beta_k \right\rceil=:\tilde{\alpha}_k^{-1}. 
\]
A straight forward calculation shows that $\lim_{k\rightarrow +\infty}\frac{\alpha_k}{\tilde{\alpha}_k}=1$, so that \eqref{a*vol} holds. 

On the other hand, there is a sequence of approximating divisorial valuations $v_i$ satisfying $v_i(q_j)=\beta_j$ for any $j\le i$ and $v_i\le v$ (see \cite{ELS03, FJ04}). 
One can similarly show that $\vol(v_i)=\alpha_i^{-1}$. So we get $\alpha_i^{-1}=\vol(v_i)\ge \vol(v)$. Combining this with \eqref{a*vol}, we see that the equality \eqref{vola*} holds.

Next we calculate the log discrepancy of $v$ by using the work of Favre-Jonsson (\cite{FJ04}). In \cite[Chapter 2]{FJ04}, Favre-Jonsson studied the representation 
of a valuation on $\bC[x,y]$ by a countable {\it sequence of key-polynomials (SKP)} and real numbers. Using their notations, the valuation considered here is represented by
$[(q_i, \beta_i)]_{i=0}^{+\infty}$. The representation via key-polynomials allowed Favre-Jonsson to approximate any valuation by a sequence of divisorial valuations (\cite[Section 3.5]{FJ04}) and calculate various numerical invariants. In particular, they obtained the formula
for $\vol(v)$ and $A(v)=A_{\bC^2}(v)$.  Their formula for log discrepancy reads that (\cite[(3.8)]{FJ04}): 
\begin{equation}\label{eq-AvFJ}
A(v)=2+\sum_{i=1}^{+\infty}m_i (\alpha_i-\alpha_{i-1}),
\end{equation}
where $\alpha_i=\alpha(v_i)$ is the {\it skewness} of $v_i$ (see \cite[Definition 3.23]{FJ04}) defined as:
\[
\alpha_i=\sup\left\{\frac{v_i(f)}{\ord_0(f)}; f\in \fm \right\}
\]
and $m_i=m(v_i)$ is the {\it multiplicity} of $v_i$ (see \cite[Section 3.4]{FJ04}).
By \cite[Lemma 3.32]{FJ04} and \cite[Lemma 3.42]{FJ04} (the roles of $x$ and $y$ are reversed), we get for our valuation $v$ that,
\[
\alpha_i=\frac{\beta_i}{c_1c_2\cdots c_{i-1}}, \quad m_i=c_1c_2\cdots c_{i-1}, \quad \text{ for any } i\ge 1.
\]
So by Favre-Jonsson's formula \eqref{eq-AvFJ}, we get:
\begin{eqnarray*}
A(v)&=&2+\sum_{i=1}^{+\infty} c_1 c_2 \cdots c_{i-1} \left(\frac{\beta_i}{c_1c_2\cdots c_{i-1}}-\frac{\beta_{i-1}}{c_1c_2\cdots c_{i-2}}\right)\\
&=&2+\sum_{i=1}^{+\infty}(\beta_i-c_{i-1}\beta_{i-1}).
\end{eqnarray*}

For an explicit example, we follow \cite{ELS03} to take $c_i$ to be the $i$-th prime number ($c_1=2, c_2=3, \dots$) and set $\beta_1=\frac{3}{2}$ and $\beta_{i+1}=c_i \beta_i+\frac{1}{c_{i+1}}$ for $i\ge 1$.
Then the volume is equal to
\begin{eqnarray*}
\vol(v)=\alpha_*^{-1}%&=&\lim_{k\rightarrow +\infty} \left(1+\sum_{i=1}^{k}\frac{1}{\prod_{j=1}^{i}c_j}\right)^{-1}\\
&=&\left(1+\frac{1}{c_1}+\frac{1}{c_1c_2}+\frac{1}{c_1 c_2 c_3}+\cdots\right)^{-1}\\
&\approx& 0.58643.
\end{eqnarray*}
On the other hand,
\begin{eqnarray*}
A(v)&=&2+\sum_{i=1}^{+\infty}(\beta_i-c_{i-1}\beta_{i-1})=2+\sum_{i=1}^{+\infty}\frac{1}{c_i}=+\infty.
\end{eqnarray*}
The last identity is Euler's result: the sum of reciprocals of the primes diverges. 
\end{enumerate}

\section{Appendix: Convex geometric interpretation of volume formulas}\label{seconvex}

In this section, we use the tool of Okounkov bodies \cite{Oko96, LM09} and coconvex sets \cite{KK14} to provide a clear geometric interpretation of our volume formulas. 

\subsection{Integration formula by Okounkov bodies and coconvex sets}

Consider a $\bZ^n$-valuation defined on $R=\bigoplus_{k=0}H^0(V, kL)$ as follows. Choose a flag:
\[
V_\bull: V=V_0\supset V_1\supset V_2\supset \cdots\supset V_{n-2}\supset V_{n-1}=\{{\rm pt}\}
\]
of irreducible subvarieties of $V$, where ${\rm codim}_V(V_i)=i$ and each $V_i$ is non-singular at the point $V_{n-1}$. Following
\cite{Oko96, LM09}, this determines a $\bZ^n$-valuation $\bV$ on $R$. This is a function satisfying the axiom of valuations ($\bZ^n$ is endowed with the lexicographic order):
\[
\bV=\bV_{V_\bull}=(\bV_0, \bV_1, \cdots, \bV_{n-1})=:(\bV_0, \bV'): R\rightarrow \bZ^{n}\cup \{\infty\}.
\]
It's defined in the following way (see \cite{Oko96, LM09} for more details)
\begin{itemize}
\item
$g=g_{k_1}+\cdots+g_{k_p}$ with $g_{k_j}\neq 0\in R_{k_j}$ and $k_1<k_2<\cdots<k_p$. Define $\bV_0(g)=k_1$ and denote $g^{(0)}=g_{k_1} \in R_{k_1}=H^0(V, k_1 L)$;
\item
Define $\bV_1(g)=\ord_{V_1}(g^{(0)})=:\nu_1$. 
After choosing a local equation for $V_1$ in $V$, $g^{(0)}$ determines a section $\tilde{g}^{(0)}\in H^0(V, \cO_V(k_1 L-\nu_1 V_1))$ that
does not vanish identically along $V_1$, and so we get by restricting a non-zero section $g^{(1)}\in H^0(V_1, \cO_{V_1}(k_1L-\nu_1V_1))$. 
\item
Define $\bV_2(g)=\ord_{V_2}(g^{(1)})=:\nu_2$.
After choosing a local equation for $V_2$ in $V_1$, $g^{(1)}$ determines a section $\tilde{g}^{(1)}\in H^0(V_1, \cO_{V_1}(k_1L-\nu_1 V_1-\nu_2 V_2))$ that does not vanish
identically along $V_2$, and so we get by restricting a non-zero
section $g^{(2)}\in H^0(V_2, \cO_{V_2}(k_1 L-\nu_1 V_1-\nu_2 V_2))$. 
\item This process continues to define $\bV'$.
\end{itemize}
Denote  by $\cS=\bV(R\setminus\{0\})$ the value semigroup of $\bV$ and by $\cC=\mathcal{C}(\cS)$ the cone of $\cS$, i.e. the closure of the convex hull of 
$\cS\cup\{0\}$. 

With the above valuation $\bV$, we follow \cite{KK14} to interpret the volumes (mixed multiplicities) as covolumes of convex sets. 
From now on, let $v_1$ be a real valuation in $\Val_{X,o}$ satisfying $A_X(v_1)<+\infty$. The graded family of ideals 
$\fa_\bull:=\fa_\bull(v_1)=\{\fa_m(v_1)\}$ give rise to a primary sequence of subsets $\cA_\bull=\{\cA_m\}$ where 
\[
\cA_m=\left\{\bV(f); f\in \fa_m(v_1)\setminus\{0\} \right\}.
\]
Define the convex set
\begin{eqnarray*}
\Gamma=\Gamma(\cA_\bull)%&=&\text{closed convex hull}\left(\bigcup_{m>0}\frac{1}{m}\bV(\fa_m(v_1)-\{0\})\right)\\
&=&\text{closed convex hull}\left(\bigcup_{m>0}\frac{1}{m}\cA_m\right).
\end{eqnarray*}
By \cite[Theorem 8.12]{KK14}, we have (notice that the multiplicity $\fe(\fa_\bull(v_1))$ is $\vol(v_1)$ in our notation, see \eqref{vol=mul}):
\[
\vol(v_1)=n! \lim_{m\rightarrow+\infty}\frac{\dim_\bC R/\fa_m(v_1)}{m^n}=n!\; {\rm covol}(\Gamma).
\]
To get $\vol(v_1)$, we will calculate the covolume $\covol(\Gamma)$ by integrating the volumes of slices $\left(\cC\setminus \Gamma\right)\bigcap\{\bV_0=\tau\}$. 
The volumes of these slices will be shown to be ``covolume"'s of rescaled Okounkov bodies. 
From now on let $\cF^xR_k$ be the filtration determined by $v_1$. We recall that (see \cite{Oko96, LM09}) the Okounkov body  $\Delta^t$ of the graded sub algebra $R^{(t)}:=\bigoplus_k \cF^{kt}R_k$ is defined as:
\[
\Delta^t:=\text{closed convex hull}\left(\bigcup_{k\ge 1}\frac{1}{k}\cS^t_k \right)\subset \bR^{n-1},
\]
where we denoted:
\[
\cS^t_k:=\{\bV'(f); f\in \cF^{kt} R_k-\{0\}\}.
\]
To see the relation between the Okounkov bodies $\{\Delta^t\}$ and convex set $\Gamma$, we consider the following set of displaced union of rescaled Okounkov bodies: 
\begin{equation}\label{tgam}
\tilde{\Gamma}:=\bigcup_{t>0} t^{-1}\left(\{1\}\times\Delta^t\right).
\end{equation}
\begin{prop}\label{propslice}
%For any $t>0$, we have the following characterization of $\Delta^t$:
%\[
%\Delta^t=t\Gamma\bigcap \{\bV_0=1\}=t\left(\Gamma\bigcap \{\bV_0=t^{-1}\}\right).
%\]
%Equivalently, t
%We have the following description of the convex set $\Gamma$.
$\tilde{\Gamma}$ is convex. Moreover its closure is equal to the closed convex set $\Gamma$:
\[
\Gamma=\overline{\tilde{\Gamma}}=\overline{\bigcup_{t>0}\left(\{t^{-1}\}\times t^{-1} \Delta^t\right)}.
\]
\end{prop}
\begin{proof}
We first show the easier inclusion:
\begin{equation}\label{usubgam2}
\overline{\bigcup_{t>0}\left(\{t^{-1}\}\times t^{-1}\Delta^t\right)}\subseteq \Gamma.
\end{equation}
Since $\Gamma$ is closed, we just need to show that:
\begin{equation}\label{usubgam}
\bigcup_{t>0}\left(\{t^{-1}\}\times t^{-1}\Delta^t\right)\subseteq \Gamma.
\end{equation}
For any $f\neq 0\in \cF^{kt}R_k\subseteq R_k$, $\bV(f)=(k, \bV'(f))\in \{k\}\times \cS^t_k$. By the definition of $\cF^{kt}R_k$, there exists $g\in R$ with $v_1(g)\ge kt$ and $\bin(g)=f$. By the definition of $\bV$, we have $\bV(g)=\bV(f)$. So we have:
\begin{equation}\label{bVfg1}
\left(1, \frac{\bV'(f)}{k}\right)=\frac{\bV(f)}{k}=\frac{\bV(g)}{k}.
\end{equation}
We claim the right-hand-side is contained in $t \Gamma$. To see this, first notice that, since $v_1(g)\ge kt\ge \lfloor kt\rfloor$, we have
\[
\frac{\bV(g)}{k}=\frac{\lfloor kt\rfloor}{kt}t\frac{\bV(g)}{\lfloor kt\rfloor}\in \frac{\lfloor kt\rfloor}{kt} t\frac{\cA_{\lfloor kt\rfloor}}{\lfloor kt \rfloor}\subset \frac{\lfloor kt\rfloor}{kt} t\Gamma.
\]
Replacing $f$ by $f^p$ and $g$ by $g^p$ for $p\gg 1$, we also have:
\[
\frac{\bV(g)}{k}=\frac{\bV(g^p)}{p k}=\frac{\lfloor pkt\rfloor }{pkt}t\frac{\bV(g^p)}{\lfloor pkt\rfloor}\in \frac{\lfloor pkt\rfloor}{pkt} t\Gamma.
\]
Taking $p\rightarrow +\infty$ and using $\Gamma$ is a closed set, we indeed get $
\frac{\bV(g)}{k}\in t\Gamma$. So by \eqref{bVfg1}, we get:
\[
\{1\}\times \frac{\cS^t_k}{k}\subset t\Gamma\bigcap \{\bV_0=1\}.
\]
Since this holds for any $k\ge 1$ and $f\neq 0\in \cF^{kt}R_k$, we get an inclusion:
\[
\{1\}\times \Delta^t=\text{closed convex hull}\left(\bigcup_{k}\frac{S^t_k}{k}\right)\subseteq t\Gamma\bigcap \{\bV_0=1\}.
\]
Equivalently we have that, 
\[
\{t^{-1}\}\times t^{-1}\Delta^{t}\subseteq \Gamma\bigcap\{\bV_0=t^{-1}\}, \text{ for any } t>0.
\]
This gives the inclusion \eqref{usubgam}.

Next we show the other direction of inclusion. For any $g\neq 0\in R$ with $v_1(g)\ge m$ and $\bin(g)=f\neq 0\in \cF^{m}R_k$. By definition, we 
have $\bV(g)=(k, \bV'(f))$ so that:
\begin{eqnarray*}
\frac{\bV(g)}{m}=\frac{1}{m}(k, \bV'(f))=\frac{k}{m}\left(1, \frac{\bV'(f)}{k}\right).
\end{eqnarray*}
Notice that
\[
\frac{\bV'(f)}{k}\in \frac{\bV'(\cF^{(m/k)k}R_k-\{0\})}{k}=\frac{\cS^{m/k}_k}{k}\subset  \Delta^{m/k}.
\]
So we see that:
\begin{equation}\label{Vgmin}
\frac{\bV(g)}{m}\in \frac{k}{m}\left(\{1\}\times \Delta^{m/k}\right).
\end{equation}
Because \eqref{Vgmin} holds for any $m\ge 1$ and any $g\neq 0\in \fa_m(v_1)-\{0\}$, we get:
\begin{equation}\label{unsub}
\bigcup_{m\ge 1}\frac{\cA_m}{m}=\bigcup_{m\ge 1}\frac{\bV(\fa_m(v_1)-\{0\})}{m}\subset \tilde{\Gamma},
\end{equation}
where $\tilde{\Gamma}$ denotes the displaced union in \eqref{tgam}. 
%\begin{equation}\label{tgam}
%\tilde{\Gamma}:=\bigcup_{t>0} t^{-1}\left(\{1\}\times\Delta^t\right).
%\end{equation}
Now we claim that $\tilde{\Gamma}$ is a convex set. Assuming the claim, we can take the closed convex hull on the left hand side of \eqref{unsub} to get the other direction of inclusion:
\[
\Gamma\subseteq \overline{\tilde{\Gamma}}=\overline{\bigcup_{t>0}t^{-1}(\{1\}\times\Delta^t)}=\overline{\bigcup_{t>0}\{t^{-1}\}\times t^{-1} \Delta^{t}}.
\]
Next we verify the claim on the convexity of $\tilde{\Gamma}$. First notice that we can rewrite:
\[
\tilde{\Gamma}=\bigcup_{\tau>0} \{\tau\}\times \tau\Delta^{\tau^{-1}}.
\]
Because each slice $\tau \Delta^{\tau^{-1}}$ is a convex set, the convexity of $\tilde{\Gamma}$ is equivalent to the inclusion relation (compare to a similar but different relation in 
\cite[(1.6)]{BC11}):
\begin{equation}\label{inconv1}
(1-s)\tau_1\Delta^{\tau_1^{-1}}+s \tau_2\Delta^{\tau_2^{-1}}\subseteq ((1-s)\tau_1+s \tau_2)\Delta^{((1-s)\tau_1+s\tau_2)^{-1}},
\end{equation}
for any $\tau_1, \tau_2\in \bR_{>0}$ and $s\in [0,1]$. We first reduce this to proving the inclusion:
\begin{equation}\label{rediscr}
(1-s)\tau_1\frac{\cS^{\tau_1^{-1}}_{k_1}}{k_1}+s\tau_2\frac{\cS^{\tau_2^{-1}}_{k_2}}{k_2}\subset ((1-s)\tau_1+s\tau_2)\Delta^{((1-s)\tau_1+s\tau_2)^{-1}}.
\end{equation}
For any two points $u_i\in {\rm int}\left(\Delta^{\tau_i^{-1}}\right), i=1,2$, we can find two sets $\{a_i, x_i\}_{i=1}^{N_1}$ and $\{b_j, y_j\}_{j=1}^{N_2}$ satisfying:
\[
\sum_{i}a_i=1=\sum_{j} b_j,\quad 0 \le a_i, b_j\le 1, \quad x_i\in \frac{S^{\tau_1^{-1}}_{k_i}}{k_i}, y_j\in \frac{S^{\tau_2^{-1}}_{l_j}}{l_j},
\] 
such that $u_1=\sum_i a_i x_i$ and $u_2=\sum_j b_j y_j$. The linear combination 
$\displaystyle
(1-s)\tau_1 u_1+s\tau_2 u_2=(1-s)\tau_1\sum_{i} a_i x_i+s \tau_2\sum_{j} b_j y_j$
can be written as a convex linear combination:
\begin{equation}\label{convcomb}
\sum_{i,j} c_{i,j} ((1-s)\tau_1 x_i+s \tau_2 y_j)
\end{equation}
with $c_{i,j}\in [0,1]$, $\sum_{j}c_{i,j}=a_i$ and $\sum_{i}c_{i,j}=b_j$. If each term in \eqref{convcomb} is in the right-hand-side of \eqref{rediscr}, the convex combination is also in it as well. For any two points $u_i\in \Delta^{\tau_i^{-1}}, i=1,2$, we can find two sequence of points $\left\{u_i^{(k)}\right\}_{k=1}^{+\infty}\subset {\rm int}\left(\Delta^{\tau_i^{-1}}\right), i=1,2$ such that $u_i^{(k)}\rightarrow u_i$ as $k\rightarrow+\infty$. By the above discussion, we have $(1-s)\tau_1 u_1^{(k)}+s \tau_2 u_2^{(k)}$ is contained in the right-hand-side of \eqref{rediscr} that is a closed set. Letting $k\rightarrow+\infty$, we get $(1-s)\tau_1 u_1+s \tau_2 u_2\in$ is contained in it too.

So finally, we show \eqref{rediscr} holds. For $i=1,2$, choose $k_i\in \bZ_{>0}$ and $f_i\in \cF^{k_i \tau_i^{-1}}R_{k_i}$. Then there exists $g_i \in R_{k_i}$ such that $v_1(g_i)\ge k_i \tau_i^{-1}$ and $\bin(g_i)=f_i$. 
We fix any $s\in [0,1]$, and for simplicity we will denote
\[
 u_i=\frac{\bV'(f_i)}{k_i}, i=1,2;\quad \tau=\tau(s)=(1-s)\tau_1+s\tau_2.
\] 
For any integer $l_1, l_2\in \bZ_{>0}$, we define
\[
g=g(l_1,l_2)=g_1^{l_1}g_2^{l_2}, \quad f=f(l_1, l_2)=f_1^{l_1}f_2^{l_2}=\bin(g).
\]
Then we have
\begin{equation}\label{valfg}
\def\arraystretch{1.2}
\begin{array}{l}
v_0(f)=l_1 k_1+l_2 k_2=:k;\\
\bV'(f)=l_1 \bV'(f_1)+l_2 \bV'(f_2)=l_1 k_1 u_1+l_2 k_2 u_2;\\
v_1(g)\ge l_1 \frac{k_1}{\tau_1}+l_2 \frac{k_2}{\tau_2}=:m.
\end{array}
\end{equation}
So we get:
\begin{equation}\label{eqf+}
f=\bin(g)\in \cF^{m }R_{k}= \cF^{(m/k)k}R_{k}.
\end{equation}
We want to show that for any $\ep>0$, there exists $l_1=l_1(\epsilon)$ and $l_2=l_2(\epsilon)$ such that 
the following properties are satisfied:
\begin{enumerate}
\item
$m/k\ge \tau^{-1}$ such that $f\in \cF^{k \tau^{-1}}R_k$ and $\frac{\bV'(f)}{k}\in \Delta^{\tau^{-1}}$. By
\eqref{valfg}, this is equivalent to:
\begin{equation}\label{eqcond1}
\frac{l_1k_1}{l_1k_1+l_2k_2}\frac{1}{\tau_1}+\frac{l_2 k_2}{l_1k_1+l_2k_2}\frac{1}{\tau_2}\ge \frac{1}{\tau}.
\end{equation}
\item
\begin{equation}\label{intapp}
\left\|\tau \frac{\bV'(f)}{k}-(1-s)\tau_1 u_1-s \tau_2 u_2\right\|\le \ep.
\end{equation}
By \eqref{valfg}, this is equivalent to:
\begin{equation}\label{eqcond2}
\left\|\left(\tau \frac{l_1k_1}{l_1k_1+l_2k_2}-(1-s)\tau_1\right)u_1+\left(\tau \frac{l_2k_2}{l_1k_1+l_2k_2}-s\tau_2 \right)u_2\right\|\le \ep.
\end{equation}
\end{enumerate}
Assuming that we can achieve these, then the proof would be completed. Indeed, by these two conditions, we would have that
for any $\epsilon>0$, $(1-s)\tau_1 u_1+s\tau_2 u_2$ is approximated arbitrarily close by elements in $\tau\Delta^{\tau^{-1}}$. 
Because $\Delta^{\tau^{-1}}$ is a closed set, we indeed get $(1-s)\tau_1 u_1+s\tau_2 u_2\in \tau \Delta^{\tau^{-1}}$.

Notice that the equalities in \eqref{eqcond1} and \eqref{eqcond2} are both satisfied if and only if
\begin{equation}\label{eqcond3}
\frac{l_2 k_2}{l_1 k_1}=\frac{s\tau_2}{(1-s)\tau_1}.
\end{equation}
More precisely, if we define the ratios:
\[
\delta=\delta(l_1, l_2)=\frac{l_2k_2}{l_1k_1}, \quad \delta_0=\frac{s\tau_2}{(1-s)\tau_1},
\]
then the left hand side minus the right hand side of \eqref{eqcond1} is equal to:
\begin{equation}\label{cond1b}
\frac{\delta-\delta_0}{(1+\delta)(1+\delta_0)}\frac{\tau_1-\tau_2}{\tau_1\tau_2};
\end{equation}
while the absolute difference on the left hand side of \eqref{eqcond2} is equal to:
\begin{equation}\label{cond2b}
\frac{\tau\|u_1-u_2\|}{(1+\delta)(1+\delta_0)}|\delta-\delta_0|.
\end{equation}
If the right-hand-side of \eqref{eqcond3} is a rational number then we can find a solution to \eqref{eqcond3}. Otherwise, it's clear
that we can find infinitely many pairs of $(l_1, l_2)\in \bZ_{>0}^2$ such that \eqref{cond1b} is strictly positive and \eqref{cond2b} 
is as small as possible. 

%\begin{eqnarray*}
%\frac{k}{\tau}\left((1-s)\tau_1\frac{\bV'(f_1)}{k_1}+s\tau_2\frac{\bV'(f_2)}{k_2}\right)&=&\frac{(1-s)k_2\tau_1\bV'(f_1)+sk_1\tau_2\bV'(f_2)}{k_1k_2}\\
%&=&\frac{(1-s)k_2 p_2\tau_1\bV'(f^{p_1}_1)+s k_1 p_1\tau_2\bV'(f^{p_2}_2)}{k_1p_1 k_2p_2}=:{\bf I}(p_1,p_2).
%\end{eqnarray*}
\end{proof}

We can now calculate the covolume of $\Gamma$ using $\bV_0$-slices. First there exists $b_1>0$ such that $\Gamma\bigcap \{\bV_0> b_1\}=\cC\bigcap\{\bV_0> b_1\}$.  
So we have:
\begin{eqnarray*}
{\rm covol}(\Gamma)&=&\vol(\cC\bigcap\{\bV_0 \le b_1\})-\vol(\Gamma\bigcap \{\bV_0 \le b_1\})\\
&=&\frac{L^{n-1}}{(n-1)!}\int_0^{b_1} \tau^{n-1}d \tau- \vol(\Gamma\bigcap\{\bV_0\le b_1\})\\
&=&\frac{L^{n-1}}{n!}b_1^n-\vol(\Gamma\bigcap \{\bV_0\le b_1\}).
\end{eqnarray*}
For the second term, we first notice that 
\begin{equation}\label{eq-2vol}
\vol(\Gamma\cap \{\bV_0\le b_1\})=\vol(\tilde{\Gamma}\cap \{\bV_0\le b_1\}),
\end{equation}
where $\tilde{\Gamma}$ was the set defined in \eqref{tgam}:
\begin{equation*}
\tilde{\Gamma}=\bigcup_{t>0} t^{-1}\left(\{1\}\times\Delta^t\right).
\end{equation*}
To see this, recall that by Proposition \ref{propslice}, $\tilde{\Gamma}$ is convex and $\Gamma=\overline{\tilde{\Gamma}}$. It's also clear that $\tilde{\Gamma}\cap \{\bV_0> b_1\}=\cC\cap \{\bV_0> b_1\}$ for $b_1$ sufficiently large. So for $b_1\gg 1$ we also have:
\[
\Gamma\cap \{\bV_0\le b_1\}=\overline{\tilde{\Gamma}\cap \{\bV_0\le b_1\}}.
\] 
Since $\tilde{\Gamma}\cap \{\bV_0\le b_1\}$ is again convex, its boundary has Lebesgue measure 0 (see \cite{Lan86} for example) and hence \eqref{eq-2vol} holds. So we get: %By Proposition \ref{propslice},  we have:
\begin{eqnarray*}
\vol(\Gamma\bigcap \{\bV_0 \le b_1\})&=&\vol(\tilde{\Gamma}\bigcap\{\bV_0 \le b_1\})\\
&=&\int_0^{b_1}\vol(\tilde{\Gamma}\bigcap\{\bV_0=\tau\})d\tau=\int_0^{b_1}\vol(\tau \Delta^{\tau^{-1}})d\tau\\
&=&\int_0^{b_1} \tau^{n-1} \vol(\Delta^{\tau^{-1}})d\tau=\int^{+\infty}_{b_1^{-1}} \vol(\Delta^t)\frac{dt}{t^{n+1}}\\
&=&\frac{1}{(n-1)!}\int^{+\infty}_{b_1^{-1}} \vol\left(R^{(t)}\right)\frac{dt}{t^{n+1}}.
\end{eqnarray*}
So we indeed recover the formula \eqref{volv1a} for the volume of $v_1$:
\begin{eqnarray}\label{form1}
\vol(v_1)&=&n!\; \covol(\Gamma)=b_1^n L^{n-1}-n\int^{+\infty}_{b_1^{-1}}\vol\left(R^{(t)}\right)\frac{dt}{t^{n+1}}.
\end{eqnarray}
As before, by integration by parts we get formula \eqref{volv1b}:
\begin{eqnarray}\label{form2}
\vol(v_1)&=& b_1^n L^{n-1}+\int_{b_1^{-1}}^{+\infty}\vol\left(R^{(t)}\right)d\left(\frac{1}{t^n}\right)\nonumber\\
&=&\int^{+\infty}_{b_1^{-1}}\frac{-d\vol\left(R^{(t)}\right)}{t^n}.
\end{eqnarray}
To get another formula, we will consider the following function
\begin{eqnarray*}
H_{\cF}(x)%&=&\inf\{\tau \in \bR_{>0}; (\tau, \tau x)\in \Gamma\}\\
&=&\inf \left\{\tau>0; x\in \Delta^{\tau^{-1}}\right\}\quad \text{ for any } x\in \Delta^{0}.
\end{eqnarray*}
%The last identity holds because we have:
%\[
%(\tau, \tau x)\in \Gamma \Longleftrightarrow \tau x\in \Gamma\bigcap\{\bV_0=\tau\}\Longleftrightarrow x\in \tau^{-1}\Gamma\bigcap\{\bV_0=1\}=\Delta^{\tau^{-1}}.
%\]
So we see that $H_{\cF}(x)$ is related to the function considered by Boucksom-Chen in \cite{BC11}. 
\begin{eqnarray*}
H_{\cF}^{-1}(x)&=&G_{\cF}(x)=\sup\{t>0; x\in \Delta^t\}.
%&=&\sup\{t>0; (t^{-1}, t^{-1}x)\in \Gamma\}.
\end{eqnarray*}
%We also have
%\begin{eqnarray*}
%&& \Gamma=\{(t, tx); x\in \Delta^0, t>0, H_{\cF}(x/t)\le t\}. %; \\
%&&\{x\in \Delta^0; \tau\ge H_{\cF}(x)\}=\tau^{-1}\left(\Gamma\bigcap \{\bV_0=\tau\}\right)=\Delta^{\tau^{-1}}.
%\end{eqnarray*}
By \cite{BC11}, $H_{\cF}^{-1}(x)$ is a concave function and it's continuous on ${\rm int}(\Delta^{0})$.
If we denote by $\lambda$ the Lebesgue measure of $\bR^{n-1}$, then we have (cf. \cite[Proof of Theorem 1.11]{BC11}):
\[
d\vol(\Delta^{\tau^{-1}})=\left(\frac{d}{d\tau}\vol(\{x\in \Delta^0; H_{\cF}(x)\le \tau\})\right) d\tau=(H_{\cF})_*\lambda.
\]
We can then calculate covolume of $\Gamma$:
\begin{eqnarray}
{\rm covol}(\Gamma)&=&{\rm covol}(\tilde{\Gamma})\\%&=&\frac{L^{n-1}}{(n-1)!}\int_0^{b_1}\tau^{n-1}d\tau-\int_0^{b_1}\tau^{n-1}\vol(\Delta^{\tau^{-1}})d\tau\\
&=&\int_0^{b_1} \tau^{n-1}\left(\vol(\Delta^0)-\vol(\Delta^{\tau^{-1}})\right)d\tau\nonumber\\
&=&\int_0^{b_1} \tau^{n-1}\left(\int_{\Delta^0} {\bf 1}_{\{\tau\le H_{\cF}(x)\}}dx\right)d\tau\nonumber\\
&=&\int_{\Delta^0} dx \int_0^{H_{\cF}(x)} \tau^{n-1} d\tau=\frac{1}{n}\int_{\Delta^0} H_{\cF}^n(x) dx.
\end{eqnarray}
So we get another formula for $\vol(v_1)$:
\begin{equation}\label{form3}
\vol(v_1)=n!\; \covol(\Gamma)=(n-1)! \int_{\Delta^0} H_\cF^n(x)dx.
\end{equation}
The relation between \eqref{form2} and \eqref{form3} is given by:
\begin{eqnarray*}
(n-1)!\int_{\Delta^0} H^n_{\cF}(x)dx%&=&(n-1)!\int_{\Delta^0} G_{\cF}^{-n}(x)dx\\
&=&(n-1)!\int_0^{+\infty} \tau^{n}( H_{\cF})_* \lambda\\
&=&(n-1)!\int_0^{+\infty} \tau^{n}d\vol(\Delta^{\tau^{-1}})\\
&=&-(n-1)! \int_0^{+\infty}  t^{-n} d\vol(\Delta^t)\\
&=&\int_0^{+\infty} \frac{-d\vol\left(R^{(t)}\right)}{t^n}.
\end{eqnarray*}
Finally, let's point out the geometric meaning of the expression in \eqref{eqintfn}. For simplicity we can assume $\lambda=1$
by rescaling the coconvex set. Then  
\begin{eqnarray*}
\Phi(1,s)&=&\int^{+\infty}_{c_1}\frac{-d\vol\left(R^{(t)}\right)}{((1-s)+s t)^{n}}\\
&=&(n-1)!\int^{c_1^{-1}}_0 \frac{d\vol(\Delta^{\tau^{-1}})}{(1-s+s \tau^{-1})^n}\\
&=&(n-1)!\int^{c_1^{-1}}_0 \frac{\tau^n}{((1-s)\tau+s)^n} (H_\cF)_*\lambda\\
&=&\frac{n!}{n}\int_{\Delta^0}\left(\frac{H_\cF}{(1-s)H_\cF+s}\right)^n dx.
\end{eqnarray*}
Comparing with \eqref{form3}, it's clear that, if we define the function 
\[
H_s(x)=\frac{H_\cF}{(1-s)H_\cF+s}. 
\] 
Then $\Phi(1,s)=n!\; \covol(\Gamma_s)$ where the convex set $\Gamma_s$ for any $s\in [0,1]$ is defined as:
\[
\Gamma_s=\left\{(t,tx); x\in \Delta^0, t>0,  H_s(x)\le t\right\}.
\]

\subsection{Examples}
We will give convex geometric interpretation of the volumes of valuations in Section \ref{sec-exnv}. 
\begin{enumerate}
\item
Let $Y\rightarrow \bC^2$ be the blow up at the origin with exceptional divisor $\bP^1$. We choose the point $\{0\}=[1,0]\in \bP^1$ and get a $\bZ^2$-valuation $\bV=(\bV_0, \bV_1)$ on $\bC[x,y]$ such that for each homogeneous polynomial $\bV_1$ gives the lowest degree in $y$ variable. Denote by $(\tau, \mu)$ the rectangular coordinates on the $\bR^2=\bZ^2\otimes_{\bZ}\bR$. Then using the previous notations, it's easy see that
\[
\cS=\left\{(a+b, b); (a,b) \in \bZ_{\ge 0}^2\right\};  \quad \cC=\cC(\mathcal{S})=\left\{(\tau, \mu)\in \bR^2_{\ge 0}; \mu\le \tau\right\}.
\]
The convex set $\Gamma$ associated to $v$ is a closed $\cC$-convex set. Because it contains the vectors:
\[
\frac{(a+b, b)}{\lfloor \gamma_1 a+\gamma_2 b\rfloor} \text{ with } a>0, b>0,
\]
it's easy to see that:
\[
\Gamma=\cC\setminus \triangle_{OAB},
\]
where the triangle $\triangle_{OAB}$ has vertices $O=(0,0)$, $A=(\frac{1}{\gamma_1},0)$ and $B=(\frac{1}{\gamma_2},\frac{1}{\gamma_2})$. Notice that on the blow up $Y$, 
the center of $v$ is $[1,0]\in \bP^1$ (resp. $[0,1]\in \bP^1$) if and only if $\gamma_2>\gamma_1$ (resp. $\gamma_2<\gamma_1$) if and only if the slope of the line $\overline{AB}$ is negative (resp. positive). 

\item
%We will give a convex geometric interpretation of this result. For this we bring back
We will use the same $\bZ^2$-valuation $\bV$ as in the previous example.
Denote by $v_0=\ord_{o}$ (where $o=(0,0)$) the canonical valuation on $\bC[x,y]$:
\[
v_0\left(\sum_{e_1, e_2\in \bZ_{\ge 0}} p_{\alpha_1, \alpha_2} x^{e_1} y^{e_2}\right)=\min\{e_1+e_2; p_{e_1,e_2}\neq 0\}.
\]
For convenience, we define the numbers:
\[
d_1=1 \text{ and } d_i=c_1c_2\cdots c_{i-1} \text{ for } i\ge 2.
\]
Notice that $\alpha_i=\frac{d_i}{\beta_i}$. 
Then we have:
\begin{lem}\label{lemv0in}
For any $k\in \bZ_{\ge 0}$, $a_0\in \bZ_{\ge 0}$ and $a_i\in \bZ_{\ge 0}\cap [0,c_i-1]$, we have
\[
v_0(q^{\bf a})=a_0+\sum_{i=1}^k a_i d_i
\]
and the homogeneous component of lowest degree of $q^{\bf a}$, denoted by $\bin(q^{\bf a})$, is 
$y^{a_0}x^{\sum_{i=1}^k a_i d_i}$. Moreover, for any linear combination of $q^{\bf a}$, 
\[
v_0\left(\sum_{\bf a} b_{\bf a} q^{\bf a}\right)=\min\{v_0(q^{\bf a}), b_{\bf a}\neq 0\}.
\]

\end{lem}
\begin{proof}
For the first statement,
it's clear that we just need to prove it for each $q_i$. 
We prove this by induction. Since $q_1=x$, $\deg(q_1)=1$ and $\bin(q_1)=1$.  Since $q_2=x^{c_1}+y^{c_1\beta_1}$, $\deg(q_2)=c_1<\beta_1 c_1$ and $\bin(q_2)=x^{c_1}$. Assume we have proved $\deg(q_i)=d_i$ and $\bin(q_i)=x^{d_i}$. To show that $q_{i+1}=q_i^{c_i}+y^{c_i\beta_i}$ satisfies the first statement, we just need to show that $c_i d_i < c_i \beta_i$. This is equivalent to $\alpha_i=\frac{d_i}{\beta_i}<1$ which indeed holds true.

To prove the second statement, we choose the minimum among $v_0(q^{\bf a})$ with $b_{\bf a}\neq 0$. Then $\bin(q^{\bf a})=y^{a_0}x^{\sum_{i=1}^k a_id_i}$. We claim that $\bin(q^{\bf a})$ can not be cancelled in the standard monomial expansion of $\sum_{\bf a} b_{\bf a}q^{\bf a}$. Indeed, if it is cancelled, then there exists ${\bf a'}$ with $p_{\bf a'}\neq 0$ and $a'_0=a_0$ such that $\bin(q^{\bf b})=y^{a_0}x^{\sum_{j=1}^{k'} a'_i d_j}$ satisfies $\sum_{j=1}^{k'} a'_j d_j=\sum_{i=1}^k a_i d_i$. By allowing zero coefficients and without loss of generality, we can assume that $k'=k$ and $c_j-1\ge a'_j> a_j\ge 0$. Then we can estimate that
\begin{eqnarray*}
0&=&\sum_{j=1}^k a'_jd_j-\sum_{i=1}^k a_id_i=(a'_k-a_k)d_k-\sum_{i=1}^{k-1} (a'_i-a_i)d_i\\
&\ge & d_k-\sum_{i=1}^{k-1} (c_i-1)d_i=\sum_{i=1}^{k-1} (d_{i+1}-c_i d_i)+d_1=1>0. 
\end{eqnarray*}
So we get a contradiction.
\end{proof}
\begin{cor}
We have the following formula:
\[
v(\fm)=\inf_{\fm}\frac{v}{v_0}=1, \quad \sup_{\fm}\frac{v}{v_0}=\alpha_*.
\] 
\end{cor}
\begin{proof}
We first notice the estimates:
\[
1\le \frac{v(q^{\bf a})}{v_0 (q^{\bf a})}=\frac{a_0+\sum_{i=1}^{k}a_i \beta_i }{a_0+\sum_{i=1}^{k}a_i d_i}\le \frac{\beta_i}{d_i}\le \alpha_*.
\]
The supreme can be proved by letting $k\rightarrow +\infty$. The infimum can proved using the second statement of the above lemma.
\end{proof}
\begin{rem}
Notice that by \eqref{vola*}, the above corollary is compatible with Favre-Jonsson's formula for $\vol(v)$ in dimension $2$ (\cite[Remark 3.33]{FJ04}): 
\[
\vol(v)=\left(\sup_{\fm}\frac{v}{\ord_0}\right)^{-1}\frac{1}{v(\fm)}.
\] 
\end{rem}

As before, let $Y\rightarrow \bC^2$ be the blow up at the origin with exceptional divisor $\bP^1$. We choose the point $\{0\}=[1,0]\in \bP^1$ then we get a $\bZ^2$-valuation $\bV=(\bV_0, \bV_1)$ on $\bC[x,y]$. Then by lemma \ref{lemv0in}, we get:
\begin{cor}
For any exponents ${\bf a}=(a_0, a_1, \dots, a_k)$ satisfying the assumption in the above lemma, we have $\bV(q^{\bf a})=(a_0+\sum_{i=1}^{k} a_i d_i, a_0)$. 
\end{cor}

\begin{prop}
The convex set $\Gamma$ associated to $v$ is equal to the region: 
\[
\Gamma=\cC\bigcap \left\{\mu \le \frac{\tau-\alpha_*^{-1}}{1-\alpha_*^{-1}}\right\}=\cC\setminus \triangle_{OAB}, 
\]
where $O=(0,0)$, $A=(\alpha_*^{-1},0)$ and $B=(1,1)$. As a consequence, $\vol(v)=2!\cdot\covol(\Gamma)=\alpha_*^{-1}$.
\end{prop}
Notice that the set $\mathcal{H}:=\left\{\mu\le \frac{\tau-\alpha_*^{-1}}{1-\alpha_*^{-1}}\right\}$ is the half space on the right hand side of the line $\mu=\frac{\tau-\alpha_*^{-1}}{1-\alpha_*^{-1}}$ which 
passes through $(1,1)$ and $(\alpha_*^{-1}, 0)$.  
\begin{proof}
By Corollary \ref{lemv0in}, we know that for any exponent ${\bf a}$ satisfying $a_0\ge 0$ and $0\le a_i\le c_i-1$, we have: 
\[
\Gamma\ni \frac{\bV(q^{\bf a})}{\lfloor v(q^{\bf a})\rfloor}=\frac{(a_0+\sum_{i=1}^k a_i d_i, a_0)}{\lfloor a_0+\sum_{i=1}^k a_i \beta_i \rfloor}
\]
Letting $a_i=0$ and $a_0\neq 0$, we see that $(1,1)\in \Gamma$. Letting $a_0=a_1=\cdots=a_{k-1}=0$ and $a_i=c_{i}-1$, we see that
\[
\Gamma\ni \frac{((c_k-1) d_k, 0)}{\lfloor (c_k-1) \beta_k\rfloor }=\left(\frac{d_k}{\beta_k},0\right)\frac{(c_k-1)\beta_k}{\lfloor (c_k-1)\beta_k\rfloor}\stackrel{k\rightarrow+\infty}{\longrightarrow} \left(\alpha_*^{-1},0\right).
\]
So $(\alpha_*^{-1}, 0)\in \Gamma$ because $\Gamma$ is closed. Because $\Gamma$ is a $\cC$-convex closed set, we deduce that $\Gamma\supseteq \cC\bigcap \mathcal{H}$. To prove 
the other direction of inclusion, by the definition of $\Gamma$ we just need to show that:
\[
\frac{\cA_m}{m}=\frac{\bV\left(\fa_m(v)-\{0\}\right)}{m}\in \cC\bigcap \mathcal{H}.
\]
In other words, for any $f\in \fa_m(v)$, we wan to show $\frac{\bV(f)}{m}\in \cC\bigcap\mathcal{H}$.  To prove this, we expand $f$ into ``monomials" $f=\sum_{\bf a}b_{\bf a} q^{\bf a}$ with each $b_{\bf a}\neq 0$. Then by the proof of Lemma \ref{lemv0in}, we see that there is an exponent ${\bf a}^*$ among the $\bf a$'s such that:
\[
\bV(f)=v_0(f)=v_0(q^{\bf a^*})=a^*_0+\sum_{i=1}^{k} a^*_i \beta_i, \text{ and } \bV_1(f)=a^*_0. 
\]
Recall that we have $v(f)=\min_{\bf a}\{v(q^{\bf a})\}$, so that we can write the normalized valuation vector as:
\begin{eqnarray*}
\frac{\bV(f)}{m}=\lambda \frac{\bV(q^{\bf a^*})}{v(q^{\bf a^*})},
\end{eqnarray*}
where 
\[
\lambda=\frac{v(q^{\bf a^*})}{m}=\frac{v(q^{\bf a^*})}{\min_{\bf a}\{v(q^{\bf a})\}}\frac{v(f)}{m}\ge 1.
\]
Because $\cC\bigcap\mathcal{H}$ is closed under rescalings, we just need to show 
\[
\frac{\bV(q^{\bf a^*})}{v(q^{\bf a^*})}=\frac{(a_0+\sum_{i}a^*_i d_i, a^*_0)}{a^*_0+\sum_{i}a^*_i \beta_i}\in \cC\bigcap\mathcal{H}.
\]
This is indeed true because it can be written as a convex combination:
\[
\frac{\bV(q^{\bf a^*})}{v(q^{\bf a^*})}=\frac{a^*_0}{a^*_0+\sum_i a^*_i\beta_i}(1,1)+\sum_{i} \frac{a^*_i}{a^*_0+\sum_i a^*_i\beta_i} \left(\frac{d_i}{\beta_i},0\right)
\]
and the set of points $\{(1,1)\}\bigcup\cup_i \{(\frac{d_i}{\beta_i},0)\}$ is contained in $\cC\bigcap\mathcal{H}$.

%Because $\covol(\Gamma)=\vol(v_1)=\alpha_*^{-1}$ by \eqref{vola*} and $\covol(\cC\bigcap\mathcal{H})={\rm Area}(\triangle_{OAB})=\alpha_*^{-1}$, we see that the two closed coconvex sets actually coincide. 
\end{proof}

\begin{center}
\renewcommand*{\arraystretch}{1.5}
\begin{tabular}{|c|c|c|c|c|}
\hline
$c_k$ & $\beta_k=v(q_k)$ & $q_k$ & $\#(\cS\setminus\cA_{c_k\beta_k})$ & $\frac{\dim_{\bC}\left(\bC[x,y]/\fa_{c_k\beta_k}\right)}{(c_k\beta_k)^2/2!}$ \\
\hline
%\thickhline
$c_1=2$ & $\frac{3}{2}$     &    $x$      &  $5$  & $1.11111$ \\
\hline
$c_2=3$ & $\frac{10}{3}$    &    $x^2+y^3$      &  $38$  & $0.76$ \\
\hline
$c_3=5$ & $\frac{51}{5}$   & $(x^2+y^3)^3+y^{10}$ & $810$ & $0.62284$\\
\hline
$c_4=7$ & $\frac{358}{7}$ & $ q_3^5+y^{51}$& $37923$ & $0.59179$ \\
\hline
$c_5=11$ & $\frac{3939}{11}$ & $q_4^7+y^{358}$ & $4553318$ & $0.58693$\\
\hline
\end{tabular}
\end{center}
Figure \ref{figam} shows $\frac{1}{c_k\beta_k}\left(\cS\setminus \cA_{c_k \beta_k}\right)$ for $k=2,3,4$.  As an example of Proposition \ref{propdim}, we let $m=c_2 \beta_2=3\cdot 10/3=10$ so that $\dim_{\bC}(\bC[x,y]/\fa_m(v))=\#(\cS\setminus \cA_{10})=38$. On the other hand, we can write down the basis of $\cF^m R_k$ for $k=6,7,8, 9$ for which $0\subsetneq \cF^{10} R_k\subsetneq R_k$:
\begin{enumerate}
\item $k=6$: $x^6=\bin((x^2+y^3)^3)=\bin(q_2^3)$.
\item $k=7$: $x^7$, $x^6y$, $x^5y^2=\bin\left(y^2 x (x^2+y^3)^2\right)=\bin(y^2q_1 q_2^2)$.
\item $k=8$: $x^8$, $x^7y$, $x^6 y^2$, $x^5 y^3$, $x^4 y^4$. 
\item $k=9$: $x^9$, $x^8y$, $x^7 y^2$, $x^6 y^3$, $x^5 y^4$, $x^4 y^5$, $x^3 y^6$, $x^2 y^7$.
\end{enumerate}
Notice that the elements $x^6, x^5y^2$ show that the equality $\cF^{m} R_k=R_k\bigcap \fa_m(v)$, which holds in the $\bC^*$-invariant case by Lemma \ref{lemsat}, does not hold in general.
\begin{figure}[h]
  \begin{center}
  \subfigure[$c_2\beta_2=10$]{\label{gam2}\includegraphics[height=3.5cm]{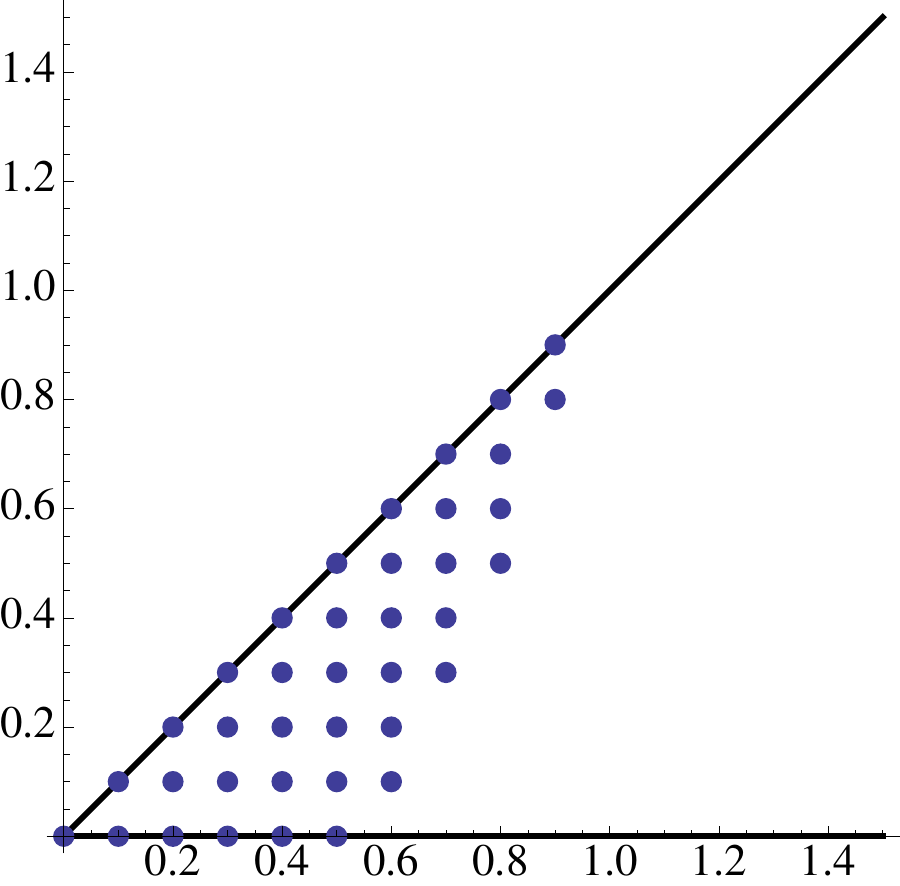}}
  \hskip 0.6cm
    \subfigure[$c_3\beta_3=51$]{\label{gam3}\includegraphics[height=3.5cm]{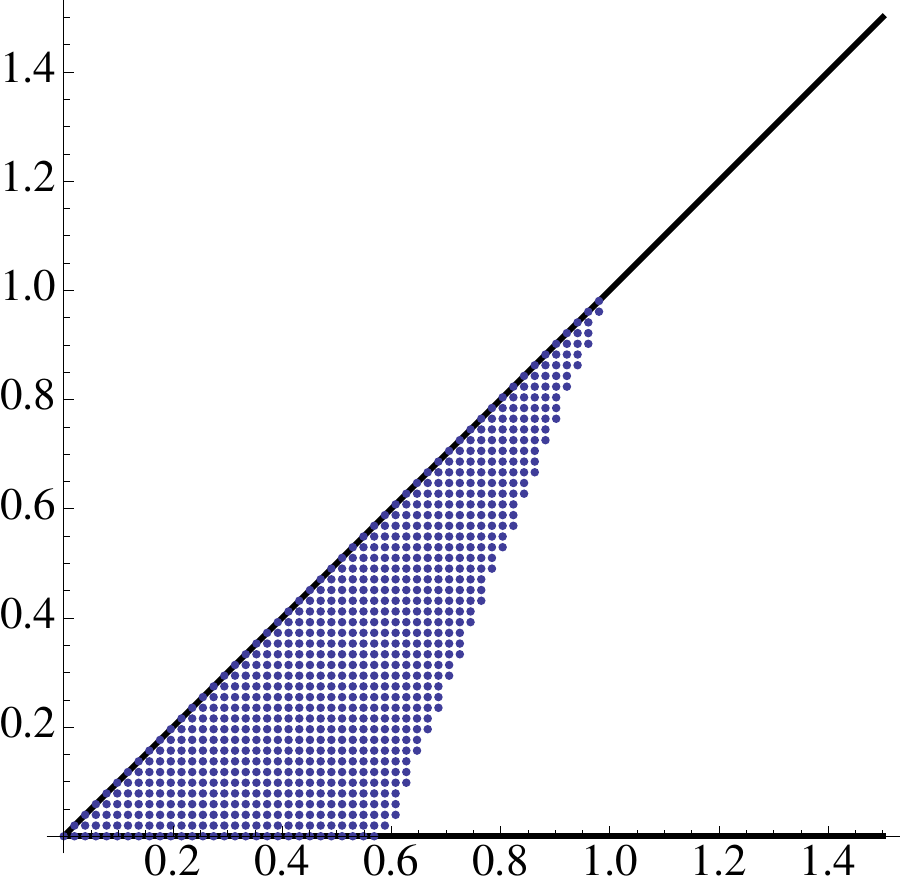}}
   \hskip 0.6cm
   \subfigure[$c_4\beta_4=358$]{\label{gam4}\includegraphics[height=3.5cm]{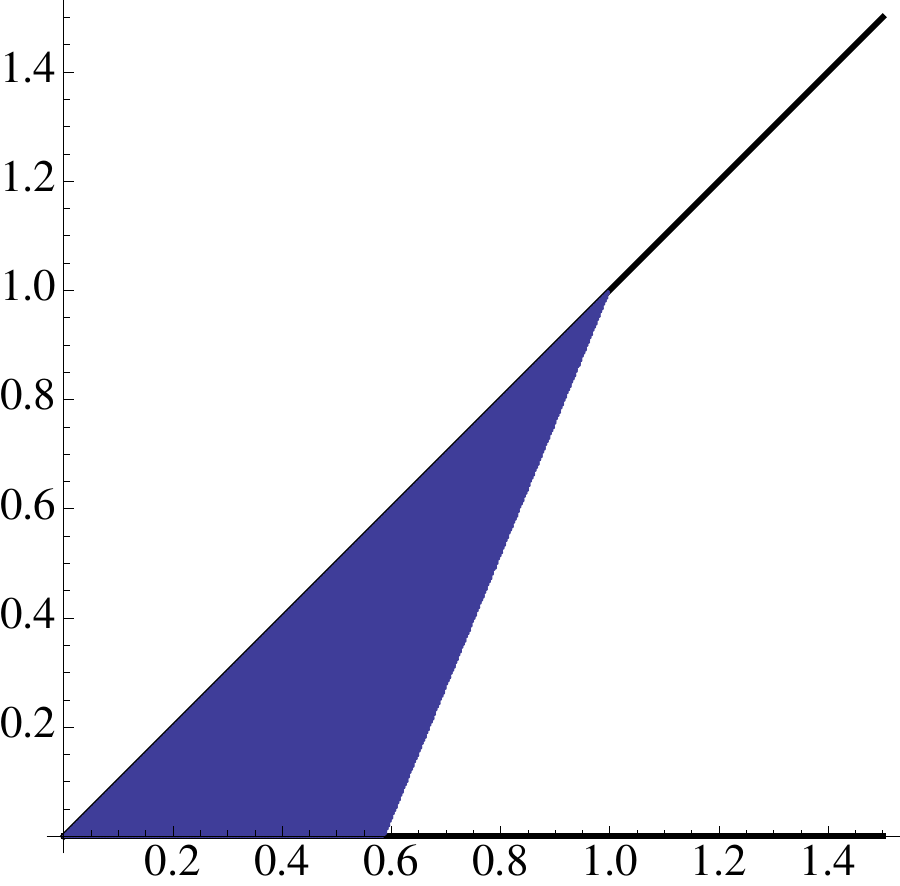}}   
  \end{center}
  \caption{Complement of Primary Sequence}
  \label{figam}
\end{figure}

\end{enumerate}
%\begin{center}
%\renewcommand*{\arraystretch}{1.5}
%\begin{tabular}{|c|c|c|c|c|}
%\hline
%$c_k$ & $\beta_k=v(q_k)$ & $q_k$ & $\#(\cS\setminus\cA_{c_k\beta_k})$ & $\frac{\dim_{\bC}\left(\bC[x,y]/\fa_{c_k\beta_k}\right)}{(c_k\beta_k)^2/2!}$ \\
%\hline
%\thickhline
%$c_1=2$ & $\frac{3}{2}$     &    $x$      &  $5$  & $1.11111$ \\
%\hline
%$c_2=5$ & $\frac{16}{5}$    &    $x^2+y^3$      & $93$   & $0.72656$  \\
%\hline
%$c_3=7$ & $\frac{113}{7}$   & $(x^2+y^3)^5+y^{16}$ & $4054$ & $0.63498$\\
%\hline
%$c_4=11$ & $\frac{1244}{11}$ & $ q_3^7+y^{113}$& $480069$ & $0.62043$ \\
%\hline
%$c_5=13$ & $\frac{3939}{11}$ & $q_4^7+y^{358}$ & $4553318$ & $0.58693$\\
%\hline
%\end{tabular}
%\end{center}

%\newpage

\section{Postscript Note:}\label{sec-ps}

After the second version of the preprint on arXiv was posted, we were informed by K. Fujita that he independently prove the results in Theorem \ref{thmdiv} (see \cite{Fuj16}). Moreover, he obtained a similar result for the uniform K-stability.

\section{Acknowledgment}
The author is partially supported by NSF DMS-1405936.
I would like to thank the math department at Purdue University for providing a nice environment for my research. I would like to thank especially Laszlo Lempert, Kenji Matsuki, Bernd Ulrich and Sai-Kee Yeung for their interest and encouragement. I am grateful to Kento Fujita for pointing out a mistake in an earlier draft which helps me to get a correct statement of the result. I also would like to thank Chenyang Xu for helpful comments, Yuchen Liu for fruitful discussions, Tommaso de Fernex and Mircea Musta\c{t}\u{a} for their interests in this work, and anonymous referees for very useful comments which help me to improve the presentation.

\noindent
li2285@purdue.edu

\noindent
Department of Mathematics, Purdue University, West Lafayette, IN 47907-2067 USA

\end{document}